\documentclass[11pt]{article}

\usepackage[USenglish]{babel}
\usepackage[T1]{fontenc}
\usepackage[ansinew]{inputenc}
\usepackage{graphicx}
\usepackage{amsmath}

\usepackage{amsfonts}
\usepackage{amssymb}
\usepackage{dsfont}
\usepackage{enumitem} 
\usepackage{tikz}
\usetikzlibrary{calc} 

\usepackage{color}
\usepackage{mathrsfs}
\usepackage{fullpage}

\usepackage{amsthm}
\definecolor{darkred}{rgb}{0.6,0,0.2}
\definecolor{darkblue}{rgb}{0,0.3,0.7}
\usepackage[colorlinks,linkcolor=darkblue,citecolor=darkred]{hyperref}
\usepackage[capitalize,nameinlink,noabbrev]{cleveref}

\crefname{equation}{}{}
\crefname{enumi}{}{}

\DeclareMathOperator{\Tr}{Tr}

\DeclareMathOperator{\Ima}{Im}

\DeclareMathOperator{\spec}{spec}
\DeclareMathOperator{\cc}{cc}

\DeclareMathOperator{\RR}{\mathbb{R}}
\DeclareMathOperator{\CC}{\mathbb{C}}
\DeclareMathOperator{\NN}{\mathbb{N}}
\DeclareMathOperator{\Mat}{\mathrm{M}}

\newcommand{\ii}{\mathrm{i}}
\newcommand{\ee}{\mathrm{e}}
\newcommand{\abs}[1]{\left| #1 \right|}
\newcommand{\absN}[1]{| #1 |}
\newcommand{\absn}[1]{\bigl| #1 \bigr|} 
\newcommand{\event}[1]{\big\{ #1 \big\}}
\newcommand{\absb}[1]{\bigl| #1 \bigr|}

\newcommand{\absbb}[1]{\biggl| #1 \biggr|}
\newcommand{\norm}[1]{\left\| #1 \right\|}
\newcommand{\normbb}[1]{\biggl\| #1 \biggr\|}
\newcommand{\R}{\mathbb{R}}
\newcommand{\C}{\mathbb{C}}
\newcommand{\N}{\mathbb{N}}
\newcommand{\E}{\mathbb{E}}

\newcommand{\ccc}{\mathfrak{c}}
\newcommand{\patx}{\boldsymbol{x}}
\newcommand{\paty}{\boldsymbol{y}}
\newcommand{\patchi}{\boldsymbol{\chi}}

\newcommand{\eps}{\varepsilon}
\DeclareMathOperator{\re}{Re}
\newcommand{\im}{\Ima} 

\renewcommand{\Im}{\Ima}

\theoremstyle{plain}
\newtheorem{thhm}{Theorem}[section]
\newtheorem{thm}[thhm]{Theorem}
\newtheorem{lemma}{Lemma}[section]
\newtheorem{prop}{Proposition}[section]

\newtheorem{defi}{Definition}[section]
\newtheorem{corollary}{Corollary}[section]

\newtheorem{remark}{Remark}[section]
\newtheorem*{lemma*}{Lemma}
\theoremstyle{remark}

\numberwithin{equation}{section}

\newcommand*\barr[1]{\overline{#1}}
\newcommand*\bset[1]{[#1]}

\newcommand*\ubarr[1]{{\widehat{#1}}}

\newcommand{\nc}{\normalcolor}

\makeatletter
\def\blfootnote{\xdef\@thefnmark{}\@footnotetext}
\makeatother

\begin{document}
	\blfootnote{Date: \today}
	\blfootnote{MSC2020 Subject Classifications: 46L54, 60B20, 15B52.}

	\title{\textbf{Critical Erd{\H o}s-R\'enyi digraph: \\ all eigenvectors away from zero are delocalized}}
	\author{Johannes Alt\thanks{Institute for Applied Mathematics, University of Bonn. 
	Email: \href{mailto:johannes.alt@iam.uni-bonn.de}{johannes.alt@iam.uni-bonn.de}} 
		\hspace*{0.5cm} \and \hspace*{0.5cm} Sarah Timhadjelt\thanks{Institut Fourier, Universit\'e Grenoble Alpes. 
		Email: \href{mailto:sarah.timhadjelt@univ-grenoble-alpes.fr}{sarah.timhadjelt@univ-grenoble-alpes.fr}}
	} 
	\date{}
	
	\maketitle
	
	\vspace*{-0.3cm}

	\begin{abstract}
	We consider the adjacency matrix of the directed Erd{\H o}s-R\'enyi graph. As long as the expected degree is larger than the logarithm of the number of vertices, the graph is connected, we show that all eigenvectors are completely delocalized. 
	Below this critical scale, we prove eigenvector delocalization if 
	the corresponding eigenvalue is away from zero. 
	This contrasts the \emph{undirected} or Hermitian setting, where large eigenvalues have localized eigenvectors \cite{alt2021delocalization}.
	Our results also hold for sparse random matrices with independent entries, which can be viewed as weighted Erd{\H o}s-R\'enyi digraphs. 
	\end{abstract}

	\tableofcontents 
	
	\addtocontents{toc}{\protect\setcounter{tocdepth}{1}} 
	
	\section{Introduction}

Adjacency matrices of random graphs provide a natural framework for probing the persistence of 
spectral 
phenomena characteristic of random matrices. 
In particular, a number of spectral transitions have been rigorously established for the adjacency matrix of the (undirected) Erd{\H o}s-R\'enyi graph. This is a random graph model on $N$ vertices, where each edge of the complete graph is kept with probability $d/N$ independently of the other edges. Here, $d \in (0,N)$ is an $N$-dependent parameter. 
The spectral transitions mentioned above are steered by the expected degree $d$, or more precisely, how it scales with $N$.  

\blfootnote{We refer to \cref{sec:notation} for the definitions of the relations $\asymp$ and $\ll$ as 
well as to \cref{eq:def_vector_norms} for the definitions of the $\ell^\infty$-norm and the $\ell^2$-norm.} 

We focus on the eigenvector structure and denote by $H$ the normalized adjacency matrix, which is divided by $\sqrt{d}$. 
If $d \asymp N$ then $H$ is \emph{dense}, i.e.\ it has order $N^2$ many nonzero entries, and a Wigner matrix\footnote{They are Hermitian $N\times N$ random matrices with i.i.d.\ entries, up to the symmetry constraint, whose moments are all finite uniformly in $N$.}. 
Every eigenvector of a Wigner matrix is \emph{completely delocalized}, i.e.\ its mass is approximately uniformly distributed among its components \cite{MR3916110,MR2871147,MR3068390,MR2481753}. Formally, 
$\mathbf{u} \in \C^N$ is called completely delocalized if $\norm{\mathbf{u}}^2_\infty \leq N^{-1 + \delta} \norm{\mathbf{u}}^2_2$ for any constant $\delta >0$. Since the $\ell^2$-norm and 
the $\ell^\infty$-norm of any $\mathbf{u}\in \C^N$ satisfy 
\begin{equation} \label{eq:inequality_vector_norms} 
 N^{-1} \norm{\mathbf{u}}_2^2 \leq \norm{\mathbf{u}}_\infty^2 \leq \norm{\mathbf{u}}_2^2, 
\end{equation} 
completely delocalized vectors have approximately the smallest $\ell^\infty$-norm. 
Delocalized eigenvectors arise from the interplay of a large number of matrix entries and 
therefore contain a large amount of ``randomness''. Therefore, they provide  
strong evidence that the underlying microscopic eigenvalue process follows the universal random-matrix statistics.

If $d \ll N$ then $H$ becomes \emph{sparse}, that is, most of its entries vanish, 
which leads to a different structure for some eigenvectors of $H$, when $d$ is sufficiently small. 
We now describe the known results about the eigenvectors of $H$, their structure depends on $d$ and the location of the corresponding eigenvalue. 
\nc 	
	All\footnote{For the purpose of this overview, we ignore the eigenvalue of $H$ near $\sqrt{d}$ which is induced by the nontrivial expectation of the entries of $H$. For $d \gg \sqrt{\log N}$, it is well separated from the other eigenvalues of $H$.} eigenvectors of $H$ are completely delocalized if $d \geq (b_* + \eps) \log N$ 
	with $b_* = (2 \log 2 -1)^{-1}$ and any $\eps>0$ \cite{MR2964770,MR3098073,MR4021251,alt2021delocalization,ADK_deloc_everywhere}.  
	For $d \leq (b_* - \eps) \log N$, eigenvalues outside $[-2-o(1), 2 + o(1)]$ emerge  \cite{ADK_outliers,MR4234995}. For $d \ll \log N$, their emergence has already been shown in \cite{MR3945756} while for $d \gg \log N$ the largest and smallest 
	eigenvalue converge to 2 and -2, respectively, \cite{MR2964770,BBK_spec_radii}. 
	In a region near the spectral edge, each eigenvector with eigenvalue outside $[-2-o(1), 2 + o(1)]$ is localized in the sense that that the mass of the eigenvector is exponentially decaying in the graph distance around a unique vertex \cite{alt2023poisson,ADK_localized}. 
 	Localized vectors have a large $\ell^\infty$-norm, relative to their $\ell^2$-norm, and approximately saturate the upper bound 
in \cref{eq:inequality_vector_norms}. 
	For all other eigenvectors outside $[-2-o(1), 2 +o(1)]$, delocalization was excluded in \cite{alt2021delocalization}. 
 In contrast, eigenvector delocalization persists in $[-2 + o(1), -o(1)]\cup [o(1), 2 - o(1)]$
 as long as $d \gg \sqrt{\log N}$ \cite{alt2021delocalization}. 
 We refer to \cite[Figure~1.1]{ADK_localized} for an illustration of the different spectral phases of $H$. 

Here, we study the directed Erd{\H o}s-R\'enyi random graph $\mathbb{G}=\mathbb{G}(N,d/N)$, or Erd{\H o}s-R\'enyi digraph for short, which arises by allowing for directed edges in the above definition of the Erd{\H o}s-R\'enyi graph.  
Since all entries of the adjacency matrix $M$ of $\mathbb{G}$ are independent,  
 $M$ is non-Hermitian and even non-normal random matrix. Such random matrices are generally less well understood than their Hermitian counterparts. 
 The eigenvalues of $M$ are complex and follow the famous circular law, the uniform measure on the unit disc, when $N$ and $d$ tend to infinity. See \cite{girko1985circular,bai1997,MR2663633,MR2575411,MR2722794,MR2735731} in the dense and  \cite{MR2977992,BasakRudelson2019,RudelsonTikhomirov2019} in the sparse case.  
 Here and in the following, we divide $M$ by $\sqrt{d}$ as in the undirected setting. 
The eigenvalues of $M$ actually concentrate close to the unit disc in the complex plane apart from one eigenvalue of multiplicity one near $\sqrt{d}$ \cite{MR4206685,MR4649432}; we also refer to \cref{prop:spectrum_M} below and the references mentioned thereafter. 

	We investigate whether the eigenvectors of $M$ are delocalized or whether localized eigenvectors can occur.
	For random matrices with independent, centred entries of uniformly finite moments, 
	eigenvector delocalization was proved in \cite{MR3405592} and in 
	\cite{AEKinhomogeneous}. Up to the nonzero expectation of its entries, $M$ is of this form when $d\asymp N$. 
	As long as $d \geq N^\delta$ for any constant $\delta >0$, all eigenvectors of $M$ 
	are completely delocalized \cite{HeDigraph}. 
	Our main result is the delocalization of all eigenvectors of $M$ down to the critical scale $
	d \asymp \log N$, when $\mathbb{G}$ looses its connectivity and isolated vertices 
	as well as leaves emerge, and the subcritical regime $d \ll \log N$. 
	Owing to the emergence of leaves and other ``small'' graphs attached to the giant 
	component of $\mathbb{G}$, delocalization cannot hold if the eigenvalue is close to zero and $d < \log N$; see \cref{rem:leaves} below. 
	Our results hold, in particular, for eigenvectors with eigenvalues near the edge $\abs{w} \approx 1$ of the circular law. However, they do not cover the eigenvector with eigenvalue near $\sqrt{d}$. We prove eigenvector delocalization for general sparse matrices $M$ with independent entries of identical variance, which can be interpreted as adjacency matrices of Erd{\H o}s-R\'enyi digraphs
	with independent edge weights. 
	
		While the Erd{\H o}s-R\'enyi graph commonly serves as a model for undirected real-life networks, 
	the Erd{\H o}s-R\'enyi digraph or sparse matrices $M$ with independent entries are natural models when directionality is crucial, e.g.\ in 
	food webs or neural networks. 
	As a specific example, such $M$ can be viewed as the weight matrix of a (sparse) neural network, encoding the presence and strengths of synaptic connections between neurons.

	Our results show that, for the Erd{\H o}s-R\'enyi digraph or sparse random matrices with independent entries,  
	all eigenvectors are delocalized down to the subcritical scale -- away from zero when $d < \log N$. 
This contrasts with the delocalization-localization transition in the Hermitian setting on the critical scale $d \asymp \log N$ 
and provides a new instance of the phenomenon that spectral characteristics of random matrices are more robust in the non-Hermitian setting.

	Despite the absence of localized eigenvectors in the non-Hermitian setting, the analysis of eigenvector delocalization is more involved than in the Hermitian case. 
	Since $X$ has independent entries, $X$ is non-normal and, therefore, it is more challenging to extract information about eigenvectors from auxiliary objects like the resolvent. 
	This is commonly circumvented via Girko's Hermitization trick \cite{girko1985circular}, which relates (left and right) eigenvectors of $X$ to kernel vectors of 
	\[ H_w = \begin{pmatrix} 0 & (M-w) \\ (M-w)^* \end{pmatrix} \]
	when $w \in \C$ is the corresponding eigenvalue of $X$. 
	The delocalization of such kernel vectors can be deduced by controlling the diagonal entries of the resolvent $G=(H_w - \ii \eta)^{-1}$ 
on the spectral scale $\eta = N^{-1 + \delta}$ just above the typical eigenvalue spacing of $H_w$. 
To our knowledge, this strategy was first used in \cite{AEKinhomogeneous} to obtain delocalization for non-Hermitian matrices.

	To control the diagonal resolvent entries, we use a $2\times 2$-Schur complement formula which 
	respects the block structure of $H_w$ to derive from self-consistent equations, which are approximately satisfied by the diagonal resolvent entries $G_{xx}$.  
	The exact solution $v_x$ of the self-consistent equation for $G_{xx}$ depends on the in- and outdegree of the corresponding vertex $x$. A similar ansatz was used in the undirected, critical case in  \cite{alt2021delocalization}. 
	Here, the extra difficulty is that $v_x$ can diverge, even in the more stable 
bulk spectrum, i.e. away from the spectral edge and the origin, when the indegree is small 
and the outdegree is large or vice versa. The corresponding approximants in the undirected case are bounded in the bulk. 
The function $v_x$ induce a family of measures on the complex plane, which contain as special cases the 
 circular law and the directed Kesten-McKay law, i.e.\ the Brown measure of the sum of $k$ free Haar unitaries for certain in- and outdegrees; see \cref{rem:v_beta_measures} below.

	\section{Main results} \label{sec:main_results}
	
	We study the eigenvector behaviour of sparse non-Hermitian random matrices, which generalize 
	the adjacency matrix of the directed Erd\H{o}s-R{\'e}nyi graph\footnote{The precise definition of the directed Erd\H{o}s-R{\'e}nyi graph is given before \cref{cor:delocalization_unconditional} below.} with $N$ vertices and expected degree $d$.  More precisely, we consider matrices of the form 
	\begin{equation}\label{eq:def_M_X_plus_f}
		M= X + f (\mathbf{e}\mathbf{e}^{\ast} - 1/N), 
	\end{equation}
	where $N \in \N$, $ f\in \CC $ and $\mathbf{e} = \frac{1}{\sqrt{N}}(1,\cdots, 1)^{\ast}$. 
	Here and throughout the paper, scalars appearing in matrix identities such as $1/N$ in \cref{eq:def_M_X_plus_f} are interpreted as the corresponding multiple of the identity matrix. 
	Our assumptions on the random matrix $X$ are summarised in the next definition. 
	We always write $[N] = \{ 1, \ldots, N\}\subset \N $ for the first $N$ positive integer. 
	\begin{defi}\label{defiX}
		We call $X = (X_{xy})_{x, y \in [N]} \in \mathrm{M}_N(\CC)$ a \emph{sparse (centred) non-Hermitian random matrix} 
		with sparsity parameter $d \in (0,N/2)$ if the following holds. 
		\begin{enumerate}[label=(\alph*)]
			\item $(X_{xy})_{x, y \in [N]}$ is a family of independent, centred random variables, in particular, $\mathbb{E}[X_{xy}]= 0$ for all $x$, $y \in [N]$,
			\item  \label{item:variance_off_diag} $\E[\abs{X_{xy}}^2] = N^{-1}$ for all $x \neq y \in [N]$,
			\item  \label{item:variance_diag} 
			 there is a constant $K>0$ such that  $\E[\abs{X_{xx}}^2] \leq K N^{-1}$ for all $x \in [N]$, 
			\item \label{item:def_X_decay} 
			there is a constant $K>0$ such that $\abs{X_{xy}} \leq Kd^{-1/2}$ almost surely for all $x$, $y \in [N]$.
		\end{enumerate}
	\end{defi}
	
	 We note that the condition \cref{defiX} \cref{item:def_X_decay} is equivalent to $\E \abs{X_{xy}}^k \leq \frac{C}{Nd^{(k-2)/2}}$ for all $k\in \NN$ and $x$, $y \in [N]$ due to \cref{item:variance_off_diag} and \cref{item:variance_diag}.

	Throughout the paper, we use the following notion for events, whose complement occurs with probability smaller than any inverse polynomial power in $N$. 
	
	\begin{defi}[Very high probability event]\label{def:very_high_probab} 
		Let $\Omega$ and $\Xi$ be events that depends on $N \in \NN$ and $\nu >0$. We say that $\Omega$ holds \emph{with very high probability} if for each $\nu>0$, there is $\mathcal C \equiv \mathcal C_\nu$ such that 
		\[ 
		\mathbb{P}(\Omega_{N,\nu}) \geq 1 - \mathcal C_\nu N^{-\nu} 
		\] 
		for all $N \in \NN$.  We say that $\Omega$ holds \textup{with very high probability on the event} $\Xi$ if the event $\Omega \cup \Xi^{c}$ holds with very high probability.
	\end{defi} 
	Throughout the paper, we use the convention that calligraphic letter such as $\mathcal C$ or $\mathcal D$ can tacitly depend on the exponent $\nu$ from \cref{def:very_high_probab} in statements which hold with very high probability. 
	
	\subsection{Eigenvector delocalization for $M$} 
	
	With a matrix $X \in \C^{N\times N}$, we associated the \emph{normalized degrees} 
	\begin{equation}\label{eq:def_beta_i_x} 
		\beta^1_x:= \sum_{y \in [N]\setminus\{x\}} |X_{xy}|^2 \qquad \text{and} \qquad 	\beta^2_x:= \sum_{y\in [N] \setminus \{ x\}} |X_{yx}|^2
	\end{equation}
	for $x \in [N]=\{1, \ldots, N\}$. 
	They should be thought of as (normalized) in- and out-degree of $x$ when $X$ is the (weighted) adjacency matrix of a graph with vertex set $[N]$. 
	
	For any vector $\mathbf u=(u_x)_{x \in [N]} \in \C^N$, we define its $\ell^2$-norm $\norm{\mathbf u}_2$ and its $\ell^\infty$-norm $\norm{\mathbf u}_\infty$ through 
	\begin{equation} \label{eq:def_vector_norms} 
		\norm{\mathbf u}_2 := \biggl(\sum_{x \in [N]} \abs{u_x}^2\biggr)^{1/2}, \qquad \qquad \norm{\mathbf u}_\infty := \max_{x \in [N]} \abs{u_x}. 
	\end{equation} 
	Clearly, $N^{-1} \norm{\mathbf u}_2^2 \leq \norm{\mathbf u}_\infty^2 \leq \norm{\mathbf u}_2^2$ for all $\mathbf u \in \C^N$. 
	The next theorem states that eigenvectors of $M$ almost saturate the first bound.

	 A vector $\mathbf u \neq 0$ is called \emph{completely delocalized} if for each $\delta>0$, $\norm{\mathbf u}^2_\infty \leq N^{-1 + \delta}\norm{\mathbf u}_2^2$ with very high probability. 
	 
	Before we state the next theorem about the eigenvectors of $M$, we briefly describe its spectrum. If $1 \le d \le N^{\delta}$ then with probability at least $1 - o(1)$, all eigenvalues of $M$ are contained in the disc $\mathrm{D}_{1 + Cd^{-1/2}}$ apart from one eigenvalue $\lambda$ near $f$. We refer to \cref{prop:spectrum_M} and the explanations thereafter for details and references. 
	 The eigenvalue $\lambda$ has multiplicity one and we denote by $\mathbf{u}_0$ an 
	 associate eigenvector, i.e.\ 
	 \begin{equation} \label{eq:eigenvector_u_0} 
	 M \mathbf{u}_0 = \lambda \mathbf{u}_0 \qquad \qquad \lambda = f + o(1). 
	 \end{equation} 
	 
	 With the definitions  
	 	\begin{equation} \label{eq:def_d_hash} 
	d_\mathrm{l} :=  (\log N)^{1/2}, 
	\qquad d_\mathrm{b} :=  (\log N)^{12/13}, \qquad 
	d_\mathrm{e} :=  (\log N)^{24/25}, 
	\end{equation} 
	 we consider the three regimes for the expected degree $d$ given by $\# \in \{ \mathrm{l}, \mathrm{b}, \mathrm{e}\}$ and  
	\begin{equation}\label{eq:regimed}
	\mathcal D \,  d_\# \leq  d \le N^{\delta/4}  
	\end{equation}
for some constant $\mathcal D \geq 1$.  
In the following theorem, $d_\mathrm{l}$, $d_\mathrm{b}$ and $d_\mathrm{e}$ define the lower bounds on $d$ under the lower bound assumption on $\beta_x^i$, in the bulk spectrum and near the spectral edge, respectively. 

		\begin{thm}[Eigenvector delocalization]\label{thm:condideloc}
	 Let $\delta \in (0,1)$ be fixed. Let $M$ be as in \cref{eq:def_M_X_plus_f} with $X$ as in \cref{defiX} and $\abs{f} \leq N^{\delta/6}$. Then the following holds. 
		\begin{enumerate}[label=(\roman*)]
			\item \label{item:thm_cond_deloc_small_eigenvec}
			 If $d$ satisfies \cref{eq:regimed} with $\# = \mathrm{l}$ then  all eigenvectors of $M$ with eigenvalues in $\{ w \in \CC \colon \abs{w} \le 1-\delta\}$ are completely delocalized with very high probability on the event $\{ \beta_{x}^i \ge \delta \text{ for all } x,i \}$. 
			\item 
			 If $d$ satisfies \cref{eq:regimed} with $\# = \mathrm{b}$ then all eigenvectors of $M$ with eigenvalues in $\{ w \in \CC \colon \delta \leq \abs{w} \leq 1- \delta\}$ are completely delocalized with very high probability. 
		\item\label{item:thm_cond_deloc_large_eigenvec} If $d$ satisfies \cref{eq:regimed} with $\# = \mathrm{e}$ and $|f| \ge 2 $ with very high probability, all normalized eigenvectors $\mathbf{u}$  with eigenvalue in $\{w \in \CC \colon \delta \leq \abs{w} \leq 2\}$ are completely delocalized.  
		\end{enumerate}
		Throughout this theorem, eigenvectors refers to the left and right eigenvectors of $M$. 
	\end{thm}

	The previous theorem is proved in \cref{sec:proof_delocalization} below. 
	We highlight that in \cref{thm:condideloc} \cref{item:thm_cond_deloc_large_eigenvec} all eigenvalues apart from one have modulus at most $1 + Cd^{-1/2}$ 
	with probability at least $1-o(1)$, see \cref{eq:eigenvector_u_0} and \cref{prop:spectrum_M} below. 
	
	In the next result, we verify the conditions on $\beta_x^i$ from \cref{thm:condideloc} for some choices of $X$ and in some regimes of $d$.  
	To that end, for $d \in (0,N/2)$, we denote by $\mathbb{G} \equiv \mathbb{G}(N,d/N)$ the directed Erd{\H o}s-R\'enyi graph with vertex set $[N]$ and edge probability $d/N$. That is, each edge of the complete directed graph on $[N]$ is kept with probability $d/N$ independently from all other edges. Formulated differently, the adjacency matrix $\mathrm{Adj}(\mathbb G)$ of $\mathbb{G}$ has i.i.d.\ Bernoulli-distributed entries with success probability $d/N$.  For all $z \in \C$ and $r > 0$ we denote closed discs in the complex plane by 
	\begin{equation*}
		\mathrm{D}_{r}:= \{ w \in \C \colon \abs{w} \le r\}, ~~ \text{and} ~~ \mathrm{D}_r(z) :=\{ w \in \C \colon  \abs{w-z} \le r\}.
	\end{equation*}

	\begin{corollary}[Delocalization in Erd{\H o}s-R\'enyi digraph] \label{cor:delocalization_unconditional} 
	Let $\delta>0$ be a small enough constant and $(1+ \delta) \log N \le d \le N^{\delta/4}$. 
 Then, with probability at least $1-o(1)$, all eigenvectors of $\mathbb{G}(N,d/N)$ are completely delocalized except for the one associated with the eigenvalue of largest modulus.
	\end{corollary} 

	The proof of \cref{cor:delocalization_unconditional} is presented in \cref{sec:proof_cor_deloc_ER_graph} below.  
	If $d \geq N^\delta$ for some constant $\delta \in (0,1)$ then delocalization of all eigenvectors of $M$ away from $f$ has been shown in \cite[Theorem~1.2 (ii)]{HeDigraph}. 
	
	\begin{remark}[Optimality of $d \geq \log N$] \label{rem:leaves}
	The lower bound $d \geq (1 + \delta) \log N$ is optimal for the delocalization of all (left and right) eigenvectors of $\mathrm{Adj}(\mathbb{G})$ with $\mathbb{G} = \mathbb{G}(N,d/N)$. 
	For $d \leq (1 - \delta) \log N$, the giant component of $\mathbb{G}$ has leaves with high probability. They induce standard basis vectors as left or right eigenvector, depending on the direction of the edge connecting the leaf and the giant component, with eigenvalue zero.  Therefore, eigenvector delocalization breaks drastically, even for the giant component, 
	for $d < \log N$. This transition occurs at the same threshold as in the undirected setting \cite{alt2021delocalization,ADK_deloc_everywhere}.
	\end{remark} 
	
	\subsection{Local law for Hermitization}

	The results in the previous subsection will follow from a local law for the Hermitization $H(w)$ of the matrix 
	$M$ from \cref{eq:def_M_X_plus_f}, which we introduce in \cref{hermitizationofM} below.
	The local law asserts that the resolvent of $H(w)$ is well-approximated by a matrix 
	which only depends the normalized degrees $\beta_x^i$ from \cref{eq:def_beta_i_x} and the spectral parameters on all spectral scales slightly above the typical eigenvalue spacing. 
	
	We first introduce some notation for the statement of the local law.  
	The \emph{Hermitization} of $M$ is the Hermitian matrix 
	\begin{equation}\label{hermitizationofM}
		H=H(w) :=\begin{pmatrix} 
			0                       &       M-w \\
			
			M^{\ast}- \overline{w} & 0
		\end{pmatrix}  \in \Mat_{2N}(\CC) 
	\end{equation}
	with the complex parameter $w \in\C$. 
	For the following definitions, let $w \in \C$ and $\eta >0$.
	We denote the corresponding Green function or resolvent of $H(w)$ at $\ii \eta$ by 
	\begin{equation}\label{Greenfunction}
		G=G(w,\ii\eta) = (H(w)- \ii \eta)^{-1}, 
	\end{equation}
	where the inverse exists as $H(w)$ is Hermitian. 
	For the normalized trace of $G(w,\ii \eta)$, we write 
	\begin{equation} \label{eq:def_g} 
		g(w,\ii \eta) = \frac{1}{2N} \Tr G(w,\ii \eta),  
	\end{equation} 
	which is also the Stieltjes transform at $\ii \eta$ of the empirical spectral measure $(2N)^{-1} \sum_{\lambda \in \mathrm{spec}H(w)}\delta_\lambda$ of $H(w)$.  
	We denote by $v(w,\eta)$ the unique positive solution of 
	\begin{equation} \label{eq:v} 
		\frac{1}{v(w,\eta)}= \eta + v(w,\eta) + \frac{\abs{w}^2}{\eta + v(w,\eta)}. 
	\end{equation} 
	For $\beta= (\beta^1,\beta^2) \in[0,\infty)^2$ and $w$, $\eta$ as above, we define  
	\begin{align}  
		v_\beta(w,\eta) & := \frac{\eta + \beta^2 v(w,\eta)}{\abs{w}^2 + (\eta + \beta^1 v(w,\eta))(\eta + \beta^2 v(w,\eta))},  
		\label{eq:def_v_beta}\\ 
		u_\beta(w,\eta) & := \frac{1}{\abs{w}^2 + (\eta + \beta^1 v(w,\eta))(\eta + \beta^2 v(w,\eta))}. 
		\label{eq:def_u_beta} 
	\end{align} 
	We recall the definitions of $\beta_x^1$ and $\beta_x^2$ from \cref{eq:def_beta_i_x} for $x \in [N]$ and define the vectors 
	\[ \beta_x:= (\beta^1_x,\beta^2_x), \quad \text{and} \quad   \ubarr{\beta}_x := (\beta^2_x,\beta^1_x). \]
To unify the notation in the next theorem, we introduce the indicator functions 
\begin{equation} \label{eq:def_chi} 
	\chi_{\mathrm{l}}:= \mathds{1}\left( \beta_{x}^i \ge \delta~\text{for all}~ x\in [N],\, i =1,2\right), \qquad \chi_\mathrm{b} := \chi_\mathrm{e} := 1, 
	\end{equation} 
	as well as the spectral domains 
	\begin{equation} \label{eq:def_spectral} 
	\mathrm{S}_{\mathrm{l}} := \mathrm{D}_{1-\delta}, \qquad \qquad 
	\mathrm{S}_{\mathrm{b}} := 	\mathrm{D}_{1-\delta}\setminus \mathrm{D}_\delta, \qquad \qquad 
	\mathrm{S}_\mathrm{e} := \mathrm{D}_2\setminus \mathrm{D}_\delta 
	\end{equation} 
	for some constant $\delta \in (0,1)$. 
	
In the following theorem, we use the notations $\ll$ and $o(1)$, which are defined in \cref{sec:notation} below. 	

		\begin{thm}[Conditional Local Law for $H(w)$]\label{thm:locallaw_H}
		Let $ 0 < \delta \le 1/2$ be fixed and $X$ be a sparse matrix as in  \cref{defiX}. 
		Let $\# \in \{\mathrm{l}, \mathrm{b}, \mathrm{e}\}$.    
		We define $M$ as in \cref{eq:def_M_X_plus_f} for some (deterministic) $f$ such that $0 \le \abs{f} \le N^{\delta/6}$, and define $G$ and $g$ as in \cref{Greenfunction} and \cref{eq:def_g} respectively. 
		If $d$ satisfies $d_\# \ll d \leq N^{\delta/4}$ then, with very high probability, for all $w \in \mathrm{S}_\#$ and all $ N^{-1 + \delta} \le \eta \le 1$ we have
		\begin{equation}
			\chi_\# \max_{x \in [N]} \Big( \absb{G_{xx}(w,\ii \eta) - \ii v_{\beta_x}(w,\eta)} +
			\absb{G_{\ubarr{x}\ubarr{x}}(w,\ii \eta) - \ii v_{\ubarr{\beta}_x}(w,\eta)}
			\Big) =o(1) ,
		\end{equation}
		\begin{equation}
			\chi_\# \max_{x \in [N]} \Big(
			\absb{G_{x\ubarr{x}}(w,\ii\eta) + w u_{\beta_x}(w,\eta)} 
			+ \absb{G_{\ubarr{x}x}(w,\ii\eta) + \overline{w} u_{{\beta}_x}(w,\eta)} \Big)  =o(1),
		\end{equation}
		\begin{equation}
			\chi_\#  \max_{\stackrel{x,y \in [2N]}{x \neq y \neq \ubarr{x}}}\absb{G_{xy}(w, \ii \eta)}  =o(1),
		\end{equation}
		\begin{equation}
			\chi_\#  \absb{g(w,\ii \eta) - \ii v(w,\eta)}  =o(1).
		\end{equation}
	\end{thm}
	
		For $d \leq (\log N)^{3/2}$, the previous theorem follows directly from \cref{thm:locallaw_edge} below, which has precise error estimates and holds as long as  $d \geq \mathcal D \, d_\#$. Its proof takes up most of this work and is presented at the end of \cref{section:proof_local_law_H} below. 
	For $(\log N)^{3/2}\leq d \leq N^\delta$, we automatically have $\chi_\# \equiv 1$ with very high probability. Therefore, the statement and proof of the local law for $H(w)$ are simpler and presented in \cref{thm:locallaw3} below. After the statement of \cref{lem:sumXabGa_eq_sumGa}, we deduce \cref{thm:locallaw_H} in this regime.

		\begin{remark}[Special cases of $v_{\beta}$] \label{rem:v_beta_measures} 
	With $v$ from \cref{eq:v}, the integral 
	\begin{equation} \label{eq:relation_v_circular_law} 
	-\frac{1}{2\pi} \int_0^\infty \Delta_w v(w,\eta) \, \mathrm{d} \eta
	\end{equation} 
     gives the density of the circular law, i.e.\ the uniform probability measure on the unit disc in the complex plane. This obviously corresponds to $\beta^1=\beta^2 = 1$. Analogously, for $\beta^1 = \beta^2 = \frac{k}{k-1}$
	and a fixed $k\in \N$, 
	\cref{eq:relation_v_circular_law} with $v$ replaced by $v_\beta$ yields the Brown measure of the sum of $k$ free Haar unitaries, 
	also called oriented Kesten-McKay law, which coincides with the Brown measure of the directed 
	$k$-regular tree. For general $\beta^1$ and $\beta^2$, the analogous construction yields a much richer family of probability measures on $\C$. 
	\end{remark}

	\textbf{Acknowledgements.}
	The authors would like to thank Charles Bordenave and Antti Knowles for insightful discussions about \cite{BBK_spec_radii}.
	The authors gratefully acknowledge funding from the Deutsche Forschungsgemeinschaft (DFG, German Research Foundation) under Germany's Excellence Strategy - GZ 2047/1, project-id 390685813. The second author also gratefully acknowledges funding from the ANR DYNACQUS, Grenoble ANR-24-CE40-5714-02.
	
	\subsection{Notations}
	\label{sec:notation}

Let $a$ and $b$ be positive quantities which may depend on $N$. 
We write $a \lesssim b$ if there is a constant $C>0$ independent of $N$ such that $a \leq C b$ for all $N \in \N$. If $a$ and $b$ depend on other parameters, we will specify if $C$ is uniform for certain ranges of these parameters. 
	We write $a \gtrsim b$ if $b \lesssim a$ and $a \asymp b$ if $a\lesssim b$ and $b \lesssim a$. We write $a = o(1)$ if $a \to 0$ as $N \to \infty$. We write $a \ll b$ if $a/b =o(1)$.
	
	For each $n \in \N$, we write $[n] := \{1, \ldots, n\}$ for the set of the first $n$ positive integers.

	We denote the complex upper half-plane by $\CC_+ := \{ z \in \CC\colon \Ima(z)>0\}$ and for a positive integer $N$, 
the algebra of $N$-dimensional matrices by $\mathrm{M}_N (\CC)$. For matrices $M$, 
	$\norm{M}$ denotes the operator norm induced by the $\ell^2$-norm or Euclidean norm on the underlying spaces, see \cref{eq:def_vector_norms}. 
	For an $N\times N$-matrix $M$, we denote by $\|M\|_{\infty}:= \max_{i,j \in [N]} |M_{ij}|$ the entrywise maximum norm of $M$. 
	In identities involving matrices, scalars are interpreted as multiples of the identity matrix. 

	Throughout the paper, we follow the convention that calligraphic letter like $\mathcal C$ or $\mathcal D$ can tacitly depend on the exponent $\nu$ from \cref{def:very_high_probab} in statements which hold with very high probability.

\section{Preliminaries} 

	\subsection{Approximate self-consistent equations and error terms}\label{subsection:Approx_sc_eq} 
	
	In this section, we derive the approximate self-consistent equations for the entries of the resolvent $G$ of the Hermitization $H(w)$, see \cref{Greenfunction} and \cref{hermitizationofM}, respectively. Moreover, we estimate a first set of error terms in these equations. 
	
	To formulate these results, we now introduce some notation used throughout the remainder of this work. 
	Let $Q=(Q_{xy})_{x,y \in [N]} \in \mathrm{M}_N(\CC)$ be a matrix. For each $T \subset [N]$, we define the matrix 
	\begin{equation}\label{eq:def_M_(_T_)} 
		Q^{(T)}:= (Q_{xy}\mathds{1}\left(\{x,y\} \cap T = \emptyset\right))_{x,y\in [N]} \in \mathrm{M}_{N}(\CC). 
	\end{equation}
	It corresponds to the minor of $Q$, where all rows and columns of $Q$ with index in $T$ are removed. 
	If $x \in [N]$ and $T= \{x\}$, we set $Q^{(x)} := Q^{(\{x\})}$. Moreover,  we set $Q^{(Tx)}:=Q^{(T\cup \{x\})}$ for each $T \subset [N]$ and $x \in [N] \setminus T$.

	Let $T \subset [N]$. For the Hermitization $H(w)$ from \cref{hermitizationofM} and $w \in \C$, 
	we define the matrix 
	\begin{equation}\label{hermitizationofM(T)}
		H^{\bset{T}}:= H^{\bset{T}}(w) :=\left(
		\begin{array}{ccc}
			0                       &       M^{(T)}-w \\
			
			(M^{(T)})^{\ast}- \overline{w}  & 0
		\end{array}
		\right).
	\end{equation}
	Since $H(w)^{[T]}$ is Hermitian, $H(w)^{[T]}-z$ is invertible for each $z \in \C_+$ and we denote the 
	corresponding resolvent by 
	\begin{equation}\label{eq:GreenfunctionM(T)}
		G^{\bset{T}}=G^{\bset{T}}(w,z) = (H^{\bset{T}}(w)- z)^{-1}.
	\end{equation}
	Furthermore, we use the following notion for the indices in $[2N]$. For $x \in [2N]$ and $T \subset [N]$, we set
	\begin{equation*}
		\ubarr{x}:= x+N \mod 2N, \qquad \qquad \ubarr{T}:= \{ x + N:~ x \in T\}.
	\end{equation*}
	
	For a subset $T \subset [N]$, we write for a summation over the elements of $[N]$, which are not in $T$, 
	\begin{equation*}
		\sum_{x}^{(T)} \,\, := \sum_{x \in [N]\backslash T}.
	\end{equation*}
In the next lemma, we derive self-consistent equations, which couple the diagonal entries of $G$ and the diagonal entries of its off-diagonal $N \times N$-blocks.  
	
	\begin{lemma}[Approximate self-consistent equations]\label{lem:ASCEq}
		For all $x\in[N]$, $w \in \CC$ and $z \in \CC_+$ we have
		\begin{align*}
			1 &= -z G_{xx} - \bigg(\sum_{y}^{(x)} |X_{xy}|^{2}G^{\bset{x}}_{\ubarr{y}\ubarr{y}}\bigg) G_{xx} - w G_{ \ubarr{x}x}+ Y^1_{x} G_{xx} + Z^{1}_{x}G_{ \ubarr{x}x}\\
			0 &= -z G_{x\ubarr{x}} - \bigg(\sum_{y}^{(x)} |X_{xy}|^{2}G^{\bset{x}}_{\ubarr{y}\ubarr{y}}\bigg) G_{x\ubarr{x}} - w G_{ \ubarr{x}\ubarr{x}} + Y^1_{x} G_{x\ubarr{x}} + Z^{1}_{x}G_{ \ubarr{x}\ubarr{x}}\\
			0 &=  -z G_{\ubarr{x}x} - \bigg(\sum_{y}^{(x)} |X_{yx}|^{2}G^{\bset{x}}_{yy}\bigg) G_{\ubarr{x}x} -\barr{w} G_{xx} + Y^2_{x} G_{\ubarr{x}x} + Z^{2}_{x}G_{ xx}\\
			1 &=  -z G_{\ubarr{x}\ubarr{x}} - \bigg(\sum_{y}^{(x)} |X_{yx}|^{2}G^{\bset{x}}_{yy}\bigg) G_{\ubarr{x}\ubarr{x}} -\barr{w} G_{x\ubarr{x}} + Y^2_{x} G_{\ubarr{x}\ubarr{x}} + Z^{2}_{x}G_{ x\ubarr{x}}
		\end{align*}
		where the error terms are given by
		\begin{align*}
			Y_{x}^1 &= - \sum_{a \neq b}^{(x)} X_{xa} G^{\bset{x}}_{\ubarr{a}\ubarr{b}} \barr{X}_{xb}  - \sum_{a,b}^{(x)} \biggl(\frac{f}{N}\left[X_{xa} G_{\ubarr{a}\ubarr{b}}^{\bset{x}} + G_{\ubarr{a}\ubarr{b}}^{\bset{x}} \barr{X}_{xb} \right] + \frac{f^2}{N^2}G_{\ubarr{a}\ubarr{b}}^{\bset{x}}\biggr) ,\\
			Z_x^1 &=X_{xx} - \sum_{a,b}^{(x)} \left(X_{xa} + \frac{f}{N} \right) G_{\ubarr{a}b}^{\bset{x}} \left(X_{bx} + \frac{f}{N}\right)\\
			Y_{x}^2 &= - \sum_{a\neq b}^{(x)} \barr{X}_{ax} G_{ab}^{\bset{x}}X_{bx} - \sum_{a,b}^{(x)} \biggl(\frac{f}{N}\left[\barr{X}_{ax} G_{ab}^{\bset{x}} + G_{ab}^{\bset{x}} X_{bx} \right] + \frac{f^2}{N^2}G_{ab}^{\bset{x}}\biggr) \\
			Z_x^2 &= \barr{X}_{xx} - \sum_{a,b}^{(x)}\left(\barr{X}_{ax} + \frac{f}{N} \right) G_{a\ubarr{b}}^{\bset{x}} \left(\barr{X}_{xb} + \frac{f}{N}\right).
		\end{align*}
	\end{lemma}
	
	\begin{proof}
		We fix $x \in [N]$. Introducing the matrices 
		\begin{align*}
			A_{x} &= \left( \begin{array}{cc}-z                            & M_{xx} -w          \\
				\overline{M}_{xx}-\overline{w}& -z  \end{array}\right),  \\
			B_{ x} & = \left(\begin{array}{cc}0\cdots\cdots\cdots\cdots\cdots\cdots \cdots \cdots \cdots 0 & M_{x1}~ M_{x2}\cdots M_{x x-1}~ M_{x x+1} \cdots M_{x N} \\ \overline{M}_{1x} ~ \overline{M}_{2x} \cdots \barr{M}_{x-1 x} ~ \barr{M}_{x+1 x} \cdots \barr{M}_{Nx} & 0 \cdots\cdots\cdots\cdots\cdots\cdots \cdots \cdots \cdots 0  \end{array}\right)
		\end{align*}
		we note that  up to relabelling the rows and columns,  
		\begin{equation*} 
			G = \begin{pmatrix} A_{ x} & B_{ x}  \\ (B_{x})^* &  H^{[x]} -z  \end{pmatrix}^{-1}. 
		\end{equation*} 
		Therefore, the Schur's complement formula  
		yields for the upper left $2\times 2$-block 
		\[ \begin{pmatrix} G_{xx} & G_{x\ubarr{x}} \\ G_{\ubarr{x}x} & G_{\ubarr{x}\ubarr{x}} 
		\end{pmatrix} = \big(A_x - B_x G^{[x]} (B_x)^*\big)^{-1}. \] 
		We multiply this from the left by $A_x - B_x G^{[x]} (B_x)^*$, which completes the proof of the lemma.
	\end{proof}
	
	To study typical vertices in the next section, we need to analyse the resolvent after removing a number of rows and columns, which diverges with $N$, from the full Hermitization $H$.
	We show in the next lemma that the resolvent entries remain bounded after such procedure if the resolvent entries of the full matrix $H$ are bounded. 	
	We introduce a possibly $N$-dependent parameter $ \Gamma $ and a corresponding indicator function
	\begin{equation}\label{eq:theta}
		1\le \Gamma \le 1 + \frac{\log N}{d}, ~~\text{and} ~~\theta := \mathds{1}\left(\max_{x,y\in [2N]} \abs{G_{xy}}\le \mathcal C \Gamma\right) .
	\end{equation}
The proofs of the two following lemmas are given in \cref{subsection:proofs_approx_sc}.
	\begin{lemma}[Upper bound on resolvent entries] \label{lem:roughboundsGT}
	 Let $\delta \in (0,1)$ and $\Gamma$ be as in \cref{eq:theta}. 
		Then there exists a constant $\mathfrak{c}= \mathfrak{c}(\nu,\delta) >0$ such that for each $ \frac{\Gamma^2}{64 \mathfrak c} \leq d \leq N^{\delta/2}$, for all deterministic $T\subset [N]$ satisfying $\abs{T} \le \mathfrak{c}d / \Gamma^2$, we have 
		\begin{equation}\label{eq:boundGTGTu}
			\theta \max_{x,y \notin T \cup \ubarr{T}} \absb{G^{\bset{T}}_{xy}} \le 2  \mathcal C  \Gamma
		\end{equation}
		with very high probability uniformly for $w \in \C$ and $z \in \C_+$ satisfying $\Im z \geq N^{-1+\delta}$. 
		Furthermore, under the same assumptions on $T$ and for any $u \in [N] \setminus T$, with very high probability,  we have 
		\begin{equation}\label{eq:GTandGTuclose}
			\theta \max_{x,y \notin T \cup \ubarr{T} \cup \{ u, \ubarr{u}\}} \absb{G_{xy}^{\bset{Tu}} - G_{xy}^{\bset{T}}} \le \mathcal{C} \frac{\Gamma^{3}}{d}.
		\end{equation}
	\end{lemma}
	 
The proof of the previous lemma, crucially rely on suitable resolvent identities, which we derive in \cref{sec:resolvent_identities}, in order to follow a similar proof strategy as in the analogous result \cite[Lemma~4.14]{alt2021delocalization} in the undirected setup. These identities are designed to remove two corresponding 
	rows and columns at once to keep the special structure of the Hermitization.

	We now estimate the error terms in the self-consistent equations from \cref{lem:ASCEq}
	as well as the offdiagonal entries of the resolvent. 

	\begin{lemma}[Error estimates] \label{lem:roughboundsrest}
		Let $\delta \in (0,1)$,  $\Gamma$ be as in \cref{eq:theta} and $\mathfrak c$ as in \cref{lem:roughboundsGT}.  
		If $\frac{\Gamma^2}{32\mathfrak c} \leq d \leq N^{\delta/8}$ then 
		\begin{align}
		\theta \max_{x \in [N],\, i\in \{1,2\}} \left|Y^{i}_{x}\right|  \le \mathcal{C}  \frac{\Gamma}{d}, \qquad 
		\theta \max_{x \in [N],\, i\in \{1,2\}} \left|Z^{i}_{x}\right| & \le \mathcal{C} \biggl( \frac{\Gamma}{d} + \frac{1}{\sqrt{d}} \biggr) \label{eq:boundrest}\\
			\theta \max_{\substack{x,y \in [2N],\\y\notin \{x,\ubarr{x}\}}} \left|G_{xy}\right| &\le  \mathcal{C}  \frac{\Gamma^2}{\sqrt{d}} \label{eq:roughboundGd-1/2}\\
			\theta \max_{\substack{x,y \in [2N],\\ \{x,y\}\cap \{a,\ubarr{a}\} =\emptyset }} \bigl|G_{xy} - G_{xy}^{\bset{a}}\bigr| &\le  \mathcal{C}\frac{\Gamma^3}{d} \label{eq:roughdifferenceGGTd-1}
		\end{align}
		with very high probability
		uniformly for all $w \in \C$, $a \in [N]$ and $z \in \C_+$ satisfying $\Im z \geq N^{-1+\delta}$. 
	\end{lemma}

	We end this section with two auxiliary results which will be essential later. 
	The first one describes a special symmetry of the resolvent $G$ which originates of the block structure of the Hermitization. 
	The second one is a general upper bound on the normalised degrees $\beta_x^1$ and $\beta_x^2$ from \cref{eq:def_beta_i_x}. 
	
	\begin{lemma}[Trace identity of resolvent blocks]\label{lem:traceg1=g2}
		Let $0< d < N/2$, $z \in \C_+$ and $w \in \CC$. 
		Then
		\begin{equation*}
			\sum_{y\in [N]} G_{yy} = \sum_{y\in [N]} G_{\ubarr{y} \ubarr{y}}.
		\end{equation*}
	\end{lemma}
	
	\begin{proof}
		We give the proof for $T= \emptyset$. The general case follows by replacing $H(w)$ by $H(w)^{\bset{T}}$. For $z \in \C_+$, we deduce from the Schur complement formula that 
		\begin{align*}
			G 	= \begin{pmatrix} z((M-w)(M-w)^*-z^2)^{-1} & ((M-w)(M-w)^*-z^2)^{-1}(M-w) \\ 
			(M-w)^*((M-w)(M-w)^*-z^2)^{-1} 
			& z((M-w)^* (M-w)-z^2)^{-1} 
			\end{pmatrix} . 
		\end{align*}
		Since $\left(M - w \right)\left(M - w \right)^\ast$ and $\left(M - w \right)^\ast\left(M - w \right)$ have the same spectrum with the same multiplicities, the matrices in the upper left and lower right $N\times N$ blocks of $G$ have the same spectrum with the same multiplicities.
		Therefore, the traces of these blocks coincide, i.e.\ $\sum_{y \in [N]} G_{yy}  = \sum_{y \in [N]} G_{\ubarr{y}\ubarr{y}}$. 
	\end{proof}
	
	\begin{lemma}[Bound on normalized degrees] \label{lem:roughboundonbetaix}
	Let $1 \leq d \leq N/2$. Then 
		\begin{equation*}
			\beta_{x}^i \le \mathcal C \biggl( 1 + \frac{\log N}{d}\biggr)
		\end{equation*}
		for all $x \in [N]$ and $i \in \{1,2\}$ with very high probability. 
	\end{lemma}
	\begin{proof}
		We fix $x \in [N]$ and bound $\beta_x^1$. The same argument with $X_{xy}$ replaced by $X_{yx}$ yields the estimate for $\beta_x^2$. 
		Since $\beta_{x}^1 - \frac{N-1}{N} = \sum_{y}^{(x)}(\abs{X_{xy}}^2 - \mathbb{E}\abs{X_{xy}}^2)$, we conclude from Markov's inequality and \cref{eq:sumaiXi2} that
		\begin{align*}
			\mathbb{P}\biggl( \beta_x^1  \ge \mathcal{C}\left(1+ \frac{\log N}{d} \right)\biggr) &\le \biggl((\mathcal{C}-1) \biggl( 1 + \frac{\log N}{d} \biggr)\biggr)^{-r} \mathbb{E}\biggl(\sum_{y}^{(x)} \bigl(\abs{X_{xy}}^2 - \mathbb{E}\abs{X_{xy}}^2\bigr) \biggr)^r \\
			&\le  \biggl(\frac{4}{\mathcal{C}-1} \biggl( 1 + \frac{\log N}{d} \biggr)^{-1} \biggl( \frac{r}{d} \vee \sqrt{\frac{r}{d}} \biggr)\biggr)^r
		\end{align*}
		for each $r \in 2\N$ and $\mathcal C >0$. 
		We choose $r = \nu d (1 + \frac{\log N}{d})$ for $\nu \geq 1$ and $\mathcal C =4 \nu \ee + 1$, which completes the proof. 
	\end{proof}

	\subsection{Typical vertices} 
	\label{sec:typical_vertices} 
	
	We now introduce a notion of typical vertices for the sparse non-Hermitian matrices from \cref{defiX}.	These matrices generalise the adjacency matrix of the directed Erd{\H o}s-R\'enyi graph.

\begin{defi}[Typical vertices]\label{def:typicalvertices}
		For $\varphi>0$, the set $\mathcal{T}_{\varphi}$ of \textup{typical vertices} is defined by 
			\begin{equation*}
			\mathcal{T}_{\varphi} := \Bigl\{ x \in [N] :~\max_{i \in \{1,2\}}\absb{\Phi^i_x}\vee \absb{\Psi^i_x }\le \varphi\Bigr\}, 
		\end{equation*}
		where for $x \in [N]$ we wrote
		\begin{equation*}
			\Phi^1_x:= \sum_{y}^{(x)}\bigg(|X_{xy}|^2 - \frac{1}{N}\bigg), \qquad \Psi^1_x:=  \sum_{y}^{(x)}\bigg(|X_{xy}|^2 - \frac{1}{N}\bigg) G^{\bset{x}}_{\ubarr{y}\ubarr{y}},
		\end{equation*}
		\begin{equation*}
			\Phi^2_x:= \sum_{y}^{(x)}\bigg(|X_{yx}|^2 - \frac{1}{N}\bigg),\qquad \Psi^2_x:=  \sum_{y}^{(x)}\bigg(|X_{yx}|^2 - \frac{1}{N}\bigg) G^{\bset{x}}_{yy}.
		\end{equation*}
	\end{defi}
 
A vertex is \emph{typical} according to the previous definition if its normalized in- and outdegree deviate from its expectation 1 less than the error parameter $\varphi$. This is encoded in the conditions on $\Phi^1_x$ and $\Phi^2_x$. Moreover, we require the \emph{local} average $\sum_{y}^{(x)} \abs{X_{xy}}^2 G_{\ubarr{y}\ubarr{y}}^{[x]}$ of the resolvent entries over neighbours of $x$ 
to be near the \emph{global} average $\frac{1}{N}\sum_{y}^{(x)} G_{\ubarr{y}\ubarr{y}}^{[x]}$ and analogously for $\sum_{y}^{(x)} \abs{X_{yx}}^2 G_{yy}^{[x]}$. 
This notion of typical vertices is motivated by \cite[Definition~4.6]{alt2021delocalization}. Note that this notion depends on the spectral parameter $w \in \C$ of the non-Hermitian matrix $X$ and the spectral parameter $z \in \C_+$ of the Hermitization $H$, see \cref{hermitizationofM} and \cref{Greenfunction}.  

The following \cref{prop:mostverticesandneigharetypical} is the main result of this section 
and states that most vertices are typical and most neighbours of each vertex are typical with very high probability. This will be crucial to prove the local law for the Hermitization in \cref{section:proof_local_law_H} below. 
For the present section, we restrict to 
\begin{equation}\label{eq:regimedlog2}
 \frac{\Gamma^2}{64 \mathfrak c} \leq  d \le \frac{(\log N)^2}{(\log \log N)^3 }
\end{equation}
where $\Gamma$ is introduced by \eqref{eq:theta} and $\ccc>0$ is the constant given by \cref{lem:roughboundsGT}. The upper bound is only needed in \cref{lem:controlatypicalvertices} below. 
In fact, in \cref{lem:sumXabGa_eq_sumGa} below, we will see that all vertices are typical at least for $d \geq (\log N)^{3/2}$. 

\nc

	\begin{prop}[Number of typical vertices] \label{prop:mostverticesandneigharetypical}
Let $\delta \in (0,1)$. Let $\Gamma$ and $d$ satisfy \cref{eq:theta} and \eqref{eq:regimedlog2}. Let $\theta$ be defined as in \eqref{eq:theta}. Then there exists a constant $\tilde{\ccc}>0$ depending only on $\nu$ and $\delta$ such that with the definition 
   \begin{equation}\label{eq:new_varphi}
 \varphi^3  = \tilde{\ccc} \frac{\Gamma^5 \log N}{d^2}
		 \end{equation}
  the following holds uniformly for $w \in \C$ and $z \in \C$ with $\Im z \geq N^{-1 + \delta}$  with very high probability on the event $\{\theta =1\}$.
		\begin{enumerate}[label=(\roman*)]
			\item \label{mostverticestypical} Most vertices are typical:
			\begin{equation*}
				\abs{\mathcal{T}_{\varphi}^c} \le \ee^{\frac{\varphi^2d}{2^{14} (\ee\Gamma)^2}}  + 4N \ee^{-\frac{\varphi^2d}{2^{13} (\ee\Gamma)^2}} .
			\end{equation*}
			\item \label{mostneighboursaretypical} Most neighbours of each vertex are typical:
			\begin{equation*}
				\Biggl(\sum_{y\in \mathcal{T}^c_{\varphi}}^{(x)} \abs{X_{xy}}^2\Biggr) \vee \Biggl(\sum_{y\in \mathcal{T}^c_{\varphi}}^{(x)} \abs{X_{yx}}^2 \Biggr)\le 2 \ccc \varphi\Gamma^{-3} + 5\mathcal{C}d^4\ee^{-\frac{\varphi^2d}{2^{14} (\ee\Gamma)^2}}
			\end{equation*}
			uniformly for $x \in [N]$.  
		\end{enumerate}
	\end{prop}

		The previous proposition is the analogue of \cite[Proposition 4.8]{alt2021delocalization} 
		in the setting of non-Hermitian sparse random matrices. The proof of \cref{prop:mostverticesandneigharetypical} is presented in \cref{subsection:proofs_typical_vert}. In fact, it follows an analogous  strategy as the proof of \cite[Proposition~4.8]{alt2021delocalization} in   \cite[Section~4.2]{alt2021delocalization}.
	
	\subsection{Properties of $v$ and $v_\beta$}\label{section:properties_v_v_beta}
	
	In this section, we state and prove a number of estimates on $v$ and $v_\beta$ 
	defined in \cref{eq:v} and \cref{eq:def_v_beta}, respectively. 
	They will be used throughout the remainder of this work. 
	 To formulate these estimates, we use the notations introduced in \cref{sec:notation}.  
	 The implicit constants in $\lesssim$ and $\asymp$ will be allowed to depend on two constants, $\delta$ and $L$.

	\begin{lemma}[Properties of $v$]\label{lem:profv}
		Let $v$ be the unique solution of \cref{eq:v}. Let $\delta \in (0,1)$ and $L>0$ be constants.
		Then, for all $w\in \CC$ and $\eta \in (0 , \infty)$, we have
		\begin{equation}\label{eq:v_leq_1}
			v(w,\eta) \le \min \biggl\{1, \frac{1}{\eta}\biggr\}.  
		\end{equation}
		Uniformly for all $\eta \in (0,L]$ and $w \in \mathrm{D}_2$, we have 
		\begin{equation} \label{eq:v_full_scaling} 
			v(w,\eta) \asymp_L \begin{cases} (1 - \abs{w}^2)^{1/2} + \eta^{1/3} & \text{ if } \abs{w} \leq 1, \\ 
				\frac{\eta}{\abs{w}^2 - 1 + \eta^{2/3}} & \text{ if } \abs{w} \geq 1. \end{cases}
		\end{equation}
		In particular, for all $\eta \in (0, L]$ and $w \in \mathrm{D}_2$, we have
		\begin{equation}\label{eq:vgeeta}
			v \gtrsim_{L} \eta.
		\end{equation}
		Finally, for all $\eta \in (0,L]$ and $w \in \mathrm{D}_{1-\delta}$, we have 
		\begin{equation}\label{eq:v_asymp_1} 
			v(w,\eta) \gtrsim_{\delta,L} 1. 
		\end{equation}   
	\end{lemma}

	The previous lemma is well-known in the literature about (local) circular laws,   see e.g.\ \cite{BYY2014,AEKinhomogeneous,HeDigraph}. 
	For the convenience of the reader, we present the short proof in \cref{subsection:proof_properties_v}.

	\begin{lemma}[Properties of $v_\beta$] \label{lem:properties_v_beta} 
		Let $\beta = (\beta_1, \beta_2) \in [0,\infty)^2$ and let $v_\beta$ be defined as in \eqref{eq:def_v_beta}. 
		Then, for all $\eta \in (0, \infty)$, $w \in \C$ and $\beta=(\beta^1,\beta^2) \in [0,\infty)^2$, we have 
		\begin{align}\label{eq:lipschitz_v_beta}
			|v-v_{\beta}| & \leq v_\beta|\beta^{1} -1| + |\beta^{2} -1 |, \\ 
			v_\beta & \leq \eta^{-1}.  \label{eq:v_beta_leq_eta_inverse} 
		\end{align} 
		Moreover, if $L >0$ and $\delta \in (0,1)$ are fixed, then for all $\eta \in (0,L]$ and $w \in \mathrm{D}_{1-\delta}$, we have 
		\begin{equation}\label{eq:encadrementmi}
			v_\beta \asymp_{\delta,L} \frac{\eta + \beta^{2}}{\abs{w}^2 + (\eta + \beta^1)(\eta + \beta^{2})} \leq \frac{1 + \beta^{2}}{\abs{w}^2 + \beta^{1} \beta^{2}}  .  
		\end{equation}
		
	\end{lemma}

		For the next lemma, we recall  the definition of $u_{(\beta^1,\beta^2)}$ for $\beta^1$, $\beta^2 \in [0,\infty)$ from \cref{eq:def_u_beta}. 
		For future reference, we note that 
	\begin{equation}\label{eq:def_u_beta2}
		u_{(\beta^1,\beta^2)} = \frac{v_{(\beta^1,\beta^2)}}{\eta + \beta^2 v} = \frac{v_{(\beta^2,\beta^1)}}{\eta + \beta^1 v} = u_{{(\beta^2,\beta^1)}}.
	\end{equation}	
	
\begin{proof}
		From \eqref{eq:def_v_beta}, we conclude that 
		\begin{equation}\label{eq:inverse_v_beta}
			\frac{1}{v_\beta} =\frac{(\beta^1 v + \eta)(\beta^{2}v + \eta) + |w|^{2}}{\beta^{2}v + \eta}= \beta^1 v + \eta + \frac{|w|^2}{\beta^{2}v + \eta}.
		\end{equation}
		From this, we directly conclude \eqref{eq:v_beta_leq_eta_inverse}. 
		Moreover, \eqref{eq:inverse_v_beta} implies 
		\[ v - v_\beta = v v_\beta \bigg( \frac{1}{v_\beta} - \frac{1}{v} \bigg) = v_{\beta} v^2 (\beta^1-1) + 
		\frac{v_\beta v^2\abs{w}^2}{(\eta + v)(\eta + \beta^2 v)}(1-\beta^2). 
		\] 
		Hence, \eqref{eq:lipschitz_v_beta} follows as
		$v\leq 1$ for all $w \in \C$ and $\eta>0$ due to \eqref{eq:v_leq_1}, which also implies 
		\[ 
		\frac{v_\beta v^2 |w|^2}{(v+\eta)\left( \beta^{2}v + \eta \right)} \le \frac{|w|^2}{|w|^2 +  \left(\beta^{1}v+\eta \right) \left(\beta^{2}v+\eta \right)} \le 1. 
		\] 
		Fix $\delta \in (0,1)$ and $L>0$. 
		Since $v\asymp_{\delta,L} 1$ on $\mathrm{D}_{1-\delta}$ by \cref{lem:profv}, the first relation in \eqref{eq:encadrementmi} follows from \eqref{eq:def_v_beta}. The upper bound in \eqref{eq:encadrementmi} is then a 
		consequence of $\eta \in (0,L]$. 
	\end{proof}

	\begin{lemma}[Lipschitz-continuity] \label{lem:v_beta_u_beta_Lipschitz} 
		Let $w \in \C$. Then, for any $\beta \in [0,\infty)^2$, the function $\eta \mapsto v_\beta(w,\eta)$ is Lipschitz-continous on $[N^{-1}, \infty)$ with constant $N^2$. Moreover, if $\beta \in [0,N)^2$ then $\eta \mapsto u_\beta(w,\eta)$ is 
		Lipschitz-continuous on $[N^{-1},\infty)$ with constant $N^6$. 
	\end{lemma} 
	
	\begin{proof} 
	For $v_\beta$, this follows from \cref{eq:v_beta_m_beta}, \cref{lem:m_1_Stieltjes} 
	and \cref{lem:Stieltjes_proper}.  Then the Lipschitz-continuity of $u_\beta$ is 
	a consequence of \cref{eq:def_u_beta}, the $N^2$-Lipschitz-continuity of $v$ and 
	$\max\{\beta^1,\beta^2\} \leq N$. 
	\end{proof}

		\section{Local law for Hermitization}  \label{section:proof_local_law_H} 
		
	In this section, we prove a version of \cref{thm:locallaw_H} with stronger error estimates in the bulk spectrum. 
	
	To simplify its formulation, we first define some notation. We introduce the control parameters
	\begin{equation}\label{eq:def_Lambda}
		\Lambda = \Lambda(w,\eta) = \Lambda_{\mathrm{d}} \vee \Lambda_{\mathrm{do}} \vee \Lambda_{\mathrm{o}}, 
	\end{equation}
whose constituents are defined through 
	\begin{align*}
		\Lambda_{\mathrm{d}} &= \max_{x \in [N]} \Big( \absb{G_{xx}(w,\ii \eta) - \ii v_{\beta_x}(w,\eta)} +
		\absb{G_{\ubarr{x}\ubarr{x}}(w,\ii \eta) - \ii v_{\ubarr{\beta}_x}(w,\eta)}
		\Big),\\
		\Lambda_{\mathrm{do}} &= \max_{x \in [N]} \Big( 
		\absb{G_{x\ubarr{x}}(w,\ii\eta) + w u_{\beta_x}(w,\eta)} 
		+ \absb{G_{\ubarr{x}x}(w,\ii\eta) + \overline{w} u_{{\beta}_x}(w,\eta)} \Big), \\
		\Lambda_{\mathrm{o}} &=\max_{\stackrel{x,y \in [2N]}{x \neq y \neq \ubarr{x}}}\absb{G_{xy}(w, \ii \eta)}.
	\end{align*}
	Here, $\Lambda_{\mathrm{d}}$ controls the behaviour of the diagonal entries of $G$, $\Lambda_{\mathrm{do}}$ the one of the diagonal entries of its off-diagonal blocks (cf.\  \eqref{hermitizationofM} and \eqref{Greenfunction}) and $\Lambda_{\mathrm{o}}$ the one of the remaining, i.e.
	the generic off-diagonal, entries.
	
	 The next theorem covers three cases in parallel, the one with a lower bound on $\beta_x^i$, 
	the bulk regime and the edge regime. 
	To unify the notation, we recall the definitions of $\chi_\#$ and $\mathrm{S}_\#$ from \cref{eq:def_chi} and \cref{eq:def_spectral}, respectively.
 Correspondingly, we define 
	\[ 
	\Gamma_{\mathrm{b}} :=\Gamma_{\mathrm{e}} := 1 + \frac{\log N}{d},~~\text{and}~~ \Gamma_{\mathrm{l}} :=1.
	\] 
	With these choices of $\Gamma$, we set $\varphi$ as in \cref{eq:new_varphi} and, for $w \in \C$ and $\eta >0$, set 
	\begin{equation}\label{eq:def_up_bound} 
	\psi_w(\eta) := \min \biggl\{  {\varphi}^{1/3}, \frac{ {\varphi}}{v(w,\eta)^2 + \frac{\eta}{v(w,\eta)}}\biggr\}, 
	\end{equation}
	where $v$ is as in \cref{eq:v}. Finally, we recall $d_\#$ from \cref{eq:def_d_hash}  
	and work under the assumption that 
	\begin{equation}\label{eq:regimedlogN2_and_logNalpha}
	\mathcal D d_\# \le d \le \frac{(\log N)^{2}}{\mathcal D (\log \log N)^3} 
	\end{equation}
	for some large enough constant $\mathcal D\geq 1$. 
	In particular, we have $\mathcal C \Gamma_\#^2\psi\leq 1/2$ on $\mathrm{S}_\#$ if 
	$\mathcal D \geq 1$ in \cref{eq:regimedlogN2_and_logNalpha} is chosen large enough.

	\begin{thm}[Conditional local law]\label{thm:locallaw_edge}
	 Fix $\# \in \{ \mathrm{l}, \mathrm{b}, \mathrm{e}\}$. Let $d$ satisfy \eqref{eq:regimedlogN2_and_logNalpha}. 
		Fix $ 0 < \delta \le 1/2$ and $L \ge 1$. Let $X$ be a sparse matrix as in \cref{defiX}, define $M$ as in \eqref{eq:def_M_X_plus_f} with some deterministic $f\in \CC$ satisfying $0 \le \abs{f} \le N^{\delta/6}$, $G$ as in \eqref{Greenfunction} and $g$ as in \eqref{eq:def_g}. 
		Then  
		\begin{align} 
			\chi_\# \Lambda(w, \eta) & \leq \mathcal C\Gamma_\#^2 \psi_w(\eta) \label{eq:local_law_Lambda} \\ 
			\chi_\#  \absb{g(w,\ii \eta) - \ii v(w,\eta)} & \leq \mathcal C \psi_w(\eta)	\label{eq:loclawStieljes2} 
		\end{align}
		with very high probability uniformly for all $w \in \mathrm{S}_\#$ and $\eta \in [N^{-1 + \delta}, L]$.
	\end{thm}
	Throughout the following, we often drop the index of $\Gamma$ to simplify the notation. In that case, the argument holds for either choice of $\Gamma$.
	
	We notice that owing to \eqref{eq:new_varphi} and the lower bound on $d$ in \eqref{eq:regimedlogN2_and_logNalpha}, $\varphi >0 $ satisfies 
	\begin{equation} \label{eq:d_inverse_square_root_smaller_varphi_a}
		\max\{d^{-1/2}, \Gamma^3 d^{-1}, \Gamma d^{-1}\} \le \varphi \le 1/2.
	\end{equation}
	 Owing to \cref{prop:mostverticesandneigharetypical} \cref{mostverticestypical}, 
	 for any $1 \leq \Gamma \leq 1 + \frac{\log N}{d}$, with very high probability we have
	\begin{equation}\label{eq:number_atypical_vertices} 
	 \Gamma 	\frac{\abs{\mathcal{T}^c_{\varphi}}}{N} \le\Gamma \frac{e^{\frac{\varphi^2d}{2^{10} (e\Gamma)^2}}}{N}  + 4 \Gamma e^{-\frac{\varphi^2d}{2^{9} (e\Gamma)^2}}   \le \mathcal{C}\varphi.
	\end{equation} 
	Throughout this section, the constants $\ccc>0$ and $\tilde{\ccc} >0 $ are 
	chosen as in \cref{lem:roughboundsGT} and \cref{lem:controlatypicalvertices}, respectively; see also \cref{eq:regimedlog2}.

	\subsection{Analysis of $s$} 
	
	We denote the average of the diagonal entries $G_{xx}$ over the typical vertices $x \in \mathcal T_\varphi \cup \widehat{\mathcal T}_\varphi$ by 
	\begin{equation}\label{eq:meanoftypicalvertices}
		s:= \frac{1}{2\abs{\mathcal{T}_{\varphi}}} \sum_{x \in \mathcal{T}_{\varphi}} \left(G_{xx} + G_{\ubarr{x}\ubarr{x}}\right). 
	\end{equation} 
	In this subsection, we will study $s$ instead of $g$ and show that it is close to $\ii v$. 
	Eventually, we will prove that $g$ is well approximated by $s$. 
	The first step is to derive a cubic equation for $s$ from the relations in \cref{lem:ASCEq}.  
	
 Throughout the following, we will often work on the event $\{ \Lambda \leq 1\}$. 
	With appropriately chosen $\Gamma$, we conclude $\{ \Lambda \leq 1\} \subset \{ \theta = 1\}$ from \cref{lem:properties_v_beta} and the definition of $\theta$ from 
	\cref{eq:theta} as well as \cref{lem:roughboundonbetaix} or the definitions of $\chi_\#$ and $\mathrm{S}_\#$. In particular, the results previously shown on the event $\{ \theta =1\}$ are applicable on $\{ \Lambda \leq 1\}$. 
	
	\nc

	\begin{lemma}\label{lem:restmoy-s}
		Let $L \ge 1$ be a constant. Then on the event $\{ \theta =1\}$ we have with very high probability 
		\begin{align*}
			\bigg|	\sum_{y}^{(x)} \abs{X_{xy}}^2 G_{\ubarr{y}\ubarr{y}}^{\bset{x}} -s \bigg|\le \mathcal C \varphi  ~~ \text{and}~~\bigg| \sum_{y}^{(x)} \abs{X_{yx}}^2 G_{yy}^{\bset{x}} -s \bigg| \le  \mathcal{C}\varphi 
		\end{align*}
		uniformly for all $x \in \mathcal T_{\varphi}$ and all $N^{-1+ \delta} \le \eta \le L$ and $ \abs{w} \le 2$. 

	\end{lemma}

	\begin{proof}
		For any $x \in \mathcal{T}_\varphi$, we first observe that 
		\begin{equation*}
			\sum_{y}^{(x)} \abs{X_{xy}}^2 G_{\ubarr{y}\ubarr{y}}^{\bset{x}} =s + \eps^1_x, \qquad \qquad \sum_{y}^{(x)} \abs{X_{yx}}^2 G_{yy}^{\bset{x}} =s + \eps^2_x,
		\end{equation*}
		where we used \cref{def:typicalvertices} and \cref{lem:traceg1=g2} and  introduced 
		\begin{align}
			\eps^1_x &:=-\frac{\abs{\mathcal{T}_{\varphi}^c}}{N} s + \frac{1}{2N} \sum_{y \notin \mathcal{T}_{\varphi}}\left( G_{yy} + G_{\ubarr{y}\ubarr{y}} \right) - \frac{1}{N} G_{\ubarr{x} \ubarr{x}} +\frac{1}{N} \sum_{y}^{(x)} \left(G_{\ubarr{y}\ubarr{y}}^{\bset{x}} - G_{\ubarr{y}\ubarr{y}}\right) + \Psi_{x}^1, \label{eq:defi_epsilon_1}\\
			\eps^2_x &:=-\frac{\abs{\mathcal{T}_{\varphi}^c}}{N} s + \frac{1}{2N} \sum_{y \notin \mathcal{T}_{\varphi}}\left( G_{yy} + G_{\ubarr{y}\ubarr{y}} \right) - \frac{1}{N} G_{xx} + \frac{1}{N} \sum_{y}^{(x)}  \left(G_{yy}^{\bset{x}} - G_{yy}\right) + \Psi_{x}^2.\label{eq:defi_epsilon_2}
		\end{align}
		By \cref{def:typicalvertices} we have $\abs{\Psi_x^1}\leq \varphi$ and, thus,  \eqref{eq:roughdifferenceGGTd-1} in \cref{lem:roughboundsrest} yields 
		\[ 
		\absb{\eps_{x}^1} \le \mathcal{C}  \Gamma \nc \frac{\abs{\mathcal{T}_{\varphi}^c}}{N}+ \mathcal{C}\frac{ \Gamma \nc }{N}   +  \max_{i \in [2N] \backslash\{x,\ubarr{x}\}}\absb{G_{ii}^{\bset{x}} - G_{ii}} + \Psi_{x}^1
		\le \mathcal{C} \Gamma \nc  \frac{\abs{\mathcal{T}_{\varphi}^c}}{N} +  \mathcal C \frac{\Gamma^3}{d} \nc + \mathcal{C} \varphi \leq \mathcal C \varphi
		\] 
		with very high probability, 
		where we used $\Gamma^3 d^{-1} \le \varphi$ due to \cref{eq:d_inverse_square_root_smaller_varphi_a} as well as \cref{eq:number_atypical_vertices} 
		in the last inequality.  
		Since the same reasoning applies to $\eps_{x}^2$, we have completed the proof of \cref{lem:restmoy-s}.  
	\end{proof}

	\begin{lemma}\label{lem:meanoftypicalvertices}
		Let $L \geq 1$ be a constant. Then, on the event $\{\Lambda \leq 1\}$ we have with very high probability 
		\begin{align} \label{eq:s_G_xx_relation} 
			\absb{s^2 G_{xx} + 2\ii \eta s G_{xx} + s -(\eta^2 + \abs{w}^2) G_{xx} + \ii \eta} &\le  \mathcal{C} \varphi\\ 
			\label{eq:stabfors}
			\absb{s^3  + 2 \ii \eta s^2 +(1 - \eta^2 - \abs{w}^2)s+ \ii \eta} & \leq \mathcal C   \varphi .
		\end{align}
		uniformly for all $x \in \mathcal T_{\varphi}\cup \ubarr{\mathcal T_{\varphi}}$ and all $N^{-1+ \delta} \le \eta \le L$ and $ \abs{w} \le 2$.
	\end{lemma}
	
	We note that the left-hand side of \eqref{eq:stabfors} vanishes due to \eqref{eq:v}
	if $s$ is replaced by $\ii v$. 
	
	For $y \in \mathcal T_{\varphi} \cup \ubarr{\mathcal T}_{\varphi}$, we now show that  
	\begin{equation}
	0 \leq v_{\beta_y} \lesssim 1, \qquad \qquad 0 \leq u_{\beta_y} \lesssim 1  \label{eq:v_beta_x_typical_bounded}. 
	 \end{equation} 
	 In the case $\abs{w} \leq 1/2$, the bound on $v_{\beta_y}$ in \cref{eq:v_beta_x_typical_bounded} follows from $v \asymp 1$ by \cref{eq:v_asymp_1}, $1/2 \leq \beta_y^i \leq 3/2$ as $y \in \mathcal{T}_{\varphi} \cup \ubarr{\mathcal{T}}_{\varphi}$ and \cref{eq:def_v_beta}. 
	For $\abs{w} \geq 1/2$,  the estimate on $v_{\beta_y}$ in \cref{eq:v_beta_x_typical_bounded} is a consequence of $v \leq 1$ by \cref{eq:v_leq_1}, $\beta_y^2 \leq 3/2$ and \cref{eq:def_v_beta}. 
	The bound on $u_{\beta_y}$ is shown by an analogous argument.  
	
	On the event $\{ \Lambda \leq 1 \}$, we immediate conclude 
	\begin{equation} \label{eq:G_xx_typical} 
	\max_{y \in  \mathcal T_{\varphi} \cup \ubarr{\mathcal T}_{\varphi}} \big( \abs{G_{yy}} + \abs{G_{\ubarr{y}y}} \big) \lesssim 1. 
	\end{equation}
	\begin{proof}
		We prove the inequality for $x \in \mathcal{T}_{\varphi} $ and a similar proof holds for $G_{\ubarr{x}\ubarr{x}}$. We define 
		\begin{align*}
			\eps^3_x := Y_{x}^1 G_{xx} + Z^{1}_{x} G_{\ubarr{x}x}, \qquad \qquad  \eps^4_x := Y_{x}^2 G_{\ubarr{x}x} + Z^{2}_{x} G_{xx}, 
		\end{align*}
		and rewrite the first and third equations of \cref{lem:ASCEq}  to obtain 
		\begin{align*}
			1 &= - (\ii\eta +s+ \eps_{x}^1) G_{xx} - w G_{\ubarr{x}x} + \eps_x^3,\\
			(\ii\eta  + s + \eps^2_{x}) G_{\ubarr{x}x} &= -\barr{w}G_{xx} + \eps_{x}^4,
		\end{align*}
		where $\eps_{x}^1 = \sum_{y}^{(x)} \abs{X_{xy}}^2G_{\ubarr{y}\ubarr{y}}^{\bset{x}} - s$ and $\eps_{x}^2 = \sum_{y}^{(x)} \abs{X_{yx}}^2G_{yy}^{\bset{x}} - s$ are the error terms from \cref{lem:restmoy-s}. Next, we multiply the first line by $(\ii\eta +  s + \eps^2_{x})$ and use the second to get rid 
		of $G_{\ubarr{x}x}$ in the resulting identity. This yields  
		\begin{equation}\label{eq:Giifunctionofs}
			\begin{aligned}
				& s^2 G_{xx} + 2\ii\eta s G_{xx} + s + (\ii\eta ^2 - \abs{w}^2) G_{xx} + \ii\eta  \\ 
				& \qquad \qquad\qquad \qquad \qquad 	=- w \eps_{x}^4 - \eps_{x}^2  + (\ii\eta +s+\eps_{x}^2)(\eps_{x}^3 - \eps_{x}^1G_{xx}) - \eps_{x}^2(\ii\eta +s) G_{xx}.
			\end{aligned} 
		\end{equation}
		Now we work on the event $ \{\Lambda \leq 1\}$. On this event and for $\eta \in [N^{-1+ \delta},L]$ and $\abs{w} \le 2$, we have $\eta\vee \abs{s} \vee \max_{y \in \mathcal T_{\varphi} \cup \widehat{\mathcal T_{\varphi}}} \abs{G_{yy}} \lesssim 1$ due to \cref{eq:G_xx_typical}.  
		From \eqref{eq:boundrest} in \cref{lem:roughboundsrest}, $x \in \mathcal T_{\varphi} \cup \widehat{\mathcal T_{\varphi}}$ and \cref{eq:G_xx_typical}, we conclude $\abs{\eps_{x}^3}+\abs{\eps_{x}^4}= \mathcal{O} (  \frac{\Gamma}{d} )$.
		Noticing that $\varphi \ge \Gamma/d$, it then remains to apply \cref{lem:restmoy-s} to conclude \eqref{eq:s_G_xx_relation}. 
		Averaging \eqref{eq:s_G_xx_relation} over $x \in \mathcal T_{\varphi} \cup \ubarr{\mathcal T}_\varphi$ yields \eqref{eq:stabfors}. 
	\end{proof}
	
	We now translate the cubic equation for $v$, \eqref{eq:v}, and the approximate cubic for $s$, \eqref{eq:stabfors}, into a cubic inequality for
	\begin{equation*}
		\Delta := s- \ii v.
	\end{equation*} 
	\begin{lemma}\label{lem:3rddegeqonDelta}
		Let $\delta \in (0,1/2]$ and $L \geq 1$ be constants. Then there exists a constant $\mathcal C >0$ depending only on $\delta$, $L$ and $\nu$ such that uniformly for all $\eta_*$, $\eta^* \in [N^{-1+\delta},L]$ with $\eta_* < \eta^*$ and all $w \in \mathrm{D}_2$, the 
		event 
		\[ 
		\{ |\Delta^3 - ( 3 \ii v + 2 \ii \eta)\Delta^2 - (4 \eta v + 3 v^2 + \eta^2 + \abs{w}^2 -1) \Delta| \leq \mathcal C {\varphi} \text{ for all } \eta \in [\eta_*,\eta^*] 
		\}
		\] 
		holds with very high probability on the event $\{  \Lambda \leq 1 \nc  \text{ for all } \eta \in [\eta_*, \eta^*]\}$. 
	\end{lemma} 
	\begin{proof}
		For all $ N^{-1+ \delta} < \eta \le L$ and $w \in \CC$, on the event $\{  \Lambda(w,\eta) \leq 1  \}$, we expand $s =  \Delta + \ii v $ in the left-hand side of \eqref{eq:stabfors}
		and obtain 
		\[ 
		|\Delta^3 + ( 3 \ii v + 2 \ii \eta)\Delta^2 - (4 \eta v + 3 v^2 + \eta^2 + \abs{w}^2 -1) \Delta| \leq \mathcal C { \varphi}
		\] 
		with very high probability, 
		where the term independent of $\Delta$ vanishes due to \eqref{eq:v}. It remains to use a grid argument for $\eta \in [\eta_\ast, \eta^\ast]$ exploiting that all previous functions are $N^{2}$-Lipschitz continuous in $\eta$ on $[N^{-1},\infty)$ due to \cref{lem:v_beta_u_beta_Lipschitz} and \cref{lem:disttospect} to conclude.
	\end{proof}

	In the following lemma, the $\eta$-dependence of various quantities like $\Delta$ and $\psi_w$ (defined in \eqref{eq:def_up_bound}) is important, which we thus highlight in our notation. 
	\begin{lemma}\label{lem:Deltasmall}
		Let $\delta \in (0,1/2]$ and $L \geq 1$ be constants. Then there exists a constant $\mathcal C >0$ depending only on $\delta$, $L$ and $\nu$ such that the following holds uniformly for all $\eta_*$, $\eta^* \in [N^{-1+\delta},L]$ with $\eta_* < \eta^*$ and all $w \in \mathrm{D}_2$. If 
		$\theta (\eta^\ast) |\Delta(\eta^{\ast})| \leq \mathcal C \psi_w(\eta^\ast)$
		with very high probability, then the event
		\[ 
		\left\{ |\Delta(\eta)| \leq \mathcal C \psi_w(\eta)  \text{ for all } \eta \in [\eta_*,\eta^*] 
		\right\}
		\] 
		holds with very high probability on the event $\{  \Lambda \leq 1 \nc  \text{ for all } \eta \in [\eta_*, \eta^*]\}$. 
	\end{lemma}
	
	\begin{proof}
		Throughout this proof, we work on the event $\{  \Lambda \leq 1 \nc \text{ for all } \eta \in [ \eta_{\ast}, \eta^{\ast}]\}$. 
		We will apply the stability estimate \cref{lem:stability2} to the cubic equation from \cref{lem:3rddegeqonDelta}. To that end, we check the assumptions of \cref{lem:stability2}. 
		For $\eta >0$ and $w \in \C$, we define 
		\begin{equation*} \xi_1 := 4 \eta v + 3 v^2 + \eta^2 + \abs{w}^2 -1 = 2 \eta v + 2 v^{2} + \frac{\eta}{v}
			~~ \text{ and }~~ \xi_2 := 3\ii v + 2\ii \eta, 
		\end{equation*}
		where the equality follows from \eqref{eq:v}. Since $\eta \leq v^{-1}$ by \eqref{eq:v_leq_1}, 
		we conclude $\eta v \leq \eta^{1/2} v^{-1/2} v \leq \frac{\eta}{v} + v^2$ and, thus, $\xi_1 \asymp v^2 + \frac{\eta}{v}$ as well as  $\abs{\xi_2} \lesssim \xi_1^{1/2}$. Defining $\tilde{\xi}(w,\eta) := \absN{\absN{w}^2 - 1} + \eta^{2/3}$ for $w \in \C$ and $\eta>0$, using \eqref{eq:v_full_scaling}, distinguishing the cases $\abs{w} \leq 1$ and $\abs{w} \geq 1$ and realising that $v^2 \gtrsim \eta/v$ in the former and $v^2 \lesssim \eta/v$ in the latter case yield $\xi_1 \asymp \tilde{\xi}$ and $\abs{\xi_2} \lesssim \tilde{\xi}^{ 1/2 }$. Since $\tilde{\xi}$ is increasing in $\eta$, we have verified the assumptions of \cref{lem:stability2} for the cubic equation from \cref{lem:3rddegeqonDelta}. 
		Therefore, these two lemmas imply that the intersection of the event from \cref{lem:3rddegeqonDelta} and $\{ \Lambda \leq 1 \text{ for all } \eta \in [\eta_*,\eta^*]\}$ is contained in the event
		\[ 
		\biggl\{ |\Delta| \lesssim \min \biggl\{   {\varphi}^{1/3}\nc, \biggl(\frac{ {\varphi} \nc}{v^2 + \frac{\eta}{v}}\biggr)^{1/2} ,\frac{{\varphi}\nc}{v^2 + \frac{\eta}{v}}\biggr\} \text{ for all } \eta \in [\eta_*,\eta^*] 
		\biggr\},
		\] 
		which thus holds with very high probability. We now prove that the middle term in the previous minimum can be omitted. We abbreviate $t:= v^2 + \frac{\eta}{v}$. On the one hand, if $ {{\varphi}}/t<{{\varphi}}^{1/3}$ then we have
		${\varphi}/t \le {\varphi}^{1/2}/t^{1/2}$. 
		On the other hand, if ${\varphi}^{1/3}\le {\varphi}/t$ then we have
 	${\varphi}^{1/2} = \left(\frac{{\varphi}}{t}\right)^{1/2} t^{1/2} \ge {{\varphi}^{1/6}} t^{1/2}\geq  {\varphi}^{1/3}t^{1/2}$, which concludes the proof.
	\end{proof}
	
	\subsection{Analysis of $\Lambda$ and bootstrapping} 
	
	We now deduce from the estimates on $s-\ii v$ in the previous subsection bounds on $\Lambda$. 
	An iterative bootstrapping argument will then prove \cref{thm:locallaw_edge}. 

	\begin{prop}\label{prop:phi_Lambda_small2}
		Let $L \geq 1$ and $\delta \in (0,1/2]$ be constants and $\# \in \{\mathrm{u},\mathrm{l}\}$. Then there exists a constant $\mathcal C >0$ depending only on $\delta$, $L$ and $\nu$ such that the following holds uniformly for all $\eta_*$, $\eta^* \in [N^{-1+\delta},L]$ with $\eta_* < \eta^*$ and all $w \in \mathrm{S}_\#$. If $\Delta(\eta^*) \leq \mathcal C  \psi_w(\eta^*)$ with very high probability, then the event 
		\begin{equation*}
			\left\{ \chi_\# \Lambda (\eta)\le \mathcal{C}\Gamma_\#^2 \psi_w(\eta) \text{ and } \chi_\# \Delta(\eta) \leq \mathcal C \psi_w(\eta) ~~\text{for all}~ \eta \in [ \eta_{\ast}, \eta^{\ast}] \right\}
		\end{equation*}
		holds with very high probability on the event $\{ \Lambda \leq 1 ~\text{for all} ~ \eta \in [\eta_{\ast}, \eta^{\ast}]\}$. 
	\end{prop}
	\begin{proof}
		We fix $0 <\delta <1$ and $ L >0$ and consider $w \in \mathrm{S}_\#$ and $\eta_{\ast},\eta^\ast \in \RR$ such that $ N^{-1+\delta} \le \eta_{\ast} < \eta^{\ast} \le L$. We write $ \Omega_\Lambda :=\{ \Lambda \leq 1 \text{ for all } \eta \in [ \eta_{\ast}, \eta^{\ast}]\}$ and $\Omega_{\#} := \{ \chi_\#=1\}$. 	The estimates in this proof will hold uniformly for $w \in \mathrm{S}_\#$ and $\eta \in [\eta_*,\eta^*]$. For the sake of clarity, we avoid repeating this in most instances.
		We start with some preparations. 
		Throughout the proof, we write $\psi$ instead of $\psi_w(\eta)$ except when the arguments are crucial.  
		First we conclude from \cref{lem:Deltasmall} that there is a constant $\mathcal C>0$ depending only on $\delta$, $L$ and $\nu$ such that 
		\begin{equation} \label{eq:proof_bootstrapp_edge_aux2} 
			\abs{s - \ii v} \leq \mathcal C \psi \quad \text{ for all } \eta \in [\eta_\ast, \eta^\ast]
		\end{equation} 
		with very high probability on $\Omega_\Lambda$. In particular, this shows the desired bound on $\Delta(\eta)$. 
		
		As $\varphi   \le 1$ and $v^2 + \eta/v \lesssim_L 1$ due to \eqref{eq:vgeeta} and \eqref{eq:v_leq_1}, we have $\varphi \lesssim_L\psi$, which will be used tacitly throughout this proof. 
		
		We now turn to bounding $\Lambda$. 
		Recall that on the event $\Omega_{\Lambda}$ we have $\Lambda_{\mathrm{o}} \le \mathcal{C} \Gamma^2 d^{-1/2}\le \mathcal C \varphi $ due to \cref{eq:roughboundGd-1/2} in \cref{lem:roughboundsrest}. 
		To estimate $\Lambda_\mathrm{d}$, we deduce from \cref{lem:meanoftypicalvertices} and \eqref{eq:proof_bootstrapp_edge_aux2} that 
		\begin{equation}\label{eq:proof_bootstrapp_edge_aux3} 
			\big|G_{xx} \bigl( v^2  +2 \eta v + \eta^2 + \abs{w}^2 \bigr) -  \ii (\eta + v) \big| \le  \mathcal{C}\psi
		\end{equation}
		with very high probability on $\Omega_\Lambda$ for all $x \in \mathcal{T}_{\varphi}$. 
		By \eqref{eq:v} we have $v^2  +2 \eta v + \eta^2 + \abs{w}^2= \frac{\eta +v}{v}$. 
		Thus, multiplying \eqref{eq:proof_bootstrapp_edge_aux3} by $v/(\eta+ v)$ and using $v/(v+ \eta) \le 1$ yield 
		\begin{equation}\label{eq:Gxx=vnoboundonbeta}
			|G_{xx} - \ii v | \le \mathcal C  \psi
		\end{equation} 
 for all $x \in \mathcal T_\varphi$ with very high probability on $\Omega_{\Lambda}$. We now argue that for each $x \in [N]$, 
		\begin{align}\label{eq:useof_most_vert_typi2}
			\bigg|\sum_{y\in \mathcal{T}_{\varphi}^c} \abs{X_{xy}}^2 G_{\ubarr{y}\ubarr{y}}^{\bset{x}}\bigg| \le  \mathcal{C} \varphi \Gamma^{-2}  \quad \text{ and }\quad \bigg|\sum_{y\in \mathcal{T}_{\varphi}}\abs{X_{xy}}^2 G_{\ubarr{y}\ubarr{y}}^{\bset{x}} -\ii \beta_x^{1}v \bigg| \le  \mathcal{C} ( \beta_x^1 \psi + \varphi \Gamma^{-3}) 
		\end{align}
	with very high probability on the event $\Omega_{\Lambda}$. Indeed, we first note that for all $k,p \in \NN$ fixed 
	\begin{equation}\label{eq:bound_Gamma_non_typic}
	  \Gamma^k  d^p e^{-\frac{\varphi^2d}{2^{13} (e\Gamma)^2}} \lesssim \varphi
	 	\end{equation}
due to the definition of $\varphi$ in \cref{eq:new_varphi} and the upper bound on $d$ from \cref{eq:regimedlogN2_and_logNalpha}. 
		Combining this inequality with 	\cref{prop:mostverticesandneigharetypical} \cref{mostneighboursaretypical} and \cref{lem:roughboundsGT}
		yields the first inequality in \eqref{eq:useof_most_vert_typi2} as well as 
		\begin{equation} \label{eq:proof_bootstrapp_edge_aux5} 
			\bigg|\sum_{y\in \mathcal{T}_{\varphi}} \abs{X_{xy}}^2 - \beta_{x}^1 \bigg|\le  \mathcal{C} \varphi  \Gamma^{-3} 
		\end{equation} 
		with very high probability on $\Omega_\Lambda$. 
		Finally, \eqref{eq:roughdifferenceGGTd-1}, \eqref{eq:Gxx=vnoboundonbeta} and \eqref{eq:proof_bootstrapp_edge_aux5}  imply the last inequality of~\eqref{eq:useof_most_vert_typi2}.  
		
		The next step is to apply the bounds in \eqref{eq:useof_most_vert_typi2} to the approximate self-consistent equations. Indeed for all $x \in [N]$, we rewrite the first and third equations of \cref{lem:ASCEq} as 
		\begin{align}
			1&= - (\ii \eta + \ii \beta_{x}^1 v) G_{xx} - w G_{\ubarr{x}x} + r_{x}^1 \label{eq:proof_bootstrapp_edge_aux1} \\
			0 &= - (\ii \eta + \ii \beta_{x}^2 v) G_{\ubarr{x}x} - \barr{w} G_{xx} + r_{x}^2. 
			\label{eq:proof_bootstrapp_edge_aux4} 
		\end{align}
		Multiplying \eqref{eq:proof_bootstrapp_edge_aux1} by $(\ii \eta + \ii \beta_{x}^2 v)$ and using \eqref{eq:proof_bootstrapp_edge_aux4} yield 
		\begin{equation}\label{eq:equal_Gxx-vbeta}
			G_{xx} - \ii v_{\beta_x} = \frac{ w r_{x}^2 - (\ii\eta +\ii \beta_x^2 v) r_{x}^1}{\abs{w}^2 + (\eta + \beta_{x}^1 v)(\eta + \beta_{x}^2v)}
		\end{equation}
		with $v_{\beta}$ from \eqref{eq:def_v_beta}. 
		
		We now consider the case $\# = \mathrm{u}$. 
		Besides, using that $\abs{G_{\ubarr{x}x}} \lesssim 1$ on $\Omega_{\Lambda}$ and $\mathrm{S}_\mathrm{u}$ due to \eqref{eq:def_u_beta}, 
		we have 
			\begin{equation}\label{eq:estimate_ri_Su}
			r_x^1 = O(\abs{G_{xx}} (\varphi \Gamma^{-2} +  \beta^1_x \psi  + \Gamma d^{-1}) + \Gamma d^{-1} ) 
		, \qquad r_x^2= O(\beta^2_x \psi + \varphi \Gamma^{-2} + \Gamma d^{-1} + \abs{G_{xx}} \Gamma d^{-1} ).
		\end{equation}
	 On $\mathrm{S}_\mathrm{u}$, we conclude that
		\[ 
		G_{xx} - \ii v_{\beta_x} = O ( (1 + \beta_x^2 )\psi) + O( (1 + \beta_x^1 + \beta_x^2 ) \psi) = O( ( 1  + \beta_x^1 + \beta_x^2 ) \psi) = O( (1 + d^{-1} \log N)\psi ) 
		\]
		where we used that $(\eta + \beta_x^2 v) \abs{r_x^1} \lesssim (1 + \beta_x^2 v^2 \beta_x^2) \beta_x^1 \psi + (1 + \beta_x^2) \psi $ as $\abs{G_{xx}} 
		\lesssim 1 + v\beta_x^2$ on $\mathrm{S}_\mathrm{u}$ and $\{ \Lambda \leq 1\}$ as well as distinguished the cases $\beta_x^2 v^2 \beta_x^2 \geq 1$ and $\beta_x^2 v^2 \beta_x^2 \leq 1$. 
		So we get 
		\begin{equation}\label{eq:estimate_Gxx_Su}
\abs{G_{xx} - \ii v_{\beta_x}} \lesssim \Gamma \psi 		    
		\end{equation}
		for all $x$. We now bound $\abs{G_{ \ubarr{x}x} + \barr{w} u_{\beta_{x}}}$ for $x \in [N]$ on the event $\Omega_{\Lambda} \cap \Omega_{\mathrm u} $. To that end, we conclude from \eqref{eq:proof_bootstrapp_edge_aux1} that
		\begin{align*}
			1 + wG_{\ubarr{x}x} &= ( \eta + \beta_{x}^1 v) v_{\beta_x} + (\ii \eta + \ii \beta_{x}^1 v ) (G_{xx} - \ii v_{\beta_x}) + r_{x}^1\\
			&= 1 - \abs{w}^2 u_{\beta_x} + (\ii \eta + \ii \beta_{x}^1 v) (G_{xx} - \ii v_{\beta_x}) + r_{x}^1
		\end{align*}
		where we used again the definition of $u_{\beta}$ in \eqref{eq:def_u_beta} between the first and second line. From \eqref{eq:estimate_ri_Su} and \eqref{eq:estimate_Gxx_Su} and using that $|G_{xx}|\vee \beta^{i}_x \lesssim \Gamma$ we get 
		$\abs{G_{\ubarr{x}x} + \barr{w} u_{\beta_x}}\lesssim \Gamma^2 \psi$. 
		
For $\# = \mathrm{l}$, owing to \eqref{eq:def_v_beta} we get $v_{\beta_x} \lesssim 1$ on $\mathrm{S}_\mathrm{l}$ on the event $\Omega_{\Lambda} \cap \Omega_{\mathrm{l}}$ and therefore $\max_{x,y} \abs{G_{xy}} \lesssim \Gamma_{\mathrm{l}}$. Besides, owing to \cref{eq:v_full_scaling} we also get that $v \asymp 1$ on  $\mathrm{S}_\mathrm{l}$ and therefore $\psi \asymp \varphi $. Thus, 
		\[ r_x^1 = O( (1 + \beta_x^1)\varphi), \qquad r_x^2 = O( (1 + \beta_x^2) \varphi). \] 
		Since $v\asymp 1$ on $\mathrm{S}_\mathrm{l}$,  we conclude from the equality \eqref{eq:equal_Gxx-vbeta} that with very high probability on the event $\Omega_{\Lambda} \cap \Omega_{\mathrm{l}}$
		\[ G_{xx} - \ii v_{\beta_x} = O(\varphi) \] 
		and thus from \cref{eq:proof_bootstrapp_edge_aux4} and \cref{eq:def_u_beta2} that $G_{\ubarr{x}x} + \overline{w} u_{\beta_x} = O(\varphi)$. 
		
		In total, for $\# \in \{ \mathrm{u}, \mathrm{l}\}$, we have shown that
		\[ \chi_\# \Lambda \leq \mathcal C  \Gamma_\#^2 \psi\] 
		with very high probability on the event $\Omega_\theta\cap \Omega_{\#}$ 
		uniformly for $w \in \mathrm{S}_\#$ and $\eta \in [\eta_*, \eta^*]$. Therefore, a grid argument for $\eta \in [\eta_\ast, \eta^\ast]$ employing the $N^6$-Lipschitz continuity in $\eta$ of  $(G_{xy})_{x,y}$ and $(v_{\beta_x}, v_{\ubarr{\beta}_x},u_{\beta_x})_{x}$, see \cref{lem:disttospect} and \cref{lem:v_beta_u_beta_Lipschitz}, respectively, yields the assertion of the proposition.
		
	\end{proof}
	
	We now prove \cref{thm:locallaw_edge} by a bootstrapping procedure starting from large $\eta$ and iteratively decreasing it.

	\begin{proof}[Proof of \cref{thm:locallaw_edge}]
		Let $J:=2\vee (2 \ccc^{-1})$, where $\ccc$ is the constant from \cref{lem:bulkstab}.
		Without loss of generality, we assume that $L \geq J  +1$.  
		Take $w \in \mathrm{S}$ and $N^{-1+ \delta} \le \eta_0 \le L $. 
		We set $\eta_b:= \eta_0 + bN^{-7} $ for $b \in \NN$ and $B:= \min\{ b\in \NN: \eta_b \ge  J \}$.

		Below we will show that, for all $b \in \NN \cap [0,B]$, the event 
		\begin{equation} \label{eq:proof_locallaw2_main} 
			 \Lambda \leq 1 \text{ for all } \eta \in [\eta_b,\eta_{b+1}] \quad \text{ and } \quad 
			\abs{\Delta(\eta_b)} \leq \mathcal C \psi_w(\eta_b) \quad \text{ and } \quad
			\Lambda(w,\eta_b) \leq \mathcal C \Gamma_{\#}^2 \psi_w(\eta_b) 
		\end{equation} 
		holds 
		with very high probability on $\{ \chi_{\#} = 1\}$. 
		Once \eqref{eq:proof_locallaw2_main} is proved for all $b$, the assertion of \cref{eq:local_law_Lambda} follows for $b=0$. 
		
		We now show \eqref{eq:proof_locallaw2_main} by induction on $b$. We start it with $b = B$. 
		From \cref{lem:disttospect} and the definitions \eqref{eq:def_v_beta} and \eqref{eq:def_u_beta}, for $\eta \ge 1$ we have $ \max\{\abs{G_{xy}(w, \ii \eta)}, v_{\beta_x}(w,\eta),\abs{u_{\beta_x}} ; ~x,y \in [N]\}\leq \eta^{-1}$ and, thus, $ \Lambda \leq 1$ for all $\eta \in [\eta_B,\eta_{B+1}]$ due to the definitions of $\Lambda$ and $B$.
		Moreover, together with \eqref{eq:v_leq_1}, we also conclude $\abs{\Delta(\eta)} \leq 2/\eta \leq \ccc$ for all $\eta \in [\eta_B,\eta_{B+1}]$. Therefore, 
		\cref{lem:3rddegeqonDelta} and \cref{lem:bulkstab} imply $\abs{\Delta(\eta_{B}) }\le \mathcal C \varphi \lesssim_L \psi_w(\eta_B)$ and $\abs{\Delta(\eta_{B+1})}\leq \mathcal C \psi_w(\eta_{B+1})$. Owing to \cref{prop:phi_Lambda_small2}, this proves \eqref{eq:proof_locallaw2_main} for $b =B$.

		We now turn to the induction step and assume that \eqref{eq:proof_locallaw2_main} holds for some $b \in \NN\cap [0,B]$. Owing to \cref{lem:disttospect} and \cref{lem:v_beta_u_beta_Lipschitz}, the function $\eta \mapsto \Lambda(w,\eta)$ 
		is $ 2  N^6$-Lipschitz continuous on $[N^{-1},\infty)$. Therefore, 
		\begin{equation} \label{eq:Lambda_small_Lipschitz_2} 
			\Lambda(w,\eta) \le \mathcal{C}\Gamma_{\#}^2\psi_w(\eta_{b}) +  2  N^6 \abs{\eta_{b} - \eta} \le \mathcal{C}\Gamma_{\#}^2\psi_w(\eta_{b}) + 2/N.
		\end{equation}
		for all $\eta \in [\eta_{b-1},\eta_{b}]$ on the event $\{ \Lambda(w,\eta_b) \leq \mathcal C \Gamma_{\#}^2\psi_w(\eta_{b}) \}$. Owing to \eqref{eq:new_varphi} and \eqref{eq:def_up_bound} we have $\Lambda(w,\eta) \leq 1$ for all $\eta \in [\eta_{b-1},\eta_{b}]$ if $\mathcal{D}\geq 1$ in \eqref{eq:regimedlogN2_and_logNalpha} is large enough. Hence, \cref{prop:phi_Lambda_small2} yield the desired bounds on $\abs{\Delta(\eta_{b-1})}$ and $\Lambda(w,\eta_{b-1})$, respectively, which shows \eqref{eq:proof_locallaw2_main} for $b-1$. 
		
		Finally, to show \cref{eq:loclawStieljes2}, we add and subtract $s$ and estimate 
	\begin{align*}
 \abs{g-\ii v} & \leq \abs{g-s} + \abs{s-\ii v} 
  \leq \frac{\abs{{\mathcal T_{\varphi}}^c}}{N} \abs{s} + \frac{1}{2N} \sum_{y \notin \mathcal T_{\varphi} \cup \widehat{\mathcal T_{\varphi}}}  \abs{G_{yy}}  + \abs{s-\ii v} \leq \mathcal C \psi 
 	\end{align*}
 	Here, the last step is obtained as follows. 
	Since $v\leq 1$ by \cref{eq:v_leq_1} and $\psi \leq 1$, we have $\abs{s} \leq 2$. 
	Moreover, $\abs{G_{yy}} \lesssim \Gamma$ for all $y \in [2N]$. 
	Hence, the first and second term are bounded by $\mathcal C \Gamma \abs{{\mathcal T_{\varphi}}^c}/N \lesssim \varphi \leq \psi$ by \cref{eq:number_atypical_vertices}. 
	Owing to \cref{eq:proof_locallaw2_main} and the definition of $\Delta$, we have $\abs{s-\ii v} \lesssim \psi$. 
		Therefore, $\abs{g-\ii v} \leq \mathcal C \psi$ with very high probability, which completes the proof of \cref{thm:locallaw_edge}.
	\end{proof}
	
	\section{Local law in the supercritical regime}\label{section:loc_law_supercritical}
	
	The main result of this section is an analogue of the local law for the Hermitization, \cref{thm:locallaw_edge}, in the regime 
	\begin{equation}\label{regimed3} 
		( \log N)^{3/2} \le  d \le N^{\delta/4}
	\end{equation}
	for a fixed $\delta \in (0,1)$. Throughout this section, we always assume that \eqref{regimed3} holds. 
	We define  
	\[ \widetilde{\Lambda}(w, \eta) := \widetilde{\Lambda}_\mathrm{d}(w,\eta)  \vee \widetilde{\Lambda}_{\mathrm{do}}(w,\eta) \vee {\Lambda}_{\mathrm{o}}(w,\eta) \] 
	with 
	\begin{align*} 
		\widetilde{\Lambda}_\mathrm{d}(w,\eta) & := \max_{x \in [2N]}\abs{G_{xx}(w,\ii \eta) - \ii v(w, \eta)}, 
		\\ 
		\widetilde{\Lambda}_{\mathrm{do}}(w,\eta)  & := \max_{x \in [N]} \abs{G_{x\ubarr{x}}(w,\ii\eta)+w u(w,\eta)} +   
		\max_{x \in [N]}\abs{G_{\ubarr{x}x}(w,\ii\eta)+ \overline{w} u(w,\eta)},  
	\end{align*} 
	which corresponds to the definition of $\Lambda$ in \cref{eq:def_Lambda} with $\beta_x^i \equiv 1$. Here we defined
	$u := u_{(1,1)}$ from \eqref{eq:def_u_beta}. 
	
	For all $w\in \mathrm{D}_2$ and $\eta \in (0, +\infty)$ we set
	\begin{equation*}
		\widetilde{\psi}_w(\eta) := \min \biggl\{ \biggl(\frac{\log N}{d} \biggr)^{1/6} , \frac{1}{v^2 + \eta/v} \biggl(\frac{\log N}{d}\biggr)^{1/2} \biggr\}.
	\end{equation*}

	\begin{thm}[Local law in supercritical regime]\label{thm:locallaw3}
		Fix $ 0 < \delta \le 1/2$ and $L \ge 1$. Let $d$ satisfy \eqref{regimed3}. Let $X$ be a sparse matrix as in \cref{defiX}, define $M$ as in \eqref{eq:def_M_X_plus_f} with some deterministic $0 \le f \le N^{\delta/6}$, $G$ as in \eqref{Greenfunction} and $g$ as in \eqref{eq:def_g}. 
		Then  
		\begin{equation*}
			\widetilde{\Lambda}(w,\eta) \le \mathcal C \widetilde{\psi}_w(\eta)
		\end{equation*}
		with very high probability uniformly for all $w \in \mathrm{D}_2$ and $\eta \in [N^{-1 + \delta}, L]$.
	\end{thm}
	
	By averaging over the diagonal, the previous theorem implies 
	$ |g- \ii v | \leq \mathcal C \widetilde{\psi}$.

	In this regime we will not need typical vertices anymore. In particular we have the following concentration lemma. Recall the definition of $\theta$ given in \eqref{eq:theta}.
	\begin{lemma}\label{lem:sumXabGa_eq_sumGa}
		If \eqref{regimed3} holds then,  
		with very high probability,
		\begin{align}
			\label{eq:beta_bounded}
			\max_{x \in [N],\, i \in \{1,2\}}\, |\beta_{x}^i -1| & \le \mathcal C  \sqrt{\frac{\log N}{d} } \\ 
			\label{eq:all_vert_typi}
			\theta\max_{x \in [N]} \, \biggl|\sum_{y}^{(x)} |X_{xy}|^2G_{\ubarr{y}\ubarr{y}}^{\bset{x}} - \frac{1}{N}\sum_{y}^{(x)} G_{\ubarr{y}\ubarr{y}}^{\bset{x}}\biggr| 
			& \vee
			\biggl|\sum_{y}^{(x)}  |X_{yx}|^2G_{yy}^{\bset{x}} - \frac{1}{N} \sum_{y}^{(x)}  G_{yy}^{\bset{x}}\biggr| \le \mathcal{C}  \sqrt{\frac{\log N}{d} }.
		\end{align}
	\end{lemma}
	From \eqref{eq:beta_bounded} in the previous lemma, we conclude that $\chi_{\mathrm u} = 1$, $\chi_\mathrm{l} = 1$ and $\chi = 1$ with very high probability if \eqref{regimed3} holds.
	Before we prove \cref{lem:sumXabGa_eq_sumGa}, we use it to conclude 
	\cref{thm:locallaw_H} in the regime \eqref{regimed3}. 
	
	\begin{proof}[Proof of \cref{thm:locallaw_H} in the regime \eqref{regimed3}] 
	From \eqref{eq:beta_bounded}, \eqref{eq:lipschitz_v_beta} and \eqref{eq:def_u_beta2}, we deduce that 
	\[ v_{\beta_x} = v + O\bigg( \sqrt{\frac{\log N}{d}}\bigg), 
	\qquad u_{\beta_x} = u + O \bigg( \sqrt{\frac{\log N}{d}}\bigg). 
	\]
	Therefore, in the regime \eqref{regimed3}, \cref{thm:locallaw_H} 
	follows directly from \cref{thm:locallaw3}. 
	\end{proof}

	\begin{proof}
		Throughout, we work on the event $\{\theta =1\}$ and give the detailed proof for \eqref{eq:all_vert_typi}. Owing to Markov's inequality and \eqref{eq:sumaiXi2}, we have
		\begin{equation*}
			\mathbb{P} \biggl( \biggl|\sum_{y}^{(x)} \left( |X_{xy}|^2 - 1/N \right) G_{\ubarr{y}\ubarr{y}}^{\bset{x}}\biggr| > \eps \biggr) \le \biggl( \frac{4 \Gamma}{\eps} \biggl(\frac{r}{d} \vee \sqrt{\frac{r}{d}} \biggr) \biggr)^r
		\end{equation*}
		for all $\eps$, $r >0$ and $x \in [N]$. We set $r:= \nu \log N$, observe $r/d =  \nu \frac{\log N}{d}\ll 1$ and choose $\eps := 4 \mathrm{e} \Gamma(\frac{\nu\log N}{d})^{1/2} $ to obtain the desired result. The bound \eqref{eq:beta_bounded} is proved analogously by replacing $\Gamma$ in the definition of $\eps$ with $1$.
	\end{proof}

	From the previous lemma, we deduce directly a self-consistent equation for the normalized trace $g$ of the resolvent in the next lemma. This is the analogue of \cref{lem:meanoftypicalvertices}.  
	\begin{lemma}\label{lem:stielt_first_mean}
		Let $L \geq 1$ be a constant. Then, on the event $\{\theta =1\}$, we have  
		\begin{align*}
			|g^2 G_{xx} + 2\ii \eta g G_{xx} + g -(\eta^2 + \abs{w}^2) G_{xx} + \ii \eta| &\le  \mathcal{C} \sqrt{\frac{\log N}{d}}
		\end{align*}
		with very high probability
		uniformly for all $x\in [2N]$, all $\eta \in [N^{-1+ \delta},L]$ and $w \in \mathrm{D}_2$. 
	\end{lemma}
	\begin{proof}
		We prove the assertion for $x \in [N]$ and a similar proof holds for $x \in [2N]\setminus [N]$. We work on the event $\{\theta =1\}$. Owing to \eqref{eq:roughdifferenceGGTd-1} in \cref{lem:roughboundsrest} and \eqref{eq:all_vert_typi} in \cref{lem:sumXabGa_eq_sumGa} we have 
		\begin{equation}\label{eq:concentrate_around_Stielt}
			\biggl|\sum_{y}^{(x)} |X_{xy}|^2 G_{\ubarr{y} \ubarr{y}}^{\bset{x}} - g\biggr|\lesssim \sqrt{\frac{\log N}{d}} + \frac{1}{d} \lesssim \sqrt{\frac{\log N}{d}}. 
		\end{equation}
		Therefore, owing to \eqref{eq:boundrest} from \cref{lem:roughboundsrest},  the first and third equations of \cref{lem:ASCEq} yield 
		\begin{align*}
			1 = -(\ii \eta + g ) G_{xx} - w G_{\ubarr{x}x} + \mathcal O \bigg(\sqrt{\frac{\log N}{d}}\bigg), \qquad 
			0 = -(\ii \eta + g) G_{\ubarr{x}x} - \barr{w} G_{xx} + \mathcal O \bigg(\sqrt{\frac{\log N}{d}}\bigg).
		\end{align*}
		It remains to multiply the first equation by $(\ii \eta + g)$ and replace $(\ii \eta + g)G_{\ubarr{x}x}$ using the second equation to conclude.
	\end{proof}
	In the following, we use the abbreviation 
	\begin{equation*}
		\widetilde{\Delta}:= g - \ii v.
	\end{equation*}
	\begin{lemma}\label{lem:3rddegeqonDelta2}
		Let $\delta \in (0,1/2]$ and $L \geq 1$ be constants. Then there exists a constant $\mathcal C >0$ depending only on $\delta$, $L$, $\Gamma$ and $\nu$ such that the following holds uniformly for all $\eta_*$, $\eta^* \in [N^{-1+\delta},L]$ with $\eta_* < \eta^*$ and all $w \in \mathrm{D}_2$. If 
		$\theta (\eta^\ast) |\widetilde{\Delta}(\eta^{\ast})| \leq \mathcal C \widetilde{\psi}_w(\eta^\ast)  $ 
		with very high probability, then the event
		\[ 
		\left\{ |\widetilde{\Delta}(\eta)| \leq \mathcal C \widetilde{\psi}_w(\eta)  \text{ for all } \eta \in [\eta_*,\eta^*] 
		\right\}
		\] 
		holds with very high probability on the event $\{ \theta = 1 \text{ for all } \eta \in [\eta_*, \eta^*]\}$. 
	\end{lemma}
	
	\begin{proof}
		Averaging the inequality in \cref{lem:stielt_first_mean} over $x \in [2N]$ yields  
		\begin{equation*}
			g^3  + 2 \ii \eta g^2 +(1 - \eta^2 - \abs{w}^2)g+ \ii \eta  = \mathcal{O} \bigg(\sqrt{\frac{\log N}{d}}\bigg) 
		\end{equation*}
		uniformly for all $N^{-1+ \delta} \le \eta \le L$ and $\abs{w} \le 2$. The rest of the proof is analogous to the proof of \cref{lem:Deltasmall}, i.e apply \cref{lem:stability2} with the same functions $\widetilde{\xi}$, $\xi_1$ and $\xi_2$ and replacing $\varphi$ with $\sqrt{\frac{\log N}{d}}$.
	\end{proof}
	\begin{prop}\label{prop:phi_Lambda_small3new}
		We fix constants $L >0$ and $0 < \delta <1 $. Then there exists a constant $\mathcal{C} >0$ depending only on $L$ and $\delta$ such that for any $N^{-1+ \delta} \le \eta_{\ast} < \eta^{\ast} \le L$, if $\widetilde{\Delta}(\eta^*) \leq \mathcal C \widetilde{\psi}_w(\eta^*)$ with very high probability, then the event 
		\begin{equation*}
			\left\{  \widetilde{\Lambda}(\eta)\le \mathcal{C} \widetilde{\psi}_w(\eta)  ~~\text{for all}~ \eta \in [ \eta_{\ast}, \eta^{\ast}] \right\}
		\end{equation*}
		holds with very high probability on the event $\{ \theta =1 ~\text{for all} ~ \eta \in [\eta_{\ast}, \eta^{\ast}]\}$ 
		uniformly for $w \in \mathrm{D}_2$. 
	\end{prop}

	\begin{proof}
		We consider $\eta \in [\eta_{\ast},\eta^{\ast}]$ for $N^{-1+\delta} \le \eta_{\ast} < \eta^{\ast} \le L$ fixed. The estimates in this part of the proof will hold uniformly for $w \in \mathrm{S}$ and $\eta \in [\eta_*,\eta^*]$. For the sake of clarity, we avoid repeating this in most instances. Using then the same argumentation than in the proof \cref{prop:phi_Lambda_small2}  while using \cref{lem:stielt_first_mean} and \cref{lem:3rddegeqonDelta2} instead of \cref{lem:meanoftypicalvertices} and \cref{lem:Deltasmall}, respectively we have
		\begin{equation} \label{eq:proof_phi_Lambda_small2_aux_1} 
			\abs{g - \ii v } \leq \mathcal C \widetilde{\psi}
		\end{equation} 
		with very high probability on $\Omega_\theta$. We also needed to average the inequality in \cref{lem:stielt_first_mean} to obtain the previous inequality.
		For each $x \in [2N]$, we deduce from \cref{lem:stielt_first_mean} and \eqref{eq:proof_phi_Lambda_small2_aux_1} that \nc 
		\begin{equation}\label{eq:Giieqiv}
			\bigl|G_{xx} \bigl( v^2  +2 \eta v + \eta^2 + \abs{w}^2 \bigr) -  \ii (\eta + v) \bigr| \le  \mathcal{C} \widetilde{\psi}
		\end{equation}
		with very high probability on the event $\Omega_{\theta}$. We multiply \eqref{eq:Giieqiv} by $\frac{v}{\eta + v}$, use $v^2  +2 \eta v + \eta^2 + \abs{w}^2= \frac{\eta + v}{v}$ by \eqref{eq:v} as well as $\frac{v}{\eta + v} \leq 1$ and obtain 
		\begin{equation}\label{eq:proof_phi_Lambda_small2_aux_2} 
			|G_{xx} - \ii v | \le \mathcal C \widetilde{\psi}.
		\end{equation}
		This completes the estimate of $\tilde{\Lambda}_{\mathrm{d}}$.	To bound $ \Lambda_{\mathrm{do}}$, 
		let $x \in [N]$. We start from the first equation of \cref{lem:ASCEq}, use \eqref{eq:boundrest} and \eqref{eq:roughdifferenceGGTd-1} from \cref{lem:roughboundsrest}, \eqref{eq:all_vert_typi} from \cref{lem:sumXabGa_eq_sumGa}, the definition of $u=u_{(1,1)}$ from \eqref{eq:def_u_beta} and \eqref{eq:proof_phi_Lambda_small2_aux_2} to arrive at 
		\begin{equation}
			1 + w G_{\ubarr{x}x} = ( \eta +  v) v + \mathcal O ( \widetilde{\psi}) 
			=1 - \abs{w}^2 u + \mathcal O ( \widetilde{\psi}).\label{eq:Gxubarrx_eq_u}
		\end{equation}
		If $\abs{w} \geq 1/2$ then subtracting $1$ and dividing by $w$ in \eqref{eq:Gxubarrx_eq_u} 
		yields $\abs{G_{\ubarr{x}x} + \bar w u} \leq \mathcal C \widetilde{\psi}$. 
		If $\abs{w} \leq 1/2$ then we start from the third identity in \cref{lem:ASCEq} and reason similarly using $v \gtrsim 1$ by \eqref{eq:v_asymp_1} in that case  to conclude $\abs{G_{\ubarr{x}x} + \bar w u} \leq \mathcal C \widetilde{\psi}$. 
		An analogous reasoning for $G_{x \ubarr{x}}$ completes the estimate of $\widetilde{\Lambda}_{\mathrm{od}}$.
		Since ${\Lambda}_\mathrm{o}\leq \mathcal C d^{-1/2}$ by \eqref{eq:roughboundGd-1/2} in \cref{lem:roughboundsrest}, we have shown that 
		\begin{equation*}
			\Lambda \le \mathcal C \widetilde{\psi}. 
		\end{equation*}
		A grid argument as in the proof of \cref{prop:phi_Lambda_small2} completes the proof.
	\end{proof}
	
	\begin{proof}[Proof of \cref{thm:locallaw3}]
		 Given the previous results in this section, i.e.\ \cref{lem:sumXabGa_eq_sumGa}, \cref{lem:stielt_first_mean}, \cref{lem:3rddegeqonDelta2}
		and \cref{prop:phi_Lambda_small3new}, the proof of \cref{thm:locallaw3} 
		proceeds completely analogously to the one of 
		\cref{thm:locallaw_edge}. We therefore omit it. 
	\end{proof}

	\section{Eigenvalue locations of $M$} \label{sec:largest_eigenvalue}

 The next proposition is the main result of this section. Throughout we always count eigenvalues with their multiplicities. 

\begin{prop}\label{prop:spectrum_M} 
Let $0 < \epsilon < 1$ and $0 <\delta < \frac{1}{3}\epsilon$ be fixed. Let $M$ be as in \eqref{eq:def_M_X_plus_f} for $1 \le d \le N^{\delta/4}$ and $f \in \C$. Then there exists a  constant $C>0$ such that the following holds with probability at least $1- CN^{-\epsilon}$.
If $\abs{f} \geq 1 + 2 C d^{-1/2}$ then 
\begin{enumerate}[label=(\alph*)]
    \item \label{item:other_eigenvalues_disc} All eigenvalues of $M$ except one are contained in the disc of radius $  1 + C d^{-1/2} $.
    \item \label{item:large_eigenvalue} There is exactly one eigenvalue within the disc of radius $N^{-1/2 + 2\epsilon}|f|$ around $f$.
\end{enumerate}
If $\abs{f} \leq 1 + 2 C d^{-1/2}$ then 
\begin{enumerate}[label=(\alph*)]
\setcounter{enumi}{2} 
\item All eigenvalues of $M$ are contained in the disc of radius $  1 + 3C d^{-1/2} $.
\end{enumerate} 
\end{prop}

If $d \asymp N$ and $f \asymp \sqrt{N}$ then a version of \cref{prop:spectrum_M} was shown in  \cite[Theorem~1.9]{zbMATH06141055}. 

In the special case of the adjacency matrix of the Erd{\H o}s-R\'enyi digraph, we now list the 
results similar to \cref{prop:spectrum_M} obtained previously.  
A similar result 
was stated in \cite[Corollary~3.5]{BBK_spec_radii} with a stronger probability estimate. 
For the proof, the authors of \cite{BBK_spec_radii} referred to the version of the Bauer-Fike theorem in 
\cite{BLL2018Nonbacktracking}, see  \cite[Proposition~7]{BLL2018Nonbacktracking}. 
However, it is unclear how to justify the assumptions of \cite[Proposition~7]{BLL2018Nonbacktracking} in this setting. 
If $d \geq C \log N$, \cite[Corollary~1]{MR4206685} proved \cref{prop:spectrum_M} \cref{item:large_eigenvalue}. In \cite[Theorem~4.3]{MR4649432}, \cref{item:other_eigenvalues_disc} and \cref{item:large_eigenvalue} were shown for the Erd{\H o}s-R\'enyi digraph. 
If $d \geq N^\delta$, the essentially optimal rate $N^{-1/2+\eps}$ in \cref{prop:spectrum_M} \cref{item:other_eigenvalues_disc}
was proved in \cite[Theorem~1.2(i)]{HeDigraph} and \cref{item:large_eigenvalue} 
with a radius $O(1)$ in \cite[Lemma~4.5]{HeDigraph}.  

For the proof of \cref{prop:spectrum_M}, we use an eigenvalue criterion for $M$, 
\cref{lem:eigenvalue_criterion}, in terms of the resolvent of $X$ and analyse the resolvent through a Neumann series expansion in the proof of \cref{prop:spectrum_M} below. 
 This criterion was for example also applied \nc in \cite[Section~4.1]{BGN11} and \cite[ Sections~2, 3 and 5]{zbMATH06141055}. Moreover, \cite{zbMATH06141055} employs the Neumann series in simpler setups.

\begin{lemma}[Eigenvalue criterion for $M$] \label{lem:eigenvalue_criterion}
Let $X \in \C^{N\times N}$,  $f \in \C$ and $z \in \C \setminus \spec(X)$. Then 
\[ z \in \spec ( X + f \mathbf e \mathbf e^*) \qquad \iff \qquad 1 + f\langle \mathbf e, (X-z)^{-1} \mathbf e \rangle =0. 
\] 
\end{lemma} 

 The previous lemma coincides with \cite[Lemma 2.1]{zbMATH06141055}.
We present its short proof for the reader's convenience.

\begin{proof} 
Since $z \notin \spec(X)$, we observe that $z \in \spec(X + f \mathbf e \mathbf e^*)$ is equivalent to  
\[ 
0 = \det(X + f \mathbf e \mathbf e^*- z) = \det\big( (X-z) \big( 1 + (X-z)^{-1} f \mathbf e \mathbf e^*\big)\big) \iff \det\big( 1 + (X-z)^{-1} f \mathbf e \mathbf e^*\big) = 0.
\] 
This is equivalent to $1 + f\langle \mathbf e, (X-z)^{-1} \mathbf e \rangle =0$ due to 
$\det(1 + AB) = \det(1 + BA)$ and $(\mathbf e \mathbf e^*)^2 =\mathbf e \mathbf e^*$. 
\end{proof} 
 Owing to the previous lemma, the eigenvalues of $M$ outside of $\spec(X)$ coincide exactly with the zeros  of 
\begin{equation} \label{eq:def_spectral_function1}  
F(z) := 1 + f \langle \mathbf e , (X-z)^{-1} \mathbf e \rangle 
\end{equation}
counted with their multiplicities. 
We define $\widehat{F}(z) := 1-\frac{f}{z}$ and formally rewrite 
\begin{equation}\label{eq:def_spectral_funct}
    F(z) = \widehat{F}(z) - \frac{f}{z}\sum_{m=1}^{\infty}\frac{\langle \mathbf{e}, X^m \mathbf{e} \rangle }{z^m }.
\end{equation}
In order to compare the zeros of $F$ and $\widehat{F}$ we need the following two results on the powers of $X$. In the proof of \cref{prop:upper_bdd_spec_radius_X}, we will justify the absolute convergence of the series. The following lemma is a direct consequence of \cite[Proposition 6.1]{BBK_spec_radii}.
\begin{lemma} \label{lem:norm_X_m} 
We consider $1 \le d \le N^{\delta/4}$ with $0<\delta < 1/3$. There exists a constant $C>0$ such that if $m \in \N$ satisfies $m \le  C^{-1} \sqrt{d} \log N$ and $\nu>0$ then 
    \begin{align*}
    \|X^m \| &\le C N^{ (\nu+1)/2} m^2 \sqrt{d}
    \end{align*}
    with probability at least $1 - N^{- \nu}$.
    \end{lemma} 

\begin{proof}
    For all positive integer $m$, Markov's inequality implies 
    \begin{align*}
        \mathbb{P} \left( \|X^m\| \ge \Phi \right) \le \mathbb{P} \left( \Tr(X^m {X^\ast}^m ) \ge \Phi^2  \right)   \le \frac{\mathbb{E}\left(\Tr(X^m {X^\ast}^m ) \right)}{\Phi^2}.
    \end{align*}
    It remains to apply \cite[Proposition 6.1]{BBK_spec_radii}
    observing that \cite[(5.1)]{BBK_spec_radii} is satisfied due to 
    the upper bounds on $d$ and $m$.
\end{proof}

The next proposition is a simple consequence of \cite[Theorem~2.11]{BBK_spec_radii}. 
We deduce it here from \cref{lem:norm_X_m} since we will use a similar argument in the proof of \cref{prop:spectrum_M}. 

\begin{prop}[Upper bound on spectral radius of $X$] \label{prop:upper_bdd_spec_radius_X} 
Let $1\le d \le N^{\delta/4}$ with $0<\delta <1/3$ and $X$ be as in \cref{defiX}. There is a constant $\mathcal C>0$ such that 
\[ \max_{\lambda \in \spec(X)} \abs{\lambda} \leq 1 + \mathcal C d^{-1/2} \] 
with very high probability. 
\end{prop}

\begin{proof} 
Let $\nu>0$. We now show that the Neumann series of $X$ is absolutely convergent  with probability at least $1-N^{-\nu}$ if the spectral parameter $z$ satisfies $\abs{z} \geq 1 + \mathcal C d^{-1/2}$. 
Let $m_0 := C^{-1} \sqrt{d} \log N$ with $C>0$ from \cref{lem:norm_X_m}. 
Then \cref{lem:norm_X_m} implies that the event 
\[ \Omega := \bigl\{ \|{X^m}\| \leq C N^{\nu/2+1}m_0^2 \sqrt{d} \text{ for all } m = 1, \ldots, m_0 \bigr\} \] 
holds with probability at least $1-N^{-\nu}$. Therefore, simultaneously for all $z \in \C$ satisfying $\abs{z} \geq 1 + \mathcal C d^{-1/2}$, we have  
\begin{align} 
\sum_{m=m_0+1}^\infty \frac{\norm{X^m}}{\abs{z}^m} & \leq \sum_{k=1}^\infty \sum_{r=0}^{m_0-1} \bigg( \frac{\norm{X^{m_0}}}{(1 + \mathcal C d^{-1/2})^{m_0}} \bigg)^k \bigg( \frac{\norm{X^r}}{( 1+ \mathcal C d^{-1/2})^{r}}\bigg) \nonumber \\ 
& \leq \frac{2 C N^{\nu/2 +1} m_0^2 \sqrt{d}}{(1 + \mathcal C d^{-1/2})^{m_0}} \sum_{r=0}^{m_0-1} \frac{CN^{\nu/2 + 1} m_0^2 \sqrt{d}}{(1 + \mathcal C d^{-1/2})^{r}} \nonumber \\ 
& \leq N^{\nu + 3 - C^{-1} \sqrt{d} \log (1 + \mathcal C d^{-1/2})} \label{eq:bound_Neumann_series} 
\end{align} 
on the event $\Omega$.
Here, we estimated $\norm{X^{km_0 + r}} \leq \norm{X^{m_0}}^k \norm{X^r}$ in the first step 
and used the definition of $\Omega$ and $C N^{\nu/2 + 1} m_0^2 \sqrt{d} (1 + \mathcal Cd^{-1/2})^{m_0} \leq 1/2$ for large enough $\mathcal C$ in the second step. 

Therefore, $\sum_{m=0}^\infty \norm{X^m} \abs{z}^{-m}$ is finite simultaneously for all $z \in\C$ with $\abs{z} \geq 1 + \mathcal C d^{-1/2}$ on $\Omega$. Thus, $\spec(X) \subset D_{1 + \mathcal C d^{-1/2}}$ on $\Omega$. 
Since $\Omega$ holds with probability at least $1-N^{-\nu}$ and $\nu$ was arbitrary, the claim follows. 
\end{proof} 

\nc 

\begin{prop}[Upper bound on series summands] \label{prop:bounds_from_radii_meth}
Let $ 0 < \epsilon <1$ and $0 < \delta < \frac{1}{4} \epsilon$ be fixed. 
If $d$ satisfies 
\begin{equation}\label{eq:regimed3}
   1 \le d \le N^{\delta/4}
\end{equation}
then there exist universal constants $C$, $c>0$ such that
    \begin{equation*}
       \abs{ \langle X^{m}\mathbf{e}, \mathbf{e}\rangle } \le N^{-1/2 + \epsilon} , ~~ \text{ for all}~ m \le c \sqrt{d}\log N
    \end{equation*}
    with probability at least $1 - CN^{-\epsilon}$.
\end{prop}
The proof of \cref{prop:bounds_from_radii_meth} in \cref{section:proofbounds_from_radii_meth} below uses similar arguments as \cite[Sections 5.1 and 6]{BBK_spec_radii}. By estimating higher moments in the proof of \cref{prop:bounds_from_radii_meth}, which result in graphs with more than 2 connected components,  
we expect that the probability estimate on the above bound on $\langle X^{m}\mathbf{e}, \mathbf{e}\rangle$ 
can be substantially improved. 

Next, we use \cref{prop:bounds_from_radii_meth} to prove \cref{prop:spectrum_M}.  
\begin{proof}[Proof of \cref{prop:spectrum_M}]
    We fix $\delta >0$ and set 
    \begin{equation*}
        m_0:= c \sqrt{d} \log N
    \end{equation*}
    with a small constant $c>0$ and $0<\epsilon<1$. Throughout the proof, we work on the event $\{\spec(X) \subset \mathrm{D}_{1+C d^{-1/2}}\}$ up to considering $C> \mathcal{C}_{\nu} $ with $\nu = 2 \epsilon$ which, owing to  \cref{prop:upper_bdd_spec_radius_X}, holds with probability at least $1 - CN^{-\epsilon}$. We rewrite $F$ from \cref{eq:def_spectral_function1} as
    \begin{align}\label{eq:bound_f-g1}
        F(z) = \widehat{F}(z) - \frac{f}{z} A_{m_0}(z) - \frac{f}{z} B_{m_0}(z), 
    \end{align}
    where we used \eqref{eq:def_spectral_funct} and defined
    \begin{align*}
        A_{m_0} (z) := \sum_{m=1}^{m_0} \frac{\langle X^m \mathbf e, \mathbf e\rangle}{z^m} \qquad \text{and} \qquad B_{m_0} (z) := \sum_{m={m_0}+1}^\infty \frac{\langle X^m \mathbf e, \mathbf e\rangle}{z^m}.
    \end{align*}
   From \cref{prop:bounds_from_radii_meth} we directly have $C>0$ a constant such that $|A_{m_0}(z)| <  m_{0}N^{-1/2+\epsilon}$ uniformly in $\abs{z} \ge 1$ with probability at least $1 - CN^{-\epsilon}$. On the other hand, owing to \cref{lem:norm_X_m} and a union bound, considering $\nu = \delta/4 + \epsilon$ we obtain that with probability at least $1- CN^{-\epsilon}$, $\|X^r\| \le CNm_0^2\sqrt{d}$ for all $0 \le r \le m_{0}$. As a consequence we have for all $\abs{z} \ge 1+ C d^{-1/2} $
   \begin{align*}
       \abs{B_{m_0}(z)} &\le \sum_{m=m_0+1}^\infty \frac{\norm{X^m}}{(1+ C d^{-1/2} )^m} \\ 
       & \le \sum_{k=1}^\infty \sum_{r=0}^{m_0-1} \left(\frac{\norm{X^{m_0}}}{{(1+ C d^{-1/2} )^{m_0}}}\right)^k\left(\frac{\norm{X^r}}{(1+ C d^{-1/2} )^r}\right)\\
       &\le 2\frac{CN{m_0}^2\sqrt{d}}{(1+\delta)^{m_0}}\sum_{r=0}^{{m_0}-1}\frac{CNm_0^2\sqrt{d}}{(1+ C d^{-1/2} )^r} \\ 
       & \le  2\frac{CN^2{m_0}^5d}{(1+ C d^{-1/2} )^{m_0}}  \le 2CN^{ 3-  c \sqrt{d}  \log (1+ C d^{-1/2} )} \le \frac{1}{N}
   \end{align*}
where the last inequality holds for all $1 \le d \le N^{\delta/4}$, up to considering $ C >0$ large enough. We also apply \cref{lem:norm_X_m} together with a union bound in the third step and our choice of $m_0$. Inserting the previous bounds on $A_{m_0}$ and $B_{m_0}$ into \eqref{eq:bound_f-g1} and considering $N$ large enough we obtain that for all $  \abs{z} \ge 1+ C d^{-1/2}$ and with probability at least $1-N^{-\epsilon}$
   \begin{align}\label{eq:bound_F_minus_hat_F} 
       \absN{F(z) - \widehat{F}(z)} < N^{-1/2 + 2\epsilon}\abs{\frac{f}{z}}.
   \end{align}
   We directly conclude from  
 \cref{eq:bound_F_minus_hat_F} that $F$ has no zero in    $\C \setminus ( \mathrm{D}_{1+ C d^{-1/2}} \cup \mathrm{D}_{N^{-1/2 + 2\epsilon}|f|}(f))$. First, if  $\abs{f}\le 1+ 2 Cd^{-1/2}$ then  \cref{lem:eigenvalue_criterion} implies that all eigenvalues of $M$ are contained in the disc of radius $1+ 3 Cd^{-1/2} > 1+ 2Cd^{-1/2} + N^{-1/2+ 2 \epsilon}$. 
 Second, if $\abs{f}\ge 1+ 2 Cd^{-1/2}$ then for all $z \in \CC$ such that $\abs{z-f} = N^{-1/2 + 2\epsilon}|f|$
 we obtain from \cref{eq:bound_F_minus_hat_F} that 
    \begin{align*}
       \absN{F(z) - \widehat{F}(z)} <\frac{N^{-1/2 + 2\epsilon}|f|}{\abs{z}} = \absN{\widehat{F}(z)}.
   \end{align*}
   Rouch{\'e}'s Theorem implies that $\widehat{F}$ and $F$ have exactly the same number of zeros in the disc of radius $N^{-1/2 + 2\epsilon}|f|$ and center $f$, that is to say exactly one. 
    By \cref{lem:eigenvalue_criterion}, we conclude that $M$ has 
   exactly one eigenvalue outside $\mathrm{D}_{1 + Cd ^{-1/2}}$. This eigenvalues lies in $\mathrm{D}_{ N^{-1/2 +2\eps}\abs{f}}(f)$.  
\end{proof}

\subsection{Proof of \cref{prop:bounds_from_radii_meth}}\label{section:proofbounds_from_radii_meth}

For all positive integers $m$, Markov's inequality yields 
\begin{equation}\label{eq:bound_pb_Am2}
  \mathbb{P} \bigl( |\langle \mathbf e, X^m \mathbf e \rangle| \ge N^{-1/2+ \epsilon} \bigr) =  \mathbb{P} \bigl( |\langle \mathbf e, X^m \mathbf e \rangle|^2 \ge N^{-1 + 2 \epsilon}\bigr) \le N^{1- 2\epsilon}{\mathbb{E} |\langle \mathbf e, X^m \mathbf e \rangle|^2}
  = N^{-1 -2 \epsilon} \mathrm{B}_m, 
\end{equation}
where we introduced $\mathrm{B}_{m}:= N^2\mathbb{E}[\big|\langle X^m \mathbf e, \mathbf e \rangle \big|^2]$. We compute 
    \begin{align*}
\mathrm{B}_{m} &= \mathbb{E} \sum_{\patx\in \boldsymbol C} X_{x_0^1 x_1^1} \cdots X_{x_{m-1}^{1} x_{m}^{1}} \overline{X}_{x_{0}^{2} x_{1}^{2}} \cdots\overline{X}_{x_{m-1}^2 x_m^2}, 
    \end{align*}
where $\boldsymbol{C}$ is the set of all $\patx = ((x^s_0,...,x^s_m)_{s\in \{1,2\}}) \in ([N]^{m+1})^2$ such that for all $(a,b) \in [N]^2$
    $$\sum_{s=1}^2\sum_{k=1}^m \mathds{1}\left( (x^s_{k-1} , x^s_{k}) = (a,b) \right) \neq 1.$$
   We introduce an equivalence relation on $\boldsymbol C$ by saying that two paths $\patx=(x_0^1,\ldots, x_m^1,x_0^2, \ldots x_m^2)$ and $\boldsymbol y=(y_0^1,\ldots, y_m^1,y_0^2, \ldots y_m^2)$ are equivalent, written $\patx \sim \boldsymbol y$,  if there exists a permutation $\tau$ of $[N]$ such that $\tau(x_k^s) = {y}^s_k$ for all $k \in \{0,...,m\}$ and $s \in \{1,2\}$. We denote by $[\patx] \subset \boldsymbol C$ the equivalence class of $\patx \in \boldsymbol C$. Finally for each path $\patx \in \boldsymbol C$ we set the directed graph $G(\patx)$ defined by
\begin{equation*}
    V(\patx) = \{x_k^s,~k\in \{0,...,m \}, ~ s \in  \{1,2\}\}, ~~ E(\patx):= \{ (x_{k-1}^s, x_k^s),~k \in [m],~ s \in  \{1,2\}\}.
\end{equation*}
 For $\paty \in \boldsymbol C$ we denote by $g(\paty)$ the  sum of the genera of the connected components of $G(\patx)$. To compute the genus of a connected component, we consider it as undirected graph. Recall that for a finite, undirected and connected graph $G=(V,E)$, the genus is given by $g(G):= |E|-|V|+1$. 
We denote by $r(\patx)$ the number of connected components of $G(\patx)$. Note that $ r(\patx) \in \{1 , 2\}$. Since $g$ and $r(\cdot)$ are constant on the equivalence class, for every  path $\patx \in \boldsymbol C$ by translating \cite[proof of Lemma~5.3]{BBK_spec_radii}, we obtain
\begin{align}\label{eq:estimate_x_fixed}
    \absbb{\mathbb{E} \sum_{\paty  \in [ \patx]} X_{y^1_0 y^1_{1}} \cdots  \overline{X}_{y^2_{m-1} y^2_{m}}}\le N^{r(\patx) - g(\patx)} d^{|E(\patx)|-m}.
\end{align}
For $\patx \in \boldsymbol C$ we set
\begin{equation*}
    \mathcal{I}_2(\patx) := \{ x \in V(\patx) \colon   \mathrm{d}^+ (x)= 1 = \mathrm{d}^-(x)\} \backslash \{x_{0}^s, x_{m}^s \colon s\in  \{1,2\}\}. 
\end{equation*}
We define the directed  (multi-)graph $\hat{G}(\patx)=(\hat{V}(\patx),\hat{E}(\patx))$ with vertex set $\hat{V}(\patx) := V(\patx) \backslash \mathcal{I}_{2}(\patx)$ and the edge set $\hat{E}(\patx)$ defined as follows. 
If $a$ and $b$ are in $\hat{V}(\patx)$ and $e=(a,b) \in E(\patx)$ then $e=(a,b) \in \hat{E}(\patx)$ with weight $k_{e}=1$. 
If $a$ and $b$ are in $\hat{V}(\patx)$ and they are connected by a path in $G(\patx)$ with start point $a$ and end point $b$, which only crosses vertices 
in $\mathcal{I}_2(\patx)$ then $e=(a,b) \in \hat{E}(\patx)$. For $e\in \hat{E}(\patx)$ we define its weight $k_{e}$ equals to the number of crossed vertices in $\mathcal{I}_2(\patx)$.
Finally we denote by $\vec{k} = (k_{e})_{e\in \hat{E}(\patx)}$ the vector of weights.

We now introduce the definition of a normal graph, analogous to \cite[Definition 5.4]{BBK_spec_radii}.
\begin{defi}\label{def:normal_graph}
A pair $(U,\patchi)$ consisting of a graph $U$ and a path $\patchi=(\chi_1,e_1,\chi_2,e_2,...,e_{n-1},\chi_n)$ with $\chi_1,...,\chi_n \in V(U)$ and $e_1,...,e_{n-1} \in E(U)$ is \textup{normal} if the following holds.
\begin{enumerate}[label=(\roman*)]
    \item $V(U) = \{\chi_1,...,\chi_n\} = [s] $ where $s = |V(U)|$.
    \item The vertices of $U$ are visited by $\patchi $ in increasing order, i.e if $\chi_{i} \notin \{\chi_1,....,\chi_{i-1}\}$ then $\chi_i > \max_{j \le i-1} \chi_j$.
\end{enumerate}
 We set  $|\patchi| := n-1$ and call it the \textup{length of $\patchi$}. 
 For each $e \in E(U)$, we set $m_{e}(\patchi) := |\{ i \in [n-1] \colon e_i=e\}|$ the \textup{multiplicity} of $e$ in $\patchi$, i.e.\ the number of times the path $\patchi$ 
 traverses the edge $e$.
\end{defi}

A priori, the graph $(\hat{G}(\patx), \patx)$ for $\patx \in \boldsymbol C$ is not normal. Up to relabelling again the vertices of $\hat{G}(\patx)$ with the order of appearance in $\patx$ of each vertex, we obtain a unique $(U,\patchi) $ normal. Some of its properties are listed in the following lemma, which follows directly from the definition of a normal graph. 
\begin{lemma}\label{lem:eq_of_lem5.7}
The mapping $\mathcal G :\patx \mapsto (U, \patchi,\vec{k}) $ satisfies the following properties. 
\begin{enumerate}[label=(\alph*)] 
    \item It is injective on the set of equivalence classes $\boldsymbol C/\sim$.
    \item $g(U) = g(\patx)$, where $g(U)$ is the sum of the connected component genera of $U$.
    \item $\abs{E(\patx)} = \sum_{e\in E(U)} k_{e}$.
    \item $2m= \sum_{e \in E(U)}m_e(\patchi)k_e$ and $m_{e}(\patchi) \ge 2$ for all $e \in E(U)$.
\end{enumerate}

\end{lemma}
We sum over all equivalence classes and obtain from the previous lemma and \eqref{eq:estimate_x_fixed} that 
\begin{align*}
    \mathrm{B}_{m} &\le \sum_{(U,\patchi,\vec{k})} N^{r(U)-g(U)} d^{ \sum_{e \in E(U)} k_e - \frac{1}{2}\sum_{e\in E(U)} m_{e}(\patchi)k_e} \\
    &= \sum_{(U,\patchi,\vec{k})} N^{r(U)-g(U)} d^{- \frac{1}{2}\sum_{e \in E(U)} k_e(  m_{e}(\patchi)-2)}.
\end{align*}
	We use that for all $e \in E(U)$ we have $k_e \ge 1$ and $m_{e}(\patchi) \ge 2$ to obtain that the exponent of  $d^{-1} $ above is lower bounded by $\frac{1}{2}|\patchi| - |E(U)|$ and therefore
	\begin{equation*}
	        \mathrm{B}_{m} \le \sum_{(U,\patchi,\vec{k})} N^{r(U)-g(U)} d^{|E(U)| - \frac{1}{2}|\patchi|} .
	\end{equation*}

	The summands in the previous sum do not depend on $\vec{k}$ and for fixed $(U,\patchi)$, each possible $\vec{k}$ satisfies $\sum_{e\in E(U)} m_{e}(\patchi) k_{e} = 2m$ with $m_e(\patchi) \ge 2$ for all $e \in E(U)$. Therefore, the number of possible $\vec{k}$ is bounded by the number of $(k_{e})_{e \in E(U)}$ satisfying $k_e \ge 1$ for all $e \in E(U)$ and $\sum_{e}k_{e} \le 2m$, which is given by
	\begin{equation*}
	     \sum^{2m}_{m^\prime= \abs{E(U)}} |\{ (k_e)_{e\in E(U)}\in (\mathbb{N}\setminus{\{0\}})^{E(U)} \colon \sum_e k_e =m^\prime\} |=\nc\sum^{2m}_{m^\prime = \abs{E(U)}} \binom{m^\prime-1}{\abs{E(U)}-1} \le 2m \left( \frac{6m}{\abs{E(U)}}\right)^{\abs{E(U)}}.
	\end{equation*}
		We obtain
	\begin{equation*}
	    	      \mathrm{B}_{m} \le \mathrm{B}_{m}^{(1)} + \mathrm{B}_{m}^{(2)}, 
	\end{equation*}
	where, for $r \in \{1, 2\}$, we introduced 
	\begin{align*}
	    	  \mathrm{B}_{m}^{(r)} &= 2m\sum_{\substack{(U,\patchi)\\ r(U)=r}}  N^{r(U)-g(U)} \left(\frac{6m}{|E(U)|}\right)^{|E(U)|}d^{|E(U)| - \frac{1}{2}|\patchi|}.
	\end{align*}
	We first show an upper bound on $\mathrm{B}_{m}^{(2)}$. 
	
	\begin{lemma}\label{lem:bounds_case_ru=r}
	    For all $\patx$ such that $\mathcal G (\patx)= (U,\patchi, \vec{k})$ satisfies $r(U)=2$, every vertex has degree at least $3$ except possibly the images of $\{x_{0}^s,x_m^{s} \colon s\in \{1,2\}\}$ that have degree at least $2$. Furthermore,  
\begin{align*}
    \abs{E(U)} \le 3g(U), ~~ \abs{V(U)} \le 3g(U),~~ \text{and} ~~ g(U)\ge 2.
\end{align*}	

\end{lemma}
\begin{proof}
In that case it is important to notice that each subpath $(x_{0}^s,...,x_{m}^s)$ for $s = 1$ or $s = 2$ determines one connected component. Thus, the statement about the degrees of the vertices follows immediately. 
For each connected component $\cc \subset U$ we have
\begin{align}\label{eq:boundVcc1}
    2|E_{\cc}| = \sum_{v \in V_{\cc}} \mathrm{d}(v) \ge 4 + 3 (|V_{\cc}| -2 ) = 3|V_{\cc}| -2
\end{align}
and therefore
\begin{align*}
    g_{\cc} = \abs{E_{\cc}} - \abs{V_{\cc}} + 1 \ge \frac{1}{3} \abs{E_{\cc}}.
\end{align*}
Implementing the previous inequality in \eqref{eq:boundVcc1} and using that in this case we always have $g_{\cc} \ge 1$ we obtain the wanted upper bound for $V_{\cc}$. It remains to sum over all connected components.
\end{proof}

Since $g(U) \leq \abs{E(U)} \leq 3 g(U)$ by \cref{lem:bounds_case_ru=r}, we obtain
\begin{equation*}
    \mathrm{B}_{m}^{(2)}\leq 2m N^2\sum_{\substack{(U,\patchi)\\ r(U)=2}} N^{-g(U)} \left(\frac{6m}{g(U)}\right)^{3g(U)}d^{3g(U) - \frac{1}{2}|\patchi|}. 
\end{equation*}

For $e,\ell$ and $v$ fixed integers, the number of $(U,\patchi)$ such that $|E(U)| \le e$, $|\patchi| \le \ell$ and $|V(U)| \le v$ is bounded by $e^\ell v^e$, which can be proved as explained after 
\cite[proof of Lemma~5.8]{BBK_spec_radii}.  Therefore the number of $(U,\patchi)$ such that $g(U)=g$ and $|\patchi|= \ell$ is bounded by $(3g)^\ell (3g)^{3g}$ due to \cref{lem:bounds_case_ru=r}. 
Since $2 \leq \nc g=g(U) \leq \abs{E(U)} \leq 2m$ by \cref{lem:eq_of_lem5.7,lem:bounds_case_ru=r} and $\abs{\patchi} \leq \abs{V(\patx)} \leq 2m$, we obtain 
\begin{equation*}
       \mathrm{B}_{m}^{(2)} \leq 2m N^2\sum_{g=2}^{2m}\sum_{\ell=1}^{2m}\left(\frac{3g}{\sqrt{d}}\right)^\ell \left(\frac{C m^6d^3}{N}\right)^g 
\end{equation*}
 for a universal constant $C>0$. Hence, the bound $\sum_{\ell=1}^{2m} x^\ell \le 2m(1+x^{2m}) $ for all $x>0$ yields 
\begin{equation} \label{eq:proof_sums} 
    \mathrm{B}_{m}^{(2)} \le 4m^2 N^{2} \biggl(   \sum_{g=2}^{2m}  \biggl(\frac{C m^6d^3}{N}\biggr)^g + \sum_{g=2}^{2m}\biggl(\frac{3g}{\sqrt{d}}\biggr)^{2m} \biggl(\frac{C m^6d^3}{N}\biggr)^g \biggr).
\end{equation}
The conditions $\delta < 1/3$ and $1 \le d \le N^{\delta /4}$ imply that, for a small enough constant $c$ in the upper bound on $m$ in \cref{prop:bounds_from_radii_meth},  
\begin{equation}\label{eq:sum1}
    N^{2/3} \ge2 Cm^6d^3.
\end{equation}
Besides in the second sum on the right-hand side of \cref{eq:proof_sums}, the maximal summand is obtained for 
$$g_{\max}:= \frac{2m}{\log \left(\frac{N}{C m^6d^3}\right)}.$$
Hence, \cref{eq:sum1} and possibly shrinking $c$ in the upper bound on $m$ imply 
\begin{equation}\label{eq:max_sum2}
    \frac{3g_{\max}}{\sqrt{d}} \le \frac{1}{2}.
\end{equation}
Finally, we apply \cref{eq:sum1} and \cref{eq:max_sum2} in \cref{eq:proof_sums} and arrive at   
\begin{align}\label{eq:up_bound_for_B^1}
    \mathrm{B}_{m}^{(2)} \le C m^{14} d^{6}.
\end{align}

To bound $\mathrm{B}^{(1)}_{m}$ from above, we use the next lemma. 
\begin{lemma}\label{lem:bound_case_ru<r}
    Let $\patx \in \boldsymbol C$. For $(U, \patchi, \vec{k})=\mathcal G(\patx)$, the following holds. 
    \begin{enumerate}
      \item  
      There are at most four vertices in $V(U)$  with degree one or two and all other vertices have degree at least three. 
      \nc 
        \item $|E(U)| \le 3g(U) + 5$ and $|V(U)| \le 2g(U) + 6$.
    \end{enumerate}
\end{lemma}

\begin{proof}
    The first point is a consequence of the construction of $(U, \patchi, \vec{k})$. 
    For the second point, we conclude from the first one that 
    \begin{align*}
        2|E(U)| = \sum_{v \in V(U)} \mathrm{d}(v) 
        \ge 4 + 3(|V(U)| - 4) = 3 |V(U)| - 8. 
    \end{align*}
    Hence, 
    \begin{align*}
        g(U) = |E(U)|- |V(U)| + r(U) \ge \frac{1}{3} |E(U)| - \frac{5}{3}.
    \end{align*}
    Note that $r(U) \in \{1 ,2 \}$. 
\end{proof}
We use $|E(U)| \ge \frac{1}{2}(g(U)+1)$  and \cref{lem:bound_case_ru<r} to obtain 
\begin{align*}
    \mathrm{B}_{m}^{(1)} &\le \sum_{\substack{(U,\patchi)\\ r(U)=1}} N^{r(U)-g(U)} \left(\frac{12m}{g(U)+1}\right)^{3g(U)+5}d^{3g(U) - \frac{1}{2}|\patchi| + 5}.
\end{align*}
 We argue analogously to the estimate of $\mathrm{B}_{m}^{(2)}$, to obtain that there are at most $(3g+ 5)^\ell (2g + 6)^{3g+5}$ possible $(U,\patchi)$ satisfying $g(U)=g$ and $|\patchi|=\ell$. Therefore, we proceed analogously to our bounds on $\mathrm{B}_m^{(2)}$, i.e.\ use  $g(U) \leq \abs{E(U)} \leq 2m$ by \cref{lem:eq_of_lem5.7},
to find a universal constant $C>0$ such that 
\begin{align*}
        \mathrm{B}_{m}^{(1)} &\le N (Cm^2d)^{5}\sum_{g=0}^{2m} \sum_{\ell=1}^{2m}\left(\frac{3g+5}{\sqrt{d}}\right)^\ell \left(\frac{C m^6d^3}{N}\right)^g\\
        &\le  N(Cm^2d)^{5} 2m \biggl(  \sum_{g=0}^{2m}  \left(\frac{Cm^6d^3}{N}\right)^g + \sum_{g=0}^{2m}\left(\frac{3g + 5}{\sqrt{d}}\right)^{2m} \left(\frac{C m^6d^3}{N}\right)^g \biggr)\\
        &\le N (Cm^2d)^{6} 
\end{align*}
where we used that the maximum of the last summand above is reached for the same value $g_{\max} $ as in the estimate of $\mathrm{B}_m^{(2)}$, which satisfies $\frac{3g_{\max}+ 5}{\sqrt{d}} \le \frac{1}{2}$ by \cref{eq:sum1} and the upper bound on $m$.  Up to considering $C>0$ large enough, we have $\mathrm{B}_{m}^{(2)} \le N(Cm ^2d)^6$ and therefore, owing to the conditions on $m$ and $d$ from \cref{prop:bounds_from_radii_meth}, 
\cref{eq:bound_pb_Am2} implies
\begin{align*}
    \mathbb{P} \bigl( |\langle \mathbf e, X^m \mathbf e \rangle| \ge N^{-1/2+ \epsilon} \bigr)\le N^{-1-2\epsilon} {\mathrm{B}_{m}}\le 2N^{-2\epsilon} {(Cm^2 d)^6}  \le C^{\prime}  N^{4\delta - 2 \epsilon} \le C^\prime N^{-\epsilon}.
\end{align*}
This completes the proof of \cref{prop:bounds_from_radii_meth}. 

	\section{Proofs of main results} 
	
	In this section, we prove the remaining main results from \cref{sec:main_results}, i.e.\ 
	\cref{thm:condideloc} and \cref{cor:delocalization_unconditional}. 
	In fact, they are all consequences of the local law for the Hermitization $H(w)$.

	\subsection{Eigenvector Delocalization -- Proof of \cref{thm:condideloc}} \label{sec:proof_delocalization} 

		\begin{proof}[Proof of \cref{thm:condideloc}]
   
		We first prove the delocalization of eigenvectors assertions. For $\# \in \{\mathrm{l},\mathrm{b},\mathrm{e}\}$ and $\eta_{\delta}= N^{-1+ \delta/2}$ with very high probability on the event $\{\chi_{\#}=1\}$, 
		\begin{equation}\label{eq:unifbound}
			\sup_{w \in \mathrm{S}_{\#}} \max_{x \in [N]}\abs{G_{\ubarr{x}\ubarr{x}}(w, \ii \eta_\delta) - \ii v_{\ubarr{\beta}_{x}}(w, \ii \eta_{\delta})} \le \mathcal{C}.
		\end{equation}
	 Indeed, if we suppose that \cref{eq:unifbound} holds true with very high probability then we can conclude the proof of the theorem as follows. Note that if $\mathbf u$ is an eigenvector of $X$ with eigenvalue $w$ then $(0,\ldots, 0,\mathbf u) \in \C^{2N}$ lies in the kernel of $H(w)$. 
		Let $\# \in \{\mathrm{l},\mathrm{b},\mathrm{e}\}$ and $\mathbf{u}$ be a normalized eigenvector of $X$ with eigenvalue $w \in \mathrm{S}_{\#}$. Since $\eta $ is fixed we omit it in the notations. We consider the normalized eigenvector of $H(w)$ with eigenvalue $0$ and defined by $\mathbf{v} = (0,...,0, \mathbf{u}) \in \CC^{2N}$. We complete with an orthonormalized basis of eigenvectors $(\mathbf{v}_{i})_{i \in [2N-1]}$ of $H(w)$ and we have
		\begin{equation*}
			\mathcal{C} + v_{\ubarr{\beta}_x}(w)\ge \Ima G_{\ubarr{x}\ubarr{x}} (w)  = \sum_{i=1}^{2N} \eta_{\delta} \frac{\abs{\langle e_{\ubarr{x}},\mathbf{v}_i\rangle}^2}{\lambda_{i} (H(w))^2 + \eta_{\delta}^2} \ge \frac{\abs{\langle e_{x},\mathbf{u} \rangle}^2}{\eta_{\delta}}. 
		\end{equation*}
		 We conclude from \cref{eq:def_v_beta}, \cref{eq:encadrementmi} and \cref{lem:roughboundonbetaix}  that with very high probability,   $v_{\ubarr{\beta}_{x}}(w)  \lesssim_{\delta} \left( 1 + \frac{\log N}{d} \right)$ uniformly for $x\in [N]$ and $w \in \mathrm{S}_{\#}$ on the event $\{\chi_{\#}=1\}$. Since the previous inequalities are independent of the choice of the coordinate $x$, we obtain $\norm{\mathbf u}^2_\infty \le \mathcal C \left( 1 + \frac{\log N}{d} \right) N^{-1+ \delta/2} \le \mathcal C N^{-1+ \delta/3}$, which concludes the proof.  
		 
For the proof of \cref{eq:unifbound}, we note that without the supremum over $w$ the bound holds uniformly for $w \in \mathrm{S}_\#$ due to \cref{thm:locallaw_H} and \cref{thm:locallaw_edge} if $d \asymp d_\#$. To include the supremum, it suffices to use a grid argument and the Lipschitz-continuity of $G_{\ubarr{x}\ubarr{x}}(\cdot, \ii \eta)$ and $v_{\ubarr{\beta}_x}( \cdot, \eta)$ for $\eta = N^{-1 + \delta/2}$, which we check now. 
		From \cref{lem:disttospect} and the dependence of $H(w)$ on $w$ in \cref{hermitizationofM}, we conclude that 
		$G_{\ubarr{x}\ubarr{x}}(\cdot, \ii \eta)$ is Lipschitz-continuous on $\mathrm{D}_2$ with constant $N^2$. 
		
		For the regularity of $v_{\ubarr{\beta}_x}$ we differentiate \cref{eq:inverse_v_beta} with respect to $w$ and obtain
		\begin{equation}\label{eq:deriv_v_beta}
			\partial v_{\beta} = - \beta^1 v_{\beta}^2 \partial v - \frac{\barr{w}}{\eta + \beta^2 v} v_{\beta}^2 + \frac{\abs{w}^2}{(\eta + \beta^2 v)^2}  \beta^{2} v_{\beta}^2 \partial v.
		\end{equation}
        First, we take $\beta = (1,1)$ in \cref{eq:deriv_v_beta} and deduce 
		\begin{equation*}
			\left( 1  - \abs{w}^2\frac{ v^2 }{(\eta+ v)^2} + v^2 \right) \partial v = -\frac{\barr{w}}{\eta + v} v^2.
		\end{equation*}
		For $w\in \mathrm{D}_2$  and owing to \cref{eq:v} we have $1 - \abs{w}^2\left( v/(v+ \eta)\right)^2 = v^2 + \eta/(v+\eta) $ and therefore
$\abs{ \partial v } \le N$ for all $(w, \eta) \in \mathrm{D}_2\times [N^{-1}, \infty)$. In particular, $w \mapsto  v(w,\eta)$ is $2N$-Lipschitz on $\mathrm{D}_2$ for all $\eta \geq N^{-1}$. Thus, \cref{eq:deriv_v_beta} implies that $w \mapsto v_{\beta_x}(w,\eta)$ is Lipschitz-continuous on $\mathrm{D}_2$ with constant $N^6$ for each $\eta \in [N^{-1},\infty)$ due to the (rough) bounds $v_{\beta}(w,\eta) \leq N$ for $\eta \in [N^{-1},\infty)$ and $\beta_x^i \leq N$. This completes the proof of the Lipschitz-continuity and, hence, the one of 
		\cref{thm:condideloc}. 
	\end{proof}
	
	\subsection{Proof of \cref{cor:delocalization_unconditional}} \label{sec:proof_cor_deloc_ER_graph}

	\begin{proof}[Proof of \cref{cor:delocalization_unconditional}] 
	We note that $\mathrm{Adj}(\mathbb G)/\sqrt{d}$ is of the form \eqref{eq:def_M_X_plus_f} for some matrix $X$ as in \cref{defiX} and $f = \sqrt{d}$. Owing to \cref{thm:condideloc} \cref{item:thm_cond_deloc_large_eigenvec} and the fact that $d \geq (1+\delta) \log N \ge 2$ we have delocalization of eigenvectors with eigenvalue in $\{\sqrt{d}\delta < |\lambda| < 2 \sqrt{d}\} $. Besides, owing to \cref{prop:spectrum_M} with probability $1 - o(1)$ all eigenvalues are within the $\mathrm{D}_{(1+\delta) \sqrt{d}}$. \nc Hence, owing to \cref{thm:condideloc} \cref{item:thm_cond_deloc_small_eigenvec}, 
		it suffices to check 
		that for some small enough $\eps>0$,  the event $\{ \beta_x^i \geq \eps \text{ for all } x,i \}$ occurs with probability at least $1-o(1)$ if $d \geq (1 + \delta) \log N$. In order to use a Poisson approximation on Bernoulli variables, we first compare the variables $(\beta_{x}^i)_{i,x}$ to the normalized in and out degree variables of  $\mathbb{G}$ given by 
		\begin{equation*}
			\alpha_{x}^1 := \frac{1}{d} \sum_{y}^{(x)} \mathrm{Adj}(\mathbb G)_{xy}  ~~\text{and}~~\alpha_{x}^2 := \frac{1}{d} \sum_{y}^{(x)} \mathrm{Adj}(\mathbb G)_{yx}
		\end{equation*}
	 for all $x \in [N]$ and $i\in \{1,2\}$. Note that  
		\begin{equation*}
			\beta_{x}^i = \biggl(1-2 \frac{d}{N}\biggr)\alpha_{x}^i +   O\bigg(\frac{d}{N}\bigg). 
		\end{equation*}
	Therefore, for any $\eps >0$, a union bound and the identical distribution of  $(\beta_x^i)_{i,x}$ yield 
		\begin{align*}
			\mathbb{P}:= \mathbb{P}\big( \exists x \in [N], i \in \{1,2\},~ \beta_{x}^i \le \eps  \big) \le 2N \mathbb{P}\big(\beta_x^1 \le \eps \big) \le 2N\mathbb{P} \big( d \alpha_1^1 \le 2 \eps d\big)
		\end{align*}
		 for all large enough $N$. 
		Note that $d\alpha_x^i$ is a binomial random variable of parameter $d/N$. Owing to \cite[Lemma D.3]{alt2023poisson}, denoting $\mathcal P_{d}$ a Poisson random variable of parameter $d$ we have
		\begin{align*}
			\mathbb{P} \big( d \alpha_1^1 \le 2 \eps d\big) &\le \mathbb{P} \big(\mathcal{P}_d \le 2 \eps d\big) +  O (e^{-N^{1/4}})\\
			&\le e^{2\eps d(1- \log 2\eps) -d} +  O (e^{-N^{1/4}})
		\end{align*}
		Finally considering $d \ge (1+ \delta) \log N$ we obtain 
		\begin{equation*}
			\mathbb{P}\le 2 \exp\big(-\log N \big( (1+ \delta ) (1 - 2 \eps ( 1- \log2\eps)) -1 \big)\big) +  O ( N e^{-N^{1/4}}).
		\end{equation*}
		For some small enough $\eps>0$  we have $\kappa_{\delta} :=(1+ \delta ) (1 - 2 \eps ( 1- \log2\eps)) >1 $ and $\mathbb{P} \le 3N^{-\kappa_{\delta} +1}$.
	\end{proof}

	\section{Proofs of preliminaries} 
	
	\subsection{Proof of statements in \cref{subsection:Approx_sc_eq} }\label{subsection:proofs_approx_sc}
		\begin{proof}[Proof of \cref{lem:roughboundsGT}]
		Throughout the proof we work on the event $\{\theta =1\}$. After relabeling vertices we can suppose that $T= [k]$. For all $k \in [N]$ we set 
		$$\Gamma_k = 1 \vee  \max_{x,y \notin T \cup \ubarr{T}} \abs{G^{\bset{T}}_{xy}}, $$
		and we show by induction on $k$ that there exists $\mathcal{C} >0 $ a constant such that
		\begin{equation}\label{eq:inductionass}
			\Gamma_k \le \Gamma_0 \left( 1 + \frac{ 36 \mathcal{C}\Gamma^2}{d}\right)^k
		\end{equation}
		for all $0 \le k \le \frac{d}{ 72 \mathcal{C}\Gamma^2}$. 
		Once \cref{eq:inductionass} is proved, \cref{eq:boundGTGTu} follows immediately as combined with $\Gamma_0 \le \Gamma$ and $1+ x \le \mathrm{e}^x$, \cref{eq:inductionass} implies $\Gamma_k \le \mathrm{e}^{1/2} \Gamma_0 \le 2 \Gamma$.
		
		We now turn to the proof of \cref{eq:inductionass} and note that the initial step $k=0$ is trivial. For the induction step, let $T = [k]$ and $u =k+1$. 
		We will estimate $G_{xy}^{\bset{Tu}}$ for $x,y \in [2N]\setminus ( T \cup \ubarr{T} \cup \{u,\ubarr{u}\})$ through its representations in \cref{eq:32a} and \cref{eq:A32b}. 
		We start with the former. 
		Since  $0 \le f \le N^{\delta/6}$ and $\Ima z \ge N^{-1 + \delta}$, the Cauchy-Schwarz inequality,   \cref{lemma:wardidentity} and $\Gamma_{k+1} \geq 1$ imply
		\begin{align}
			\frac{f}{N}\biggl|\sum_{a}^{(Tu)} G_{x\ubarr{a}}^{\bset{Tu}} G_{uy}^{\bset{T}}\biggr|&\le N^{\delta/6 -1} \sqrt{N} \biggl( \sum_{a\in [2N] \backslash ( T \cup \ubarr{T} \cup \{u,\ubarr{u}\})}\left|G_{xa}^{\bset{Tu}}\right|^2\biggr)^{1/2}\Gamma_{k}
			\le N^{-\delta/3}  \Gamma_{k+1}\Gamma_k. \label{eq:boundoforf1}
		\end{align}
	 The same argument yields the same bound on the absolute value of $\frac{f}{N} \sum_{a}^{(Tu)} G_{xa}^{\bset{Tu}} G_{\ubarr{u}y}^{\bset{T}}$. Moreover, we apply     \cref{lemma:dev1} to the family $ (G_{x\ubarr{a}}^{\bset{Tu}})_{a}$, which is independent of the family $(X_{au})_{a}$, with the choices $\psi = \frac{\Gamma_{k+1}}{\sqrt{d}}$, $\gamma = \sqrt{\frac{\Gamma_{k+1}}{N\Ima z}}$. 
		Note that $\psi/\gamma \geq N^{\delta/4}$ as $\Gamma_{k+1} \geq 1$ and $d \leq N^{\delta/2}$. 
		Together with \cref{lemma:wardidentity}, this yields
		\begin{align}\label{eq:useofdev11}
			\biggl| \sum_{a}^{(Tu)} G_{x\ubarr{a}}^{\bset{Tu}} \barr{X}_{au} G_{uy}^{\bset{T}} \biggr| \le  \frac{\mathcal{C}}{\sqrt{d}}\Gamma_{k+1} \Gamma_{k}
		\end{align}
		with very high probability. The same upper bound holds for $\sum_{a}^{(Tu)} G_{xa}^{\bset{Tu}} X_{au} G_{\ubarr{u}y}^{\bset{T}}$. 
		Therefore, we deduce from \cref{eq:32a} and $d \leq N^{\delta/2}$ as a first upper bound, holding with very high probability, 
		\begin{equation}\label{eq:upbound1}
			\Gamma_{k+1} \le \Gamma_{k} + \frac{\mathcal{C}}{\sqrt{d}} \Gamma_{k+1}\Gamma_k.
		\end{equation}
		We now use \cref{eq:A32b} to derive a second bound on $\Gamma_{k+1}$. 
		Analogously to \cref{eq:useofdev11}, we show that 
		\begin{equation}\label{eq:useofdev12}
			\biggl| \sum_{a,b}^{(Tu)} G_{x\ubarr{a}}^{\bset{Tu}} \barr{X}_{au} G_{uu}^{\bset{T}} X_{ub} G_{\ubarr{b}y}^{\bset{Tu}}\biggr|\le \frac{\mathcal{C}}{d} \Gamma^{2}_{k+1} \Gamma_{k}
		\end{equation}
		with very high probability. The same upper bound holds for all summands in $\mathrm{R}_2$ from \cref{eq:def_R_2} and, thus, for $\mathrm{R}_2$ itself.  
		Similarly as \cref{eq:boundoforf1} and \cref{eq:useofdev11}, we obtain 
		\begin{equation}\label{eq:upboundforf2}
			\frac{f}{N}\biggl| \sum_{a,b}^{(Tu)} G_{x\ubarr{a}}^{\bset{Tu}} \barr{X}_{au} G_{uu}^{\bset{T}} G_{\ubarr{b}y}^{\bset{Tu}} \biggr| \le N^{- \delta/ 3 }\frac{\mathcal{C}}{\sqrt{d}} \Gamma_k \Gamma_{k+1}^{2}
		\end{equation}
		with very high probability. 
		The previous bound holds for all the terms in $\mathrm{R}_1$ from \cref{eq:def_R_1}. 
		Since $\mathrm{R}_0$ from \cref{eq:defireste} is bounded as \cref{eq:boundoforf1}, 
		we finally arrive at the second bound 
		\begin{equation}\label{eq:upbound2}
			\Gamma_{k+1} \le \Gamma_{k} + \frac{\mathcal{C}}{d} \Gamma_{k+1}^2 \Gamma_k
		\end{equation}
		with very high probability. 
		We will now combine \cref{eq:upbound1} and \cref{eq:upbound2} to obtain \cref{eq:inductionass} for $k + 1$. Let $\mathcal C$ be chosen such that \cref{eq:upbound1} and \cref{eq:upbound2} hold. As explained after \cref{eq:inductionass}, $\Gamma_k \leq  2 \Gamma$ follows from the induction hypothesis \cref{eq:inductionass} since $ k \le \frac{\mathfrak{c} d}{\Gamma^2}$ with $\mathfrak c = \frac{1}{72 \mathcal C}$. 
		Thus, $\mathcal{C}\Gamma_k/\sqrt{d} \le 2\mathcal{C}\Gamma/ \sqrt{d}\le 1/2$ due to the lower bound $d \geq \frac{\Gamma^2}{64\mathfrak c}$. Hence, \cref{eq:upbound1} yields the rough bound 
		\begin{equation}\label{eq:roughaprioriboundGTule2GT}
			\Gamma_{k+1} \le 2 \Gamma_k
		\end{equation}
		with very high probability. We plug \cref{eq:roughaprioriboundGTule2GT} into \cref{eq:upbound2} and conclude 
		\begin{equation*}
			\Gamma_{k+1} \le \Gamma_{k} \left( 1 + \frac{4\mathcal{C}}{d}  \Gamma_{k}^2 \right) \le \Gamma_{k} \left( 1 + \frac{16\mathcal{C}\Gamma^2}{d} \right),
		\end{equation*}
		where the second step is a consequence of $\Gamma_{k} \le 2 \Gamma$.
		This completes the proof of \cref{eq:inductionass} for $k +1$ and, thus, the one of \cref{eq:boundGTGTu}.
		
		To prove \cref{eq:GTandGTuclose}, we start from \cref{eq:A32b} and proceed 
		analogously to the proof of \cref{eq:upbound2} using \cref{eq:boundGTGTu} to conclude. Indeed, owing to \cref{eq:boundoforf1} and \cref{eq:boundGTGTu} we have $\mathrm{R}_0 \vee \mathrm{R}_1 \le \mathcal C \Gamma^ 2 N^{-\delta/4} \le \mathcal C \Gamma^2d^{-1}$ and owing to \cref{eq:upbound1} we have $\mathrm{R}_2 \le \mathcal C \frac{\Gamma^3}{d}$.  
	\end{proof}

\begin{proof}[Proof of \cref{lem:roughboundsrest}]
		We fix $x \in [N]$ and work exclusively on the event $\{\theta=1\}$. Owing to  \cref{eq:boundGTGTu} in \cref{lem:roughboundsGT}, there is a constant $\mathcal{C}>0$ such that 
		\begin{equation*}
			\max_{a,b \in [2N]} \abs{G_{ab}} \le \mathcal{C} \Gamma,~~\text{and}~~\max_{a,b \in [2N]\backslash\{x,\ubarr{x}\}} \absb{G_{ab}^{\bset{x}}} \le \mathcal{C} \Gamma
		\end{equation*}
		with very high probability. 
		We prove the upper bound \cref{eq:boundrest} for $Y^{1}_{x}$ and $Z^{1}_x$. The proof is independent of the choice $x \in [N]$ and for $Y^2_x$ and $Z^{2}_x$ the proofs are similar. We treat the terms appearing in $Y_{x}^1$ in \cref{lem:ASCEq} from the left to the right. For the first term appearing in $Y_{x}^1$ we have
		\begin{equation*}
			\frac{\max_{a,b \in [2N]} \absb{G^{\bset{x}}_{ab}}}{d} \le\mathcal{C} \frac{\Gamma}{d}   =: \psi,
		\end{equation*}
		and applying \cref{lemma:wardidentity}, for all $b\in [N]$ we have
		\begin{equation*}
			\frac{1}{N}\sum_{a}^{(x)}\absb{G_{\ubarr{a}\ubarr{b}}^{\bset{x}}}^2\le \frac{1}{N}\sum_{a \in [2N] \backslash\{x,\ubarr{x}\}}\absb{G_{a\ubarr{b}}^{\bset{x}}}^2= \frac{1}{N\Ima z} \Ima G_{\ubarr{b}\ubarr{b}}^{\bset{x}} \le  \mathcal{C}\frac{\Gamma}{N\Ima z} ,
		\end{equation*}
		and for all $a \in [N]$
		\begin{equation*}
			\frac{1}{N}\sum_{b}^{(x)}\absb{G_{\ubarr{a}\ubarr{b}}^{\bset{x}}}^2\le \frac{1}{N}\sum_{b \in [2N] \backslash\{x,\ubarr{x}\}}\absb{G_{\ubarr{a}b}^{\bset{x}}}^2= \frac{1}{N\Ima z} \Ima G_{\ubarr{a}\ubarr{a}}^{\bset{x}}  \le  \mathcal{C}\frac{\Gamma}{N\Ima z} =: \gamma^2.
		\end{equation*}
		In regard of \cref{defiX} and applying the second estimate of \cref{lemma:dev1} to the variables $(G_{ab}^{\bset{x}})_{a,b}$ independent of $(X_{xa})_{a}$ we have with very high probability 
		\begin{equation}
			\absbb{ \sum_{a\neq b}^{(x)} X_{xa} G_{\ubarr{a}\ubarr{b}}^{\bset{x}} \barr{X}_{xb} } \le \mathcal{C} \frac{\Gamma}{d}.
		\end{equation}
		For the second term appearing in the rest $Y^{1}_x$ we first rewrite it
		\begin{equation*}
			\frac{f}{N} \absbb{ \sum_{a,b}^{(x)} X_{xa} G^{\bset{x}}_{\ubarr{a}\ubarr{b}}} = \frac{f}{N}\absbb{ \sum_{a}^{(x)} X_{xa} \underbrace{\sum_{b}  G^{\bset{x}}_{\ubarr{a}\ubarr{b}}}_{\alpha_a}}
		\end{equation*}
		in order to then apply the first estimate of \cref{lemma:dev1}. Indeed on the one hand we have
		\begin{equation*}
			\frac{\alpha_a}{\sqrt{d}} \le \sqrt{\frac{N}{d}} \biggl( \sum_{b}^{(x)} \absb{G^{\bset{x}}_{\ubarr{a}\ubarr{b}}}^2 \biggr)^{1/2} \le \sqrt{\frac{N}{d}} \biggl( \sum_{b\in [2N] \backslash \{x,\ubarr{x}\}} \absb{G^{\bset{x}}_{\ubarr{a}b}}^2 \biggr)^{1/2} = \sqrt{\frac{N\Ima G^{\bset{x}}_{\ubarr{a}\ubarr{a}}}{d\Ima z}} \le  \mathcal{C}\sqrt{\frac{N \Gamma}{d\Ima z}}=:\psi
		\end{equation*}
		where we used Cauchy-Schwarz for the first inequality, Ward identity \cref{eq:wardidentity} for the equality and finally $\theta=1$ and \cref{eq:boundGTGTu}. On the other hand we denote $ \mathds{1}_{x,\ubarr{N}}:= (\mathds{1}(b \ge 1, b \neq \ubarr{x}))_{b\in [2N]} \in \CC^{2N}$ and we have
		\begin{align*}
			\biggl(\frac{1}{N}\sum_{a}^{(x)} \abs{\alpha_a}^2 \biggr)^{1/2} = \biggl(\frac{1}{N}\underbrace{\sum_{a}^{(x)} \absbb{\sum_{b}^{(x)} G_{\ubarr{a}\ubarr{b}}^{\bset{x}}}^2}_{= \|G^{\bset{x}}\mathds{1}_{x,\ubarr{N}}\|^2} \biggr)^{1/2} \le \frac{1}{\sqrt{N}} \|G^{\bset{x}}\|\|\mathds{1}_{x,\ubarr{N}}\| \le \frac{1}{\Ima z} =: \gamma
		\end{align*}
		where we used the fact that $\|\mathds{1}_{x,\ubarr{N}}\|= \sqrt{N-1}$ and in the last step, \cref{lem:disttospect}. The last step to apply \cref{lemma:dev1} is to upper bound the ratio of $\psi$ and $\gamma$. We have
		\begin{equation*}
			\frac{\psi}{\gamma} = \mathcal{C}\sqrt{\frac{N \Gamma \Ima z }{d}}\ge \mathcal{C}N^{\delta/2}\sqrt{\frac{\Gamma }{d}} \ge N^{\delta/4}
		\end{equation*}
		and therefore
		\begin{equation}\label{eq:bounddoublesum}
			\frac{1}{N} \absbb{ f\sum_{a,b}^{(x)} X_{xa} G^{\bset{x}}_{\ubarr{a}\ubarr{b}}} \le  \mathcal{C}\abs{ f}\sqrt{\frac{\Gamma}{Nd\Ima z}}\le \mathcal{C}N^{-2\delta/3}\Gamma^{1/2}\le \mathcal C \frac{\Gamma}{d}.
		\end{equation}
		The third term $f/N\sum_{a,b}^{(x)} G_{\ubarr{a}\ubarr{b}}^{\bset{x}} \barr{X}_{xb}$ is bounded with exactly the same reasoning. For the fourth term appearing in $Y^{1}_{x}$ we have
		\begin{equation}\label{eq:upboundforfsquare}
			\frac{1}{N^2 }\absbb{f^2\sum_{a,b}^{(x)} G_{\ubarr{a}\ubarr{b}}} \le  \mathcal{C} \abs{f}^2 \sqrt{ \frac{N \Gamma}{ \Ima z}} \le \mathcal{C} N^{-\delta/6} \sqrt{\Gamma} \le \mathcal{C}\frac{\Gamma}{d},
		\end{equation}
		where we apply Cauchy-Schwarz and Ward identity \cref{eq:wardidentity} for the first inequality. This concludes the proof for $Y^1_x$. We now rewrite $Z^1_x$ as follows
		\begin{align}\label{eq:lasttermofY1x}
			Z^1_x = X_{xx} -\sum_{a,b}^{(x)} \biggl(X_{xa}G_{\ubarr{a}b}^{\bset{x}} X_{bx}  +\left(\frac{f}{N}\left[ X_{xa}G_{\ubarr{a}b}^{\bset{x}} + G_{\ubarr{a}b}^{\bset{x}}X_{bx} \right] + \frac{f^2}{N^2} G_{ \ubarr{a}b}^{\bset{x}}\right) \biggr).
		\end{align}
		By definition we have $\abs{X_{xx}} \le \mathcal{C}d^{-1/2}$ for $\mathcal{C} >0 $ large enough. We handle the first term of the sum above applying the third estimate of \cref{lemma:dev1}. For all $a,b\in [N] \backslash \{x\}$ we set $X_{a} = X_{xa}$ and $Y_{b} = X_{bx}$. 
		Besides we have
		\begin{equation*}
			\frac{\max_{a,b} \absb{G_{\ubarr{a}b}^{\bset{x}}}}{d} \le\mathcal{C} \frac{\Gamma}{d} =: \psi,~~\max_{b}\biggl( \frac{1}{N} \sum_{a}^{(x)} \absb{G^{\bset{x}}_{\ubarr{a}b}}^2 \biggr)^{1/2}\vee \max_{a}\biggl( \frac{1}{N}  \sum_{b}^{(x)} \absb{G^{\bset{x}}_{\ubarr{a}b}}^2 \biggr)^{1/2} \le \mathcal{C}\sqrt{\frac{\Gamma}{N\Ima z}}=: \gamma
		\end{equation*}
		where we applied Ward identity \cref{eq:wardidentity} for the second upper bound. Finally owing to the upper bound on $d$ in \cref{eq:regimed} we have
		\begin{equation*}
			 \frac{\psi}{\gamma} \ge \mathcal{C} N^{\delta/2} \sqrt{\Gamma}\ge N^{\delta/4}
		\end{equation*}
		and therefore
		\begin{equation*}
			\absbb{\sum_{a,b}^{(x)} X_{xa} G_{ \ubarr{a}b}^{\bset{x}} X_{bx} }  \le \mathcal{C}\frac{\Gamma}{d}.
		\end{equation*}
		The following two terms in \cref{eq:lasttermofY1x} are upper bounded as we did in \cref{eq:bounddoublesum} and the last term of \cref{eq:lasttermofY1x} is upper bounded as in \cref{eq:upboundforfsquare}. 
		
		To prove \cref{eq:roughboundGd-1/2} we consider $x\in [N]$ and $y \in [2N] \backslash \{x,\ubarr{x} \}$ and obtain from \cref{eq:A29A} the representation  
		\begin{equation*}
			G_{xy} = - \sum_{b}^{(x)} \bigg( G_{xx} {X}_{xb}G_{\ubarr{b}y}^{\bset{x}} + G_{x\ubarr{x}} \barr{X}_{bx}G_{by}^{\bset{x}} + \frac{f}{N} \bigg( G_{xx} G_{\ubarr{b}y}^{\bset{x}} + G_{x\ubarr{x}} G_{by}^{\bset{x}}\bigg) \bigg).
		\end{equation*}
		For the first two terms in the sum above we proceed as in \cref{eq:useofdev11}.  More precisely we set $\psi := \mathcal C \frac{\Gamma}{\sqrt{d}}$ and $\gamma:= \sqrt{\frac{\Gamma}{N \Ima z}}$. We then have $\psi/\gamma \ge N^{\delta/4}$  and therefore applying 
		\cref{lemma:dev1} we obtain that the first and second terms in the sum are bounded by $\mathcal C \Gamma \psi $. For the remaining terms with the factor $\frac{f}{N}$  we proceed as in \cref{eq:boundoforf1}. To bound the entry $G_{xy}$ for $x \in [2N] \setminus [N]$ and $y \in [2N] \setminus \{x,\ubarr{x}\}$, we deduce from \cref{eq:A29B} as 
		similar representation as the one above. Then the same argument applies. 
		
		Finally, \cref{eq:roughdifferenceGGTd-1} is the special case of \cref{eq:GTandGTuclose} in   \cref{lem:roughboundsGT} with $T= \emptyset$.
	\end{proof}
	\subsection{Proof of statements in \cref{sec:typical_vertices}}\label{subsection:proofs_typical_vert}
		This section is devoted to the proof of \cref{prop:mostverticesandneigharetypical}, which follows an analogous  strategy as the proof of \cite[Proposition~4.8]{alt2021delocalization} in   \cite[Section~4.2]{alt2021delocalization}.  
		Throughout we assume that $w \in \C$ and $z \in \C$ with $\Im z \geq N^{-1 + \delta}$. 
	
 For all $T \subset [N]$ we denote
	\begin{equation}\label{thetaT}
		\theta^{[T]} :=\mathds{1} \bigg( \max_{a,b \notin T\cup \ubarr{T}} \absn{G^{\bset{T}}_{ab}} \le 2 \mathcal C \Gamma \bigg).
	\end{equation}

	The following lemma says that a bound on the entries of $G$ implies a bound on the entries of $G^{\bset{T}}$ for $T \subset [N]$ as long as $\abs{T} = o(d)$.

	\begin{lemma}\label{lem:corthetalethetaT}
		Let $\mathfrak c >0$ be as in \cref{lem:roughboundsGT}. Then, for all $T \subset [N]$ satisfying $\abs{T} \leq \mathfrak c d/\Gamma^2$ and $ \frac{\Gamma^2}{64 \mathfrak c} \leq d \leq N^{\delta/2}$, we have $\theta \le \theta^{\bset{T}}$ with very high probability. 
	\end{lemma}
	\begin{proof}
		\cref{lem:corthetalethetaT} is a direct consequence of \cref{eq:boundGTGTu} in \cref{lem:roughboundsGT}. Indeed, it implies that $\theta = \theta \theta^{\bset{T}}$ with very high probability.
	\end{proof}
		\begin{defi}\label{def:typicalverticesdecoupled}
		For $\varphi >0$, $T \subset [N]$ and $x \in [N]$, we define
		\begin{equation*}
			\Phi^{1,\bset{T}}_x:= \sum_{y}^{(Tx)}\bigg(|X_{xy}|^2 - \frac{1}{N}\bigg), \qquad \Psi^{1,\bset{T}}_x:=  \sum_{y}^{(Tx)}\bigg(|X_{xy}|^2 - \frac{1}{N}\bigg) G^{\bset{Tx}}_{\ubarr{y}\ubarr{y}},
		\end{equation*}
		\begin{equation*}
			\Phi^{2,\bset{T}}_x:= \sum_{y}^{(Tx)}\bigg(|X_{yx}|^2 - \frac{1}{N}\bigg),\qquad \Psi^{2,\bset{T}}_x:=  \sum_{y}^{(Tx)}\bigg(|X_{yx}|^2 - \frac{1}{N}\bigg) G^{\bset{Tx}}_{yy}.
		\end{equation*}
		Similarly as in \cref{def:typicalvertices}, we define the set
		\begin{equation*}
			\mathcal{T}^{\bset{T}}_{\varphi} := \Bigl\{ x \in [N]\setminus T :~\max_{i \in \{1,2\}}\absn{\Phi^{i,\bset{T}}_x}\vee \absn{\Psi^{i,\bset{T}}_x }\le \varphi\Bigr\}.
		\end{equation*}
	\end{defi}

	\begin{lemma}\label{lemma:conditionalpbestimates}
		There exists a constant $\mathcal{C}\ge 1$, such that the following holds. Let $\mathfrak{c} = \mathfrak{c}(\nu, \delta)>0$ be as in \cref{lem:roughboundsGT} and $ \frac{\Gamma^2}{64 \mathfrak c} \leq d \leq N^{\delta/2}$. If $x \in T \subset [N]$ are deterministic and $\abs{T} +1\le \mathfrak{c}d/\Gamma^2$ then for any $0 < \epsilon \le 1$ and $i \in \{1,2\}$ we have
		\begin{equation}\label{eq:decoupledpbcondi}
			\theta^{\bset{T}} \mathbb{P} \big( \big |\Phi_x^{i,\bset{T}}\big| > \epsilon \, \big| \, H^{\bset{T}} \big) \le \ee^{-\frac{ \epsilon^2 d}{2^6(\ee\Gamma)^2}}, ~~ \theta^{\bset{T}} \mathbb{P} \big( \big|\Psi_x^{i,\bset{T}}\big| > \epsilon \, \big| \, H^{\bset{T}} \big) \le    \ee^{-\frac{ \epsilon^2 d}{2^6(\ee\Gamma)^2}},
		\end{equation}
	and for any $u \notin T$
		\begin{equation}\label{eq:accroissementPhiPsiT}
			\big| \Phi_x^{i,\bset{Tu}} - \Phi_x^{i,\bset{T}} \big| \le K^2d^{-1},~~~~ \theta^{\bset{T}}\big|\Psi_x^{i,\bset{Tu}} - \Psi_x^{i,\bset{T}}\big| \le \mathcal{C}  \frac{\Gamma + \beta^{i}_x\Gamma^3}{d}
		\end{equation}
		with very high probability.
	\end{lemma}

	\begin{proof}
From \cref{eq:sumaiXi2} with $r:= \epsilon^2 d/(64 \Gamma^2 \ee^2)$, $\mathbb{E}[\abs{X_{xy}}^2]  = 1/N$, Markov's inequality 
		and $r \leq d$ as $\epsilon \leq 1$ and $\Gamma \geq 1$, 
		we conclude  
		\begin{align*}
			\theta^{\bset{T}} \mathbb{P} \big( \big|\Psi_x^{1,\bset{T}}\big| > \epsilon \big| H^{\bset{T}} \big) &= \mathbb{P} \biggl(\theta^{\bset{T}} \biggl| \sum_{y}^{(T)} \left( \abs{X_{xy}}^2 - \mathbb{E}[\abs{X_{xy}}^2]\right) G_{\ubarr{y}\ubarr{y}}^{\bset{T}}\biggr| > \epsilon \bigg| H^{\bset{T}} \biggr) \\
			&\le \left( \frac{8 \Gamma}{\epsilon} \sqrt{\frac{r}{d}} \right)^r \le e^{-\frac{ \epsilon^2 d}{2^6(e\Gamma)^2}}.
		\end{align*}
		The estimates for $\Psi_x^{2,\bset{T}}$, $\Phi_x^{1,\bset{T}}$ and $\Phi_x^{2,\bset{T}}$ are proved similarly.
		Hence, \cref{eq:decoupledpbcondi} follows. 
		
		We use Markov's inequality with $r:= \epsilon^2d/(16 \ee^2)$ and employ \cref{eq:sumaiXi2} to obtain 
			\begin{align*}
			\mathbb{P} \big( \big|\Phi_x^{1,\bset{T}}\big| > \epsilon \big| H^{\bset{T}} \big) &= \mathbb{P} \biggl(\biggl| \sum_{y}^{(T)} \left( \abs{X_{xy}}^2 - \mathbb{E}[\abs{X_{xy}}^2]\right) \biggr| > \epsilon \bigg| H^{\bset{T}} \biggr) \\
			&\le \left( \frac{4}{\epsilon} \sqrt{\frac{r}{d}} \right)^r \le \exp\bigg(- \frac{\epsilon^2 d}{16\ee^2}\bigg).
		\end{align*}
		We now prove \cref{eq:accroissementPhiPsiT}. In the case $u =x$ the statement is trivial. For $x \neq u$ we have
		\begin{equation*}
			\Phi^{1,\bset{T}}_{x} - \Phi^{1,\bset{Tu}}_x =  \left( \abs{X_{xu}}^2 - \frac{1}{N}\right)
		\end{equation*}
		and then the claim for $\Phi^1$ follows from \cref{defiX} \cref{item:def_X_decay} and holds almost surely. The same computation gives the result for $\Phi^2$. Next we have 
		\begin{equation*}
			\Psi_{x}^{1,\bset{Tu}} - \Psi_{x}^{1,\bset{T}} = \sum_{y}^{(Tu)} \left( \abs{X_{xy}}^2 - \frac{1}{N} \right) (G_{\ubarr{y}\ubarr{y}}^{\bset{Tu}} -G_{\ubarr{y}\ubarr{y}}^{\bset{T}}) - \left( \abs{X_{xu}}^2 - \frac{1}{N} \right)G_{\ubarr{u}\ubarr{u}}^{\bset{T}}.
		\end{equation*}
		From \cref{lem:roughboundsGT}, we apply \cref{eq:GTandGTuclose} to the first sum and \cref{eq:boundGTGTu} to the remaining term and obtain
		\begin{equation*}
			\theta^{\bset{T}} \abs{\Psi_x^{1,\bset{Tu}} - \Psi_x^{1,\bset{T}}} \le \sum_{y}^{(Tu)}\abs{X_{xy}}^2 \mathcal{C}\frac{\Gamma^3}{d} + \mathcal{C}2 \frac{\Gamma}{d}  \le \mathcal{C} \frac{ \beta_{x}^1 \Gamma^3+ \Gamma}{d}
		\end{equation*}
		with very high probability. The proof for $\Psi_x^{2,\bset{T}}$ is analogous.
	\end{proof}
	\begin{lemma}\label{lemma:deterministicsetaretypical} 
	 Let $ \frac{\Gamma^2}{64 \mathfrak c} \leq d \leq N^{\delta/2}$ and $x\in [N]$, $ T \subset[N]$ be deterministic  such that $x \notin T$ and $\abs{T} \le \ccc    d/ \Gamma^2 $, where $\ccc= \ccc(\delta,\nu)>0$ is the constant given by \cref{lem:roughboundsGT}. Then  
		\begin{align}
			\mathbb{P} \left( T \subset \mathcal{T}_{\varphi/2}^c,\theta =1 \right) &\le  4^{|T|} \ee^{-\frac{\varphi^2d}{2^{12} (\ee\Gamma)^2}|T|} + \mathcal{C}N^{- \nu} \label{eq:bound_P_T_subset_atypical} \\
			\mathbb{P} \left( T \subset \left(\mathcal{T}^{\bset{x}}_{ \varphi/2}\right)^c,\theta^{\bset{x}} =1 \right) &\le 4^{|T|} \ee^{-\frac{\varphi^2d}{2^{12} (\ee\Gamma)^2}|T|} + \mathcal{C}N^{- \nu}
			\label{eq:bound_P_T_subset_atypical_without_x}
		\end{align}
	for all $ \varphi \in [\mathcal C \Gamma^3 \abs{T} d^{-1},1]$.
	\end{lemma}
	\begin{proof}
		The abbreviation $\mathbb{P}_{\theta}(\Theta) := \mathbb{P}(\Theta \cap \{ \theta =1 \})$ and the definition 
		\begin{equation*}
			\Omega_x := \bigcup_{i \in \{1,2\}}\{ \absn{\Phi^{i}_x} > \varphi/2\} \cup \{ \absn{\Psi^{i}_x} > \varphi/2\}
		\end{equation*}
		for $x \in [N]$ are used throughout this proof. We have
		\begin{align*}
			\mathbb{P} \big( T \subset \mathcal{T}_{\varphi/2}^{c}, \theta =1\big) = \mathbb{P}_{\theta} \big( \bigcap_{x \in T}\Omega_x \big).
		\end{align*}
		For all $i \in \{1,2\}$ and all $x \in [N]$ we have the identity 
		\begin{equation*}
			\big\{ \bigl|\Phi^{i}_x\bigr| > \varphi/2\big\} \cup \big\{ \bigl|\Psi^{i}_x\bigr| > \varphi/2\big\} = \big\{ \bigl|\Phi^{i}_x\bigr|> \varphi/2\big\} \cup \big\{ \bigl|\Phi^{i}_x\bigr|\le \varphi/2,~\bigl|\Psi^{i}_x\big| > \varphi/2\big\}
		\end{equation*}
		and the inclusions 
		\begin{align*}
			\big\{ \bigl|\Phi^{i}_x\bigr| > \varphi/2\big\}  &\subset \{\absn{\Phi^{i,\bset{T}}_x} > \varphi/4\big\} \cup   \event{\absn{\Phi^{i}_x-\Phi^{i,\bset{T}}_x} > \varphi/4},\\
			\event{ \absn{\Phi^{i}_x} \le \varphi/2,~\abs{\Psi^{i}_x} > \varphi/2} & \subset \left\{\absn{\Psi^{i,\bset{T}}_x} > \varphi/4\right\}\cup\event{ \absn{\Phi^{i}_x} \le \varphi/2, \absn{\Psi^{i}_x - \Psi^{i, \bset{T}}_x} > \varphi/4}.
		\end{align*}
		We define the events
		\begin{equation*}
			\Omega_x^{\bset{T}} := \bigcup_{i \in \{1,2\}} \bigl\{ \absn{\Phi^{i,\bset{T}}_x} > \varphi/4\bigr\} \cup \bigl\{ \absn{\Psi^{i,\bset{T}}_x} > \varphi/4\bigr\}
		\end{equation*}
		for $x \in [N]$ and obtain 
		\begin{equation} \label{eq:pbintersectionOmegax}
			\begin{aligned}
				\mathbb{P}_{\theta} \big( \bigcap_{x\in T} \Omega_x \big) \le \mathbb{P}_{\theta} \big( \bigcap_{x\in T} \Omega^{\bset{T}}_x \big) 
				& + \sum_{\substack{ i \in \left\{1,2\right\} \\ x\in T}} \mathbb{P}_{\theta} \big( \absn{\Phi^{i}_x-\Phi^{i,\bset{T}}_x} > \varphi/4\big)\\
				&+ \sum_{\substack{ i \in \left\{1,2\right\} \\ x\in T}} \mathbb{P}_{\theta} \big( \absn{\Phi^{i}_x} \le \varphi/2, \absn{\Psi^{i}_x - \Psi^{i, \bset{T}}_x} > \varphi/4\big). 
			\end{aligned}
		\end{equation}
		We now bound the first probability on the right-hand side of \cref{eq:pbintersectionOmegax}. 
		\cref{lem:corthetalethetaT} yields 
		\begin{equation*}
			\mathbb{P}_{\theta}\big( \bigcap_{x \in T} \Omega_x^{\bset{T}} \big) \le \mathbb{P}\big( \{ \theta^{\bset{T}} = 1 \} \cap  \bigcap_{x \in T} \Omega_x^{\bset{T}}\big) + \mathcal{C}N^{-2\nu}.
		\end{equation*}
		Since the variables $(\mathbb{E}\big[ \mathds{1} ( \Omega_{x}^{\bset{T}} ) \big| H^{\bset{T}}\big])_{x \in T}$ are independent by construction, we conclude
		\begin{equation*}
			\mathbb{P}_{\theta}\big( \bigcap_{x \in T} \Omega_x^{\bset{T}} \big) \le \mathbb{E}\big[ \prod_{x \in T} \theta^{\bset{T}} \mathbb{P} \big( \Omega_{x}^{\bset{T}}\big| H^{\bset{T}}\big)\big] + \mathcal{C}N^{-2\nu}.
		\end{equation*}
		Now for all $x \in T$, considering $\epsilon = \varphi/4$ in \cref{lemma:conditionalpbestimates} we have
		\begin{align*}
			\theta^{\bset{T}} \mathbb{P} \big( \Omega_{x}^{\bset{T}} \big| H^{\bset{T}}\big) &\le \sum_{i \in \{1,2\}}\theta^{\bset{T}} \left( \mathbb{P} \big(  \absn{\Phi^{i,\bset{T}}_x} > \varphi/4 \big| H^{\bset{T}} \big)+ \mathbb{P} \big(  \absn{\Psi^{i,\bset{T}}_x} > \varphi/4\big| H^{\bset{T}}  \big) \right) \le 4  e^{-\frac{\varphi^2d}{2^{12} (e\Gamma)^2}}. 
		\end{align*}
	 We conclude that
		\begin{align*}
			\mathbb{P}_{\theta}\big( \bigcap_{x \in T} \Omega_x^{\bset{T}} \big) \le 4^{|T|} e^{-\frac{\varphi^2d}{2^{12} (e\Gamma)^2}|T|} + \mathcal{C}N^{-2\nu} .
		\end{align*}
		To estimate the sums on the right-hand side of \cref{eq:pbintersectionOmegax} we set $k= \abs{T}$ and, for $T= \{ x_1,\ldots,x_k\}$, $T_{t}:= \{ x_1,\ldots,x_t\}$ for $t =1, \ldots, k$ with $T_0 = \emptyset$. From the first estimate in \cref{eq:accroissementPhiPsiT} of \cref{lemma:conditionalpbestimates}, we conclude 
		\begin{align*}
			\bigl|\Phi_{x}^{i} - \Phi_{x}^{i,\bset{T}}\bigr| \le \sum_{t=0}^{k-1} \bigl| \Phi_{x}^{i, \bset{T_{t}}} - \Phi_{x}^{i,\bset{T_{t+1}}}\bigr| \le  \mathcal C \frac{\abs{T}}{d}
		\end{align*}
		with very high probability. Since $\varphi> 4 \mathcal C \abs{T}d^{-1}$ by assumption as $\Gamma \geq 1$, this implies that the second term on the right-hand side of \cref{eq:pbintersectionOmegax} is $O(N^{-2\nu})$.
		
 To estimate the last term on the right-hand side of \cref{eq:pbintersectionOmegax}, we note that $\beta^{i}_x\le 2$ on the event $\left\{ \abs{\Phi^{i}_x} \le \varphi/2 \right\}$ as $\varphi \le 1$. Thus, the second bound in \cref{eq:accroissementPhiPsiT} of \cref{lemma:conditionalpbestimates} implies
		\begin{align*}
			\theta \bigl| \Psi_{x}^{i} - \Psi_{x}^{i,\bset{T}}\bigr| \le \sum_{t=0}^{k-1} \theta^{\bset{T_t}}\bigl| \Psi_{x}^{i, \bset{T_{t}}} - \Psi_{x}^{i,\bset{T_{t+1}}}\bigr| \le \mathcal{C} \frac{\Gamma^3}{d} \abs{T} 
		\end{align*}
		with very high probability, as 
		$\theta \le \theta^{\bset{T_t}}  $ with very high probability for all $ 1 \le t \le k $ by  \cref{lem:corthetalethetaT}.
 Therefore, as $\varphi \geq 4 \abs{T} \mathcal C \Gamma^{3} d^{-1}$ by assumption,
 after possibly increasing  $\mathcal{C}>0$, we obtain 
		\begin{equation*}
			\mathbb{P}_{\theta} \big( \abs{\Phi^{i}_x} \le \varphi/2,\bigl| \Psi_{x}^{i} - \Psi_{x}^{i,\bset{T}}\bigr| > \varphi/4 \big) \le \mathcal{C} N^{-2\nu}.
		\end{equation*}
		By inserting the previous bounds into \cref{eq:pbintersectionOmegax} we have
		\begin{equation*}
			\mathbb{P}_{\theta}\left( \bigcap_{x \in T}\Omega_x\right) \le   4^{|T|} e^{-\frac{\varphi^2d}{2^{12} (e\Gamma)^2}|T|}  +  \mathcal{C}N^{-2\nu}+  4 \frac{\ccc d}{\Gamma^2} N^{-2 \nu} \le 4^{|T|} e^{-\frac{\varphi^2d}{2^{12} (e\Gamma)^2}|T|} \nc+ \mathcal{C}^{\prime\prime}N^{-\nu}. 
		\end{equation*}
		The proof of \cref{eq:bound_P_T_subset_atypical_without_x} is analogous. In fact, it essentially corresponds to \cref{eq:bound_P_T_subset_atypical} with the vertex set $[N]\backslash\{x\}$ up to a slightly different scaling.   
	\end{proof}
		
For the remainder of this section we need to restrict to \cref{eq:regimedlog2}.
\begin{lemma}\label{lem:controlatypicalvertices}
 Let $d$ satisfy \cref{eq:regimedlog2}.
		Then there exist constants $0 < \ccc$, $\tilde{\ccc}$ depending only on $\nu$ and $\delta$  such that the following holds with very high probability on the event $\{\theta =1\}$. Suppose $\varphi$ is defined as in \cref{eq:new_varphi},
	 then for each deterministic $V \subset [N]$, 
		 we have 
		\begin{enumerate}[label=(\roman*)]
			\item \label{(i)}$ \absn{V\cap \mathcal{T}_{\varphi/2}^{c}} \le  \ee^{\frac{\varphi^2d}{2^{14} (\ee\Gamma)^2}}  + 4\abs{V} \ee^{-\frac{\varphi^2d}{2^{13} (\ee\Gamma)^2}} $
			\item \label{(ii)} If $\absn{V} \le \frac{1}{4} \ee^{\frac{\varphi^2 d }{2^{13} (\ee\Gamma)^2}} $ then $\absn{V \cap \mathcal{T}_{\varphi/2}^c} \le \ccc\varphi d/ \Gamma^3$. 
		
		\end{enumerate}
		For any deterministic $x \in [N]$, the same estimates hold for $\bigl(\mathcal{T}_{\varphi/2}^{\bset{x}}\bigr)^c$ instead of $\mathcal{T}_{\varphi/2}^{c}$ and a random set $V \subset [N] \backslash \{x\}$ independent of $H^{\bset{x}}$.
	\end{lemma}
	
\begin{proof}
		Throughout this proof we abbreviate $\mathbb{P}_{\theta}(\Theta) := \mathbb{P}(\Theta \cap \{ \theta =1 \})$. We fix $\nu >0$ and let $\ccc>0$ be the constant from \cref{lemma:deterministicsetaretypical}. Throughout the following, when applying \cref{lemma:deterministicsetaretypical}, we use that $\min\{ d \Gamma^{-2},  d \varphi \Gamma^{-3}\}= d \varphi \Gamma^{-3}$ since $\Gamma \geq 1$ and $\varphi \leq 1$.
		To prove \cref{(ii)}, 
		 we fix  $k = \ccc \varphi d\Gamma^{-3}$, choose $\tilde{\ccc}>0$ 
		 in \cref{eq:new_varphi} such that $\frac{\varphi^2d}{2^{13} (e\Gamma)^2}k = \nu \log N $, and estimate
		\begin{align*}
			\mathbb{P}_{\theta}\big( \absn{V \cap \mathcal{T}_{\varphi/2}^c} \ge k \big) &\le \sum_{W \subset V:\abs{W}=k}\mathbb{P}_{\theta}\big( W \subset \mathcal{T}_{\varphi/2}^c \big) \le \binom{\abs{V}}{k} \Big( 4^{k} e^{-\frac{\varphi^2d}{2^{12} (e\Gamma)^2}k} + \mathcal{C}N^{-2 \nu}\Big)\\
			&\le \Big(4\abs{V}e^{-\frac{\varphi^2d}{2^{12}(e\Gamma)^2}}\Big)^k + \mathcal{C} \abs{V}^{k}N^{-2\nu}\le  e^{-\frac{\varphi^2d}{2^{13} (e\Gamma)^2}k}+ \mathcal{C}e^{\frac{\varphi^2d}{2^{13} (e\Gamma)^2}k} N^{-2\nu} \\
			&\le \mathcal{C^{\prime}} N^{-\nu}
		\end{align*}
where we applied \cref{lemma:deterministicsetaretypical} for all $ W \subset V$ of cardinality $k$ in the second step. In the third step, we used the upper bound on $\abs{V}$ and, in the fourth one, 
that our choice of $\tilde{\ccc}$ implies 
\begin{equation} \label{eq:choice_varphi_conclusion} 
\ee^{\frac{\varphi^2d}{2^{13} (e\Gamma)^2}k}= \ee^{\frac{\varphi^2d}{2^{13} (e\Gamma)^2} \frac{\ccc \varphi d}{\Gamma^3}} = N^{\nu}. 
\end{equation} 
To prove \cref{(i)}, for any $t>0$ and $\ell \in \N$, we conclude from Markov's inequality that 
\begin{align*}
			\mathbb{P}_{\theta} \big( \absn{V \cap \mathcal{T}_{\varphi/2}^c} \ge t \big) &\le \frac{1}{t^{\ell}} \mathbb{E} \biggl[ \biggl( \sum_{x \in V} \mathds{1}\big( x \in \mathcal{T}_{\varphi/2}^c \big) \theta \biggl)^{\ell}\biggr] = \frac{1}{t^\ell}\sum_{x_1,...,x_{\ell} \in V}\mathbb{P}_{\theta}\biggl( \bigcap_{j=1}^{\ell} \big\{x_j \in \mathcal{T}_{\varphi/2}^c \big\} \biggr).
		\end{align*}
	We denote by $\mathcal{P}_{\ell}$ the set of partitions of $[\ell]$, choose  $\ell =\ccc d\varphi/\Gamma^3$, apply \cref{eq:bound_P_T_subset_atypical} from \cref{lemma:deterministicsetaretypical} and obtain
		\begin{align*}
			\mathbb{P}_{\theta} \big( \absn{V\cap \mathcal{T}_{\varphi/2}^c} \ge t \big) &\le \frac{1}{t^{\ell}} \sum_{\pi \in \mathcal{P}_{\ell}}  (4\abs{V})^{\abs{\pi}} \left(  e^{-\frac{\varphi^2d}{2^{12} (e\Gamma)^2}|\pi|} + \mathcal{C}N^{-2 \nu}\right) \\
		&\le  \frac{1}{t^{\ell}} \sum_{ k=0}^\ell \binom{\ell}{k} \ell^{\ell-k}(4\abs{V})^{k} \left(  e^{-\frac{\varphi^2d}{2^{12} (e\Gamma)^2}k}  + \mathcal{C}N^{- 2\nu}\right)\\
			&\le \frac{ \left( \ell +  4\abs{V} e^{-\frac{\varphi^2d}{2^{12} (e\Gamma)^2}}  \right)^\ell + \mathcal{C} N^{-2\nu}\left(\ell +  4 \abs{V} \right)^\ell}{t^\ell}
		\end{align*}
		where we used that the number of partition of $[\ell]$ consisting of $k$ block is bounded by $\binom{\ell}{k}\ell^{\ell-k}$. We now take $t = e^{\frac{\varphi^2d}{2^{14} (e\Gamma)^2}}  + 4\abs{V} e^{-\frac{\varphi^2d}{2^{13} (e\Gamma)^2}} $ and obtain
		\begin{align*}
			\left( \frac{\ell}{t} \right)^{\ell} \le \left(\frac{\ccc d\varphi}{\Gamma^3}\right)^\ell e^{-\frac{\varphi^2d }{2^{14}(e\Gamma)^2} \ell } = \left(\frac{\ccc d\varphi}{\Gamma^3}\right)^\ell N^{-\nu/2} \le \mathcal C N^{-\nu/4}
		\end{align*}
		 from \cref{eq:new_varphi},  \cref{eq:choice_varphi_conclusion} 
		 and the upper bound in \cref{eq:regimedlog2} as well as possibly shrinking $\ccc>0$. Moreover, our  definitions of $t$ and $\ell$ as well as \cref{eq:choice_varphi_conclusion} yield \begin{align*}
		\biggl(	\frac{ 4\abs{V} e^{-\frac{\varphi^2d}{2^{12} (e\Gamma)^2}} }{t} \biggr)^\ell \le \Bigl(e^{-\frac{\varphi^2d}{2^{13} (e\Gamma)^2}}\Bigr)^\ell = N^{-\nu}.
		\end{align*}
		By combining these estimates, we obtain 
		\begin{equation*}
			\frac{ \big( \ell + 4\abs{V} e^{-\frac{\varphi^2d}{2^{12} (e\Gamma)^2}}  \big)^\ell}{t^\ell} \le 2^\ell \max \bigg\{ \bigg(\frac{\ell}{t}\bigg)^\ell, \bigg( \frac{ 4\abs{V} e^{-\frac{\varphi^2d}{2^{12} (e\Gamma)^2}} }{t}\bigg)^{\ell}\bigg\}  \le N^{ \log 2 - \nu/2}.
		\end{equation*}
		On the other hand, from \cref{eq:choice_varphi_conclusion} and our definition of $\ell$, we conclude 
		\begin{equation*}
			\left( \frac{ 4 \nc \abs{V}}{t} \right)^{\ell} \le e^{\frac{\varphi^2d \ell}{2^{14} (e\Gamma)^2}}  \le N^{\nu/2}
		\end{equation*}
		and therefore
		\begin{equation*}
			\mathcal{C} N^{-2\nu} \left( \frac{\ell + 4\abs{V}}{t}\right)^{\ell}\le \mathcal{C} N^{\log 2 - \nu}.
		\end{equation*}
		By a-priori choosing $\nu>0$ large enough, the claim follows. 
	 
	 Following the same argument and conditioning on the random set $V$ yield the statement about $\mathcal T^{[x]}$.
	\end{proof}

	\begin{lemma}\label{lem:boundthetabsetxvarphi}
		With very high probability, we have
		\begin{equation}\label{eq:accroissementvarphia}
			\abs{\Phi_{y}^{i} - \Phi_{y}^{i,\bset{x}}} \le \varphi/2, \qquad  \theta\abs{\Psi_{y}^{i} - \Psi_{y}^{i,\bset{x}}} \le \varphi/2
		\end{equation}
		for all $i \in \{1,2\}$ and all $x,y \in [N]$.
	\end{lemma}

	\begin{proof}[Proof of \cref{lem:boundthetabsetxvarphi}]
		We fix $i \in \{1,2\}$ and $x,y \in [N]$. 
	 The first estimate easily follows from the first estimate in 
		\cref{eq:accroissementPhiPsiT} with $T= \emptyset$ and $u=x$ and the definition of $\varphi$ in \cref{eq:new_varphi}. 
		
		From the second estimate in \cref{eq:accroissementPhiPsiT} with $T= \emptyset$ and $u=x$ and \cref{lem:roughboundonbetaix}, we conclude that 
		\[ 
		\theta \abs{\Psi_{y}^{i,\bset{x}} - \Psi_{y}^{i}} \le \mathcal{C}  \frac{\Gamma}{d} + \mathcal C  \frac{\Gamma^3}{d} \bigg( 1 + \frac{ \log N}{d} \bigg)  \leq \frac{\varphi}{2}, 		\] 
		where, in the last step, we used \cref{eq:new_varphi}, $1 \leq \Gamma \leq 1 + \frac{\log N}{d}$ and $d \geq (\log N)^{3/4}$.  
	\end{proof}

	\begin{proof}[Proof of \cref{prop:mostverticesandneigharetypical}]
		Item \cref{mostverticestypical} follows immediately from  \cref{lem:controlatypicalvertices} \cref{(i)} with $V= [N]$.
		
		For the proof of  \cref{mostneighboursaretypical}, we only show  
	 the upper bound on $\sum_{y\in \mathcal{T}^c_{\varphi}}^{(x)} \abs{X_{xy}}^2$. The analogous proof shows the same bound for $\sum_{y\in \mathcal{T}^c_{\varphi}}^{(x)} \abs{X_{yx}}^2$. First, \cref{lem:boundthetabsetxvarphi} implies $\mathcal{T}_{\varphi}^c  \setminus \{ x\}  \subset \big(\mathcal{T}^{\bset{x}}_{\varphi/2}\big)^c \setminus \{ x\}  $ with very high probability on the event $\{\theta =1 \}$ and hence
		\begin{equation*}
			\theta \sum_{y \in \mathcal{T}_{\varphi}^c}^{(x)} \abs{X_{xy}}^2 \le \theta \sum_{y \in \big(\mathcal{T}^{\bset{x}}_{\varphi/2}\big)^c }^{(x)} \abs{X_{xy}}^2
		\end{equation*}
		with very high probability. Owing to \cref{defiX} \cref{item:def_X_decay}, $\abs{X_{xy}}^2\le K^2/d $ almost surely. Therefore, 
		\begin{align*}
			\sum_{y \in \big(\mathcal{T}^{\bset{x}}_{\varphi/2}\big)^c }^{(x)} \abs{X_{xy}}^2 &\le \sum_{k=0}^{\log N }\sum_{y \in \big(\mathcal{T}^{\bset{x}}_{\varphi/2}\big)^c }^{(x)} \abs{X_{xy}}^2 \mathds{1}\left( d^{-k-2}\le \abs{X_{xy}}^2 \le d^{-k-1}\right) + Nd^{-\log N }\\
			&\le \sum_{k=0}^{\log N}\sum_{y \in \big(\mathcal{T}^{\bset{x}}_{\varphi/2}\big)^c } d^{-k-1} \mathds{1}\left( d^{-k-2}\le \abs{X_{xy}}^2 \le d^{-k-1}\right) + \frac{1}{N}\\
			&\le  \sum_{k=0}^{\log N } d^{-k-1}\abs{W_k \cap\left(\mathcal{T}^{\bset{x}}_{\varphi/2}\right)^c}  + \frac{1}{N}
		\end{align*}
		where we used $d^{\log N } \ge N^2$ as $d \geq \ee^2$ and defined 
		\begin{equation*}
			V_k := \left\{ y \neq x:~ \abs{X_{xy}}^2\ge d^{-k-2}\right\}.
		\end{equation*}
		Owing to \cref{lem:roughboundonbetaix} and $d^2 \geq \log N$,  we have   $ \sum_{y } \abs{X_{xy}}^2  \le \mathcal{C}d$ and therefore, $\abs{V_{k}}\le \mathcal{C}d^{k+3}$ with very high probability. We set 
		\begin{equation*}
		n:= \max \biggl\{ k \ge 0:~ \mathcal{C}d^{k+3} \le e^{\frac{\varphi^2d}{2^{13} (e\Gamma)^2}} \biggr\}. 
			\end{equation*}
			Since the family $(V_k)_{k}$ and $H^{\bset{x}}$ are independent, we can apply \cref{lem:controlatypicalvertices} and obtain
		\begin{align*}
			\sum_{k=0}^{\log N} d^{-k-1}\abs{V_k \cap\left(\mathcal{T}^{\bset{x}}_{\varphi/2}\right)^c} &= \sum_{k=0}^{n} d^{-k-1}\abs{V_k \cap\left(\mathcal{T}^{\bset{x}}_{\varphi/2}\right)^c}  + \sum_{k=n+1}^{\log N} d^{-k-1}\abs{V_k \cap\left(\mathcal{T}^{\bset{x}}_{\varphi/2}\right)^c}  \\
			&\le \sum_{k=0}^{n} d^{-k-1} \ccc d\varphi\Gamma^{-3}  + \sum_{k=n+1}^{\log N} d^{-k-1}\left( e^{\frac{\varphi^2d}{2^{14} (e\Gamma)^2}}  + 4\mathcal C d^{k+3} e^{-\frac{\varphi^2d}{2^{13} (e\Gamma)^2}}\nc \right)\\
			&\le 2 \ccc \varphi\Gamma^{-3} +  5\mathcal{C}d^2e^{-\frac{\varphi^2d}{2^{14} (e\Gamma)^2}}\log N
		\end{align*}
		with very high probability on the event $\{ \theta =1 \}$. We conclude the claimed estimate as $d^2 \geq \log N$ due to our lower bound on $d$.
	\end{proof}

	\subsection{Proofs of statements in \cref{section:properties_v_v_beta}}\label{subsection:proof_properties_v}
	\begin{proof}[Proof of \cref{lem:profv}]
		
		Equation \cref{eq:v_leq_1} is a direct consequence of \cref{eq:v}. Indeed all the terms on each sides of \cref{eq:v} are non-negative and therefore we have directly $1/v \ge \eta $. Similarly, multiplying the entire equation \cref{eq:v} by $v$ we conclude $1 \ge v^2$.

		We now show \cref{eq:v_full_scaling} and \cref{eq:vgeeta} simultaneously. We conclude $\eta + v = v(\eta + v)^2 + \abs{w}^2 v$ from \cref{eq:v}.  
		First, we assume $\abs{w} \leq 1$. Hence, all summands in 
		\begin{equation} \label{eq:v_full_scaling_proof1} 
			\eta + (1 - \abs{w}^2)v = v(\eta + v)^2 
		\end{equation} 
		are positive, 
		which implies $\eta \lesssim v$, i.e.\ \cref{eq:vgeeta}, as 
		$\eta + v \lesssim 1$ by $\eta \leq L$ and \cref{eq:v_leq_1}. 
		Hence, $v^3 \leq v (\eta +v)^2 \lesssim v^3$ implying $\eta + (1- \abs{w}^2) v \asymp v^3$ by \cref{eq:v_full_scaling_proof1}. By distinguishing which summand is larger on the left-hand side, we see that this is equivalent to $ \eta\asymp v^3$ or $(1- \abs{w}^2) \asymp v^2$. This proves the first case of \cref{eq:v_full_scaling}. 
		
		If $\abs{w} \geq 1$ then all summands of 
		\begin{equation} \label{eq:v_full_scaling_proof2} 
			\eta  = v(\eta + v)^2  + (\abs{w}^2-1)v
		\end{equation} 
		are positive. Since $\eta + v \lesssim 1$ as before and $\abs{w}^2 - 1 \leq 5$, \eqref{eq:v_full_scaling_proof2} 
		implies $\eta \lesssim v$ proving \eqref{eq:vgeeta}. 
		We conclude from $\eta \lesssim v$ and \eqref{eq:v_full_scaling_proof2} that $\eta \asymp v^3 + (\abs{w}^2 - 1)v$. 
		This implies $v \asymp \eta^{1/3}$ or $v \asymp \frac{\eta}{\abs{w}^2 -1}$, which is equivalent to \eqref{eq:v_full_scaling}. 
		
		Given \eqref{eq:v_full_scaling}, the lower bound $v \gtrsim 1$ for $w \in \mathrm{D}_{1-\delta}$, i.e.\ \eqref{eq:v_asymp_1}, is obvious. 
	\end{proof}

	\appendix 
	
	\crefalias{subsection}{appendix}
	\crefalias{section}{appendix}
	
	\section{Appendix} 
	
	\subsection{Stieltjes transforms} 

Motivated by the identity \cref{eq:v} for $v$, we consider the relation 
\begin{equation}\label{eq:deterministselfconseq}
		-\frac{1}{m}= z + m - \frac{|w|^2}{z+m}.
	\end{equation}
	for $z \in \C_+$ and $w \in \C$ and denote by $m(w,z)$ the unique solution of \cref{eq:deterministselfconseq} in $\C_+$.
	As shown in the next lemma, $z\mapsto m(w,z)$ is the Stieltjes transform of a symmetric probability measure on $\R$. In fact, with the methods developed in this paper, one can prove that this measure in the limit of the empirical spectral measure of $H(w)$ as $N$ tends to infinity.
	
	Motivated by the definition \cref{eq:def_v_beta}, for any $\beta = (\beta^1, \beta^2) \in [0,\infty)^2$, we define 
	\begin{equation} \label{eq:def_m_beta}  
		m_\beta(w,z) := \frac{z + \beta^2m(w,z)}{\abs{w}^2 - (z + \beta^1 m(w,z))(z + \beta^2 m(w,z))} 
	\end{equation} 
	for all $w \in \C$ and $z \in \C_+$. 
	In \cref{lem:m_1_Stieltjes} below, we verify the well-definedness of the definition of $m_\beta$ and show that it is the Stieltjes transform of a symmetric measure on $\R$. 
	In particular, if $z = \ii \eta$ then $m=m_{(1,1)}$ and $m_{\beta}$ are purely imaginary  for all $\beta \in [0,\infty)^2$ and we write 
	\begin{equation} \label{eq:v_beta_m_beta} 
		m_{\beta}(w,\ii\eta) = \ii v_{\beta}(w,\eta)
 	\end{equation}
	for $v_{\beta}(w,\eta)>0$ from \eqref{eq:def_v_beta}.

	\begin{lemma}[Stieltjes transform representation of $m_\beta$]
		\label{lem:m_1_Stieltjes} 
		Let $m$ be as in \eqref{eq:deterministselfconseq} and $\beta =(\beta^1, \beta^2) \in [0,\infty)^2$. Then $m_\beta(w,z)$ from \eqref{eq:def_m_beta} is well-defined for all $w \in \C$, $z \in \C_+$. Furthermore, for each $w \in \C$, the map $\C_+ \to \C$, $z \mapsto m_\beta(w,z)$ is the Stieltjes transform of a probability measure $\mu_\beta$ on $\RR$. 
		Moreover, $m_\beta(- \bar z) = - \barr{m_\beta(z)}$ for all $z \in \C_+$, i.e.\ the measure $\mu_\beta$ is symmetric. 
	\end{lemma}
	
	\begin{proof} 
		We fix $w \in \C$, suppress it from our notation and denote by $m_\beta$ the map $\C_+ \to \C$, $z \mapsto m_\beta(w,z)$.
		For the well-definedness, we note that $\im (z + \beta^1m ) >0$ and $\im (z + \beta^1m ) >0$ for all $z \in \C_+$, $w \in \C$, $\beta^1 \geq 0$ and $\beta^2 \geq 0$ as $\im m >0$ by definition. 
		Moreover, the imaginary part of the denominator vanishes iff $\re(z + \beta^2m) = - \re (z + \beta^1m)\im(z + \beta^2m)/\im (z + \beta^1 m)$. In this case, 
		the real part of the denominator reads as $\abs{w}^2 + \im(z + \beta^1 m) \im (z+ \beta^2m ) + (\re (z + \beta^1 m))^2 \im (z + \beta^2 m)/\im(z + \beta^1 m)$, which is strictly positive. 
		Thus, $m_\beta$ is well-defined and holomorphic as a function of $z$ on $\C_+$. 
		A short computation reveals that 
		\[ \im m_\beta= \frac{\abs{w}^2 \im(z + \beta^2 m) + \abs{z + \beta^2m}^2 \im (z + \beta^1 m)}{\abs{\abs{w}^2 - (z+ \beta^1m) (z + \beta^2m)}^2},  
		\]
		which shows that $\im m_\beta(z) >0$ if $z \in \C_+$. From \eqref{eq:def_m_beta}, we conclude that $\ii \eta m_\beta(\ii \eta) \to -1$ as $\eta \to \infty$ as $m(\ii \eta) \to 0$ when $\eta \to \infty$ by \eqref{eq:deterministselfconseq}. 
		Hence, we have proved that $m_\beta \colon \C_+ \to \C_+$ is a holomorphic map and 
		$\ii \eta m_\beta(\ii \eta) \to -1$ for $\eta \to \infty$. 
		Therefore, the existence of $\mu_\beta$ follows from the well-known representation theorem for Nevanlinna functions \cite[Theorem 11.9]{Rudin1987Analysis}.

		Owing to \eqref{eq:deterministselfconseq} and its uniqueness, it is easy to see that $m(-\bar z) = - \barr{m(z)}$ for all $z \in \C_+$. Thus, \eqref{eq:def_m_beta} implies the same statement for $m_\beta$.
	\end{proof} 
	
	\begin{lemma}\label{lem:Stieltjes_proper}
		If $m: \CC\setminus \R \rightarrow \CC$ is the Stieltjes transform of a probability measure on $\R$ then
		\begin{equation*}
			|m(\ii \eta)| \le \frac{1}{\eta }~ 
			\qquad \text{ and } \qquad ~| m(\ii \eta^\prime ) - m(\ii \eta) | \le \frac{\eta^{\prime} -\eta}{\eta^2}
		\end{equation*}
		for all $0 < \eta \le \eta^{\prime}$. 
	\end{lemma}
	
	\begin{proof} 
		These estimates follow immediately from the Stieltjes transform representation of $m$.
	\end{proof} 
	
	\subsection{Stability of cubic equations} \label{app:cubic_stability} 
	
	\begin{lemma}[Bulk stability]\label{lem:bulkstab}
		For every $\delta >0$, there exist constants 
		$C$, $\mathfrak{c} >0$ depending only on $\delta$ such that for all $q \in \C$ and $r \in \C$ satisfying 
		\begin{equation}\label{eq:q_r_equation} 
			q^3 + 2\ii\eta q^2 + \left(1-\eta^2 - |w|^2 \right) q +\ii \eta =r~\text{and}~ \left|q - \ii v(w,\eta) \right| \le \mathfrak{c}
		\end{equation}
		for some $(w,\eta) \in (\mathrm{D}_{1-\delta} \times (0,\infty))\cup (\C \times [1,\infty))$ we have 
		\begin{equation*}
			\left|q - \ii v(w,\eta) \right| \le \frac{C|r|}{1 + \eta^2}.
		\end{equation*}
	\end{lemma}
	
	\begin{proof}[Proof of \cref{lem:bulkstab}]
		We fix $w$ and $\eta$ as in the lemma and write $v=v(w, \eta)$ and $\Delta :=\ii v - q$ throughout the proof. We expand $q = \ii v - \Delta$ in the left-hand side of \eqref{eq:q_r_equation} and obtain
		\begin{equation}\label{eq:q_cubic_expansion} 
			q^3 + 2\ii\eta q^2 + \left(1-\eta^2 - |w|^2 \right) q +\ii \eta = -\Delta^3 + ( 3 \ii v + 2\ii \eta) \Delta^2 + \bigl( 4 \eta v + 3 v^{2} + \eta^2 + \left|w\right|^2 -1 \bigr) \Delta, 
		\end{equation}
		where the term independent of $\Delta$ vanishes due to \eqref{eq:v}. 
		Using \eqref{eq:v} again, we conclude 
		\begin{equation}\label{eq:gamma_stability_lower_bound} 
			\xi_1(w,\eta):= 4 \eta v + 3 v^{2} + \eta^2 + \left|w\right|^2 -1= 2\eta v + 2v^2 + \frac{\eta}{v} \ge c_{\delta}( 1 + \eta^2)
		\end{equation}
		for some constant $c_{\delta}>0$. Indeed, on the one hand, according to \eqref{eq:v_asymp_1} in \cref{lem:profv} with $L=1$, there exists $c_{\delta,1}>0$ such that for all $w \in \mathrm{D}_{1-\delta}$ and $\eta \in (0,1]$
		\begin{equation}
			\xi_1 \ge 2 v^2 \ge v^2 (1+ \eta^2) \ge c_{\delta,1} (1+ \eta^2).
		\end{equation}
		On the other hand, from \eqref{eq:v_leq_1}, we conclude $\xi_1 \ge {\eta}/{v} \ge \eta^2 \ge \frac{1}{2} (1 + \eta^2)$ for $w \in \CC$ and $\eta >1$. Hence, the inequality in \eqref{eq:gamma_stability_lower_bound} follows by setting
		$c_{\delta}:=c_{\delta,1} \land \frac{1}{2}$.
		Owing to \eqref{eq:v_leq_1} in \cref{lem:profv}, we have $\abs{3\ii v - 2 \ii \eta} \leq 3 + 2\eta$ for all $\eta>0$.  
		Therefore, \eqref{eq:q_r_equation}, \eqref{eq:q_cubic_expansion} and \eqref{eq:gamma_stability_lower_bound} yield 
		\begin{equation}
			|r| \ge \left| \Delta \right| \bigl( \xi_1(w,\eta) - \left|\Delta\right|^2 - |3 \ii v -2 \ii \eta| \left|\Delta\right| \bigr) \ge  \left| \Delta \right| \bigl( c_{\delta}(1 + \eta^2) - \mathfrak{c}^2 - (3 + 2 \eta)\mathfrak{c} \bigr).
		\end{equation}
		We can then choose $\mathfrak{c}>0$ small enough and depending only on $c_{\delta}$ such that for some $c>0$ depending only $\mathfrak{c}$ and $c_{\delta}$ we have $\left( c_{\delta}(1 + \eta^2) - \mathfrak{c}^2 - (3 + 2 \eta)\mathfrak{c} \right) > c(1+ \eta^2)$ for all $\eta >0$. It remains to set $C= 1/c$ to conclude the proof of \cref{lem:bulkstab}.
	\end{proof}
	
	\begin{lemma}[Cubic stability inequality]\label{lem:stability2}
		For $0 < \eta_{\ast} < \eta^{\ast} < \infty$, let $\xi_1, \xi_2 : [\eta_{\ast}, \eta^{\ast}] \rightarrow \CC$ be complex-valued functions and $\tilde{\xi},\varphi: [\eta_{\ast},\eta^{\ast}] \rightarrow \RR_{+}$ be continuous. Suppose that some continuous function $\Delta : [\eta_{\ast}, \eta^{\ast}] \rightarrow \CC$ satisfies the cubic inequality 
		\begin{equation*}
			\big| \Delta^3 + \xi_2 \Delta^2 + \xi_1 \Delta \big| \le \varphi 
		\end{equation*}
		on $[\eta_{\ast}, \eta^{\ast}]$ as well as 
		\begin{equation*}
			\abs{\Delta} \lesssim \min \left\{ \varphi^{1/3}, \frac{\varphi^{1/2}}{\tilde{\xi}^{1/2}}, \frac{\varphi}{\tilde{\xi}}\right\}
		\end{equation*}
		at $\eta^{\ast}$. If the following holds
		\begin{enumerate}
			\item $\tilde{\xi}^3/\varphi^{2}$ is a monotonically increasing function,
			\item $|\xi_1| \asymp \tilde{\xi}$ and $|\xi_2| \lesssim \tilde{\xi}^{1/2}$,
		\end{enumerate}
		then, on $[\eta_{\ast}, \eta^{\ast}]$ we have 
		\begin{equation*}
			\abs{\Delta} \lesssim  \min\left\{ \varphi^{1/3}, \frac{\varphi^{1/2}}{\tilde{\xi}^{1/2}}, \frac{\varphi}{\tilde{\xi}}\right\}.
		\end{equation*}
	\end{lemma}
	
	\begin{proof} 
	This is {\cite[Lemma~10.3]{e643140142d24edcbb5c1a963fc669bf}}. 
	\end{proof} 
	
	\subsection{Concentration bounds for sums of sparse random variables}
	In this appendix, we record a few concentration bounds for sums of sparse random variables, which are used throughout our paper. 
	
	We denote the $L^r$-norm of a random variable $X$ by $\norm{X}_r:= ( \E \abs{X}^r)^{1/r}$ for $r >0$.  
	\begin{prop}[Concentration bounds, $L^p$-norm]\label{prop:lardev}
		Let $r$ be even and $1 \le d \le N$. Let $X_1,...,X_{N}$ and $Y_1,...,Y_N$ be independent random variables satisfying
		\begin{equation*}\label{eq:condiXYconcentration}
			\mathbb{E}Z  =0, ~~~ \mathbb{E}  \abs{Z}^k  \le \frac{1}{N d^{(k-2)/2}},
		\end{equation*}
		for all $Z \in \{X_i,Y_i:~ i \in [N]\}$ and $2 \le k \le r$. Let $(a_{i})_{i \in [N]} \in \CC^N$ and $(b_{ij})_{i,j \in [N]} \in \CC^{N \times N}$ be deterministic sequences. Suppose that
		\begin{equation*}
			\bigg( \frac{1}{N} \sum_{i} \abs{a_i}^2\bigg)^{1/2} \le \gamma,~~ \frac{\max_i \abs{a_i}}{\sqrt{d}} \le \psi
		\end{equation*}
		and 
		\begin{equation*}
			\bigg( \max_{i} \frac{1}{N} \sum_j \left|b_{ij}\right|^2 \bigg)^{1/2} \vee \bigg( \max_{j} \frac{1}{N} \sum_i \left|b_{ij}\right|^2 \bigg)^{1/2} \le \gamma, ~~~ \frac{\max_{i,j} \left|b_{ij}\right|}{d} \le \psi,
		\end{equation*}
		for some $\gamma, \psi \ge 0$. Then
		\begin{align}
			\normbb{\sum_{i}a_i X_i}_{r} &\le \left( \frac{2r}{1+ 2(\log(\psi/\gamma))_{+}} \vee 2 \right) (\gamma \vee \psi),\label{eq:sumaiXi}\\
			\normbb{\sum_{i}a_i \left(\abs{X_i}^2  - \mathbb{E}\abs{X_{i}}^2\right) }_{r} &\le 2\left( 1+  \frac{2d}{N} \right) \max_{i} |a_i|\left( \frac{r}{d} \vee \sqrt{\frac{r}{d}} \right),\label{eq:sumaiXi2}\\
			\normbb{\sum_{i,j} a_{ij}X_i Y_j}_{r} &\le  \left( \frac{2r}{1+ 2(\log(\psi/\gamma))_{+}} \vee 2 \right)^2 (\gamma \vee \psi),\label{eq:sumaijXiYj}\\
			\normbb{\sum_{i\neq j} a_{ij}X_i X_j}_{r} &\le  \left( \frac{4r}{1+ (\log(\psi/\gamma))_{+}} \vee 4 \right)^2 (\gamma \vee \psi).\label{eq:sumaijXiXj}
		\end{align}
	\end{prop}
	\begin{proof}
		See \cite[Proposition 3.1, Proposition 3.2, Proposition 3.3, Proposition 3.5]{MR4021251}.
	\end{proof}
	\begin{corollary}[Concentration bounds, very high probability] \label{lemma:dev1}
		Fix $\delta \in (0,1)$. Under the assumptions of \cref{prop:lardev}, if furthermore we have $\psi/ \gamma \ge N^{\delta/ 4}$ then with very high probability
		\begin{equation*}
			\absbb{ \sum_{i} a_i X_{i}} \le \mathcal{C} \psi,~~~\absbb{ \sum_{i  \neq  j} b_{ij} X_{i}X_{j}} \le \mathcal{C} \psi ~~ \text{and} ~~  \absbb{\sum_{i,j} b_{ij} X_i Y_j} \leq \mathcal C \psi. 
		\end{equation*}
	\end{corollary}
	\begin{proof}
		For the first two estimates, see \cite[Corollary A.21]{alt2021delocalization}. For the last estimate we set $r:=\frac{1}{2} \nu \log N$ from \eqref{eq:sumaijXiYj} we have
		\begin{equation*}
			\mathbb{P} \biggl( \absbb{ \sum_{i,j} b_{ij}X_i Y_j} > \mathcal{C} \psi \biggr) \le \left(\frac{ 72 \nu^2}{\delta^2 \mathcal{C}} \right)^r
		\end{equation*}
		where we took $\nu>0$ large enough and used $\psi/ \gamma \ge N^{\delta /4}$. It remains to set $\mathcal{C} = \frac{72 e^2}{\delta^2} \nu^2$.
	\end{proof}
	
   	\subsection{Resolvent identities} \label{sec:resolvent_identities} 
   	
	In this section, for a square matrix $M$, we consider its resolvent or Green function defined by 
	\begin{equation}
		G = G(z) = G(z,M) =  (M - z)^{-1} 
	\end{equation}
	for $z \in\C \setminus \mathrm{spec}(M)$, where $\mathrm{spec}(M)$ denotes the spectrum of $M$, i.e.\ the set of its eigenvalues. 
	\begin{lemma}\label{lem:disttospect}
		Let $M$ and $M^\prime$ be Hermitian matrices. If $z \in \C\setminus \R$ and $0 < \eta \le \eta^{\prime}$ then 
		\begin{equation*}
			\|G(z, M)\| \le \frac{1}{\abs{\Im z}}, \quad 
			\|G(\ii \eta,  M) - G(\ii \eta^\prime, M)\| \le \frac{\eta^{\prime} - \eta}{\eta^2}, \quad  
			\|G(\ii \eta,  M) - G(\ii \eta, M^\prime)\| \le \frac{\|M - M^{\prime}\|}{\eta^2}.
		\end{equation*}
	\end{lemma}
	\begin{proof} 
		We first note that for normal matrices $M$, the well-known bound 
		\begin{equation*}
			\|G(z)\| \le \frac{1}{\mathrm{dist}(z, \mathrm{spec}(M))} 
		\end{equation*}
		holds. Since $\mathrm{spec}(M) \subset \R$ if $M$ is Hermitian, this immediately implies the first bound. The well-known resolvent identity $G(\ii \eta,  M) - G(\ii \eta, M^\prime) = G(\ii \eta,  M) (M-M^{\prime}) G(\ii \eta, M^\prime)$ and the first bound directly yield the second and third bounds. 
	\end{proof} 
	
	For the next lemma, we recall the notation $M^{(T)}$ for $T \subset [N]$ from \cref{eq:def_M_(_T_)} and note that $G^{(T)}(z):= (M^{(T)} - z)^{-1}$. 
	\begin{lemma}[Ward identity]\label{lemma:wardidentity}
		For a Hermitian $N \times N$-matrix $M=M^*$, $x \notin T \subset [N]$ and $z \in \C \setminus \R$, we have
		\begin{equation}\label{eq:wardidentity}
			\sum_{y}^{(T)} |G_{xy}^{(T)}(z)|^2 = \frac{1}{\Ima z}\Ima G_{xx}^{(T)} (z).
		\end{equation}
	\end{lemma}
	
	\begin{proof}
		See e.g.\ \cite[Equation (3.6)]{benaych2016lectures}.
	\end{proof}
	\begin{lemma}
		For all $N\times N$-matrices $M$, $B \subset [N]$ and $z \in B $ and $y \notin B$, we have
		\begin{equation}\label{expansion1}
			G_{zy} = - \sum_{b \in B} \sum_{a}^{(B)} G_{zb} M_{ba} G_{ay}^{(B)}.
		\end{equation}
		The previous equality holds replacing $G$ by $G^{(C)}$ and $G^{(B)}$ by $G^{(C \cup B)}$ for any $C\subset [N]$ such that $C \cap B = \emptyset$ and $y \notin C \cup B$. In particular, if $M \in \C^{2N \times 2N}$, $T \subset [N]$, $x \in [N]\backslash T$ and $y\in [2N]\setminus (T \cup \ubarr{T} \cup \{ x, \ubarr{x}\})$,  we have
		\begin{align}
		G^{\bset{T}}_{xy} &= - \sum_{b \in [2N]\setminus (T \cup \ubarr{T} \cup \{ x, \ubarr{x}\})}  \Big( G_{xx}^{\bset{T}} M_{xb} G^{\bset{Tx}}_{by} + G^{\bset{T}}_{x \ubarr{x}}  M_{\ubarr{x} b} G^{\bset{Tx}}_{by} \Big), \label{eq:A29A}\\ 
			G^{\bset{T}}_{\ubarr{x}y} &= - \sum_{b \in [2N]\setminus (T \cup \ubarr{T} \cup \{ x, \ubarr{x}\})}  \Big( G_{\ubarr{x}\ubarr{x}}^{\bset{T}} M_{\ubarr{x}b} G^{\bset{Tx}}_{by} + G^{\bset{T}}_{\ubarr{x}x }  M_{x b} G^{\bset{Tx}}_{by} \Big). \label{eq:A29B}
		\end{align}
	\end{lemma}
	\begin{proof}
		To generalize the first statement of the previous lemma to $G^{(C)}$ and $G^{(C\cup B)}$ one only has to notice that for any $B,C \subset [N]$ such that $B \cap C = \emptyset$ we have:
		$$\{G^{(C)}\}^{(B)}= G^{(C\cup B)}.$$
		To obtain \eqref{expansion1}, we consider the $(x,y)$-entry of $G$ with $x \in B$ and $y \notin B$ in the following decomposition
		\begin{equation}\label{eq:resolv_identity}
			G= G^{(B)} - G \left(M - M^{(B)}\right) G^{(B)}.
		\end{equation}
		To obtain the second statement we consider  $C = T \cup \ubarr{T}$ for some $T \subset [N] $, $x \notin T$ and $B = \{x, \ubarr{x}\}$.
	\end{proof}
	\begin{lemma}
		For all $N \times N$-matrices $M$, $B \subset [N]$ and $i,j \notin B$, we have
		\begin{equation}\label{expansion2}
			G^{(B)}_{ij} = G_{ij} + \sum_{b \in B }\sum_{a}^{(B)}G^{(B)}_{ia}M_{ab}G_{bj}.
		\end{equation}
		The previous equality holds replacing $G$ by $G^{(C)}$ and $G^{(B)}$ by $G^{(C \cup B)}$ for any $C\subset [N]$ such that $C \cap B = \emptyset$ and $ i,j \notin C \cup B$. In particular, if $M \in \C^{2N\times 2N}$, $T \subset [N]$ and $i,j \in [2N]\setminus (T \cup \ubarr{T} \cup \{ x, \ubarr{x}\})$, 
		we have
		\begin{equation}\label{eq:almost32a}
			G^{\bset{Tx}}_{ij} = G_{ij}^{\bset{T}} + \sum_{ a \in [2N]\setminus (T \cup \ubarr{T} \cup \{ x, \ubarr{x}\})\nc} \big(  G_{ia}^{\bset{Tx}} M_{ax} G_{xj}^{\bset{T}} +  G_{ia}^{\bset{Tx}} M_{a\ubarr{x}} G_{\ubarr{x}j}^{\bset{T}}\big) .
		\end{equation}
		
	\end{lemma}
	\begin{proof}
		To obtain \eqref{expansion2}, we consider the coordinate $(i,j)$ for $i,j \notin B$ in the following writing of $G^{(B)}$:
		\begin{equation*}
			G^{(B)}= G - G^{(B)} \left(M - M^{(B)}\right) G.
		\end{equation*}
		To obtain the second statement we consider again  $C = T \cup \ubarr{T}$ for some $T \subset [N] $, $x \notin T$ and $B = \{x, \ubarr{x}\}$.
	\end{proof}
	We now apply the previous lemmas to the Green function $G= (H-z)^{-1} \in \Mat_{2N}(\CC)$ defined by \cref{Greenfunction} and where $H= H(w)$ is defined by \cref{hermitizationofM} from the matrix $M = X + f(\mathbf{e}\mathbf{e}^\ast - 1/N)$ given in \cref{defiX} and \cref{eq:def_M_X_plus_f}. We consider $T\subset[N]$, $x \in [N]\backslash T$ and $i,j \in [2N]\backslash ( T \cup \ubarr{T} \cup \{x,\ubarr{x}\} )$, apply \cref{eq:almost32a} and obtain
	\begin{equation}\label{eq:32a}
		G_{ij}^{\bset{Tx}} = G_{ij}^{\bset{T}} + \sum_{a}^{(Tx)}\left\{ G_{i\ubarr{a}}^{\bset{Tx}} \barr{X}_{ax}G_{xj}^{\bset{T}}+ G_{ia}^{\bset{Tx}} X_{ax} G_{\ubarr{x}j}^{\bset{T}} + \frac{f}{N} \left(G_{i \ubarr{a}}^{\bset{Tx}} G^{\bset{T}}_{xj} + G_{ia}^{\bset{Tx}} G_{\ubarr{x}j}^{\bset{T}} \right)\right\}.
	\end{equation}
	We introduce the short-hands 
	\begin{align}
		\mathrm{R}_0(T,i,j) & = \frac{f}{N} \sum_{a}^{(Tx)}  \biggl( G_{i \ubarr{a}}^{\bset{Tx}} G^{\bset{T}}_{xj} + G_{ia}^{\bset{Tx}} G_{\ubarr{x}j}^{\bset{T}} \biggr) \label{eq:defireste}\\
		\mathrm{R}_1(T,i,j)& = \frac{f}{N} \sum_{a,b}^{(Tx)} \biggl(G_{i\ubarr{a}}^{\bset{Tx}} \barr{X}_{ax} \left\{ G_{xx}^{\bset{T}}G_{\ubarr{b}j}^{\bset{Tx}} + G_{x\ubarr{x}}^{\bset{T}} G_{bj}^{\bset{Tx}} \right\} + G_{ia}^{\bset{Tx}} X_{ax} \left\{ G_{\ubarr{x}x}^{\bset{T}} G_{\ubarr{b}j}^{\bset{Tx}} + G_{\ubarr{x}\ubarr{x}}^{\bset{T}}  G_{bj}^{\bset{Tx}} \right\}\biggr) 
		\label{eq:def_R_1}\\ 
		\mathrm{R}_2(T,i,j) & = \sum_{a,b}^{(Tx)} \bigg\{G_{i\ubarr{a}}^{\bset{Tx}} \barr{X}_{ax} \left\{ G_{xx}^{\bset{T}} X_{xb} G_{\ubarr{b}j}^{\bset{Tx}} + G_{x\ubarr{x}}^{\bset{T}} \barr{X}_{xb}G_{bj}^{\bset{Tx}} \right\} 
		\notag \\ 
		& \qquad \qquad \qquad \qquad \qquad \qquad \qquad  +  G_{ia}^{\bset{Tx}} X_{ax} \left\{ G_{\ubarr{x}x}^{\bset{T}} X_{xb} G_{\ubarr{b}j}^{\bset{Tx}} + G_{\ubarr{x}\ubarr{x}}^{\bset{T}} \barr{X}_{xb} G_{bj}^{\bset{Tx}} \right\} \bigg\}\label{eq:def_R_2} 
	\end{align}
	and apply \cref{eq:A29A} and \cref{eq:A29B} to the terms $G_{xj}^{\bset{T}}$ and $G_{\ubarr{x}j}^{\bset{T}}$ in \cref{eq:32a}, respectively, to obtain 
	\begin{align}
		G_{ij}^{\bset{Tx}} = G_{ij}^{\bset{T}} - \mathrm{R}_2(T,i,j)
		+ \mathrm{R}_0(T,i,j)+ \mathrm{R}_1(T,i,j). \label{eq:A32b}
	\end{align}

	\phantomsection
	
	\addcontentsline{toc}{section}{References}
	
	\bibliographystyle{amsalpha-nodash}
	\bibliography{biblio.bib}
\end{document}